\documentclass{tran-l}

\usepackage{amsmath,amssymb,amscd,amsxtra,verbatim}
\usepackage[colorlinks=true]{hyperref}

\newtheorem{Thm}{Theorem}[section]
\newtheorem{Prop}[Thm]{Proposition}
\newtheorem{Lem}[Thm]{Lemma}
\newtheorem{Cor}[Thm]{Corollary}
\newtheorem*{MainThm}{Theorem}

\DeclareMathOperator{\Hom}{Hom}

\DeclareMathOperator{\End}{End}
\DeclareMathOperator{\Aut}{Aut}
\DeclareMathOperator{\Ext}{Ext}
\DeclareMathOperator{\ext}{ext}
\DeclareMathOperator{\Spec}{Spec}
\DeclareMathOperator{\Proj}{Proj}
\DeclareMathOperator{\rep}{rep}
\DeclareMathOperator{\inj}{inj}
\DeclareMathOperator{\Fl}{Fl}
\DeclareMathOperator{\Gr}{Gr}
\DeclareMathOperator{\modcat}{mod}
\DeclareMathOperator{\GL}{GL}
\DeclareMathOperator{\Stab}{Stab}
\DeclareMathOperator{\Orb}{Orb}
\DeclareMathOperator{\Transp}{Transp}
\DeclareMathOperator{\Der}{Der}
\DeclareMathOperator{\Ker}{Ker}
\DeclareMathOperator{\Coker}{Coker}
\DeclareMathOperator{\Ima}{Im}

\newcommand{\red}{\mathrm{red}}
\newcommand{\bM}{\mathbb M}
\newcommand{\bC}{\mathbb C}
\newcommand{\bR}{\mathbb R}
\newcommand{\bZ}{\mathbb Z}
\newcommand{\bP}{\mathbb P}
\newcommand{\bA}{\mathbb A}

\newcommand{\ptx}{\mathsf x}
\newcommand{\pty}{\mathsf y}
\newcommand{\ptz}{\mathsf z}

\newcommand{\updot}{^{\hbox{\raise1pt\hbox{\large\bfseries .}}}}
\newcommand{\nfold}[1]{_{\hbox{\raise-1pt\hbox{$\scriptscriptstyle #1$}}}}

\numberwithin{equation}{section}

\makeatletter
\renewcommand\subsubsection{\@startsection{subsubsection}{3}%
  \z@{.5\baselineskip\@plus.7\baselineskip}{-.5em}%
  {\normalfont\bfseries}}
\makeatother

\begin{document}

\title{Irreducible components of quiver Grassmannians}

\author{Andrew Hubery}
\address{Bielefeld University, Germany}
\email{hubery@math.uni-bielefeld.de}


\begin{abstract}
We consider the action of a smooth, connected group scheme $G$ on a scheme $Y$, and discuss the problem of when the saturation map $\Theta\colon G\times X\to Y$ is separable, where $X\subset Y$ is an irreducible subscheme. We provide sufficient conditions for this in terms of the induced map on the fibres of the conormal bundles to the orbits. Using jet space calculations, one then obtains a criterion for when the scheme-theoretic image of $\Theta$ is an irreducible component of $Y$.

We apply this result to Grassmannians of submodules and several other schemes arising from representations of algebras, thus obtaining a decomposition theorem for their irreducible components in the spirit of the result by Crawley-Boevey and Schr\"oer for module varieties.
\end{abstract}

\maketitle

\setcounter{tocdepth}{1}
\tableofcontents

\section{Introduction}

In \cite{CBS} Crawley-Boevey and Schr\"oer provided an important generalisation of the work of Kac and Schofield on general properties of quiver representations, extending this to finite-dimensional modules over any finitely-generated algebra (associative and unital, but not necessarily commutative). More precisely, let $K$ be an algebraically-closed field and $\Lambda$ a finitely-generated $K$-algebra. Then there is a scheme $\rep_\Lambda^d$ together with an action of $\GL_d$ such that, for each field extension $L/K$, the $\GL_d(L)$-orbits on $\rep_\Lambda^d(L)$ are in bijection with the isomorphism classes of $d$-dimensional modules over $\Lambda\otimes_KL$. Taking the direct sum of modules yields a `direct sum' morphism
\[ \Theta \colon \GL_{d+e}\times\rep_\Lambda^d\times\rep_\Lambda^e \to \rep_\Lambda^{d+e}, \]
so given irreducible components $X\subset\rep_\Lambda^d$ and $Y\subset\rep_\Lambda^e$, we can construct their `direct sum' to be the irreducible subset
\[ \overline{X\oplus Y} := \overline{\Theta(\GL_{d+e}\times X\times Y)} \subset \rep_\Lambda^{d+e}. \]
The main result of Crawley-Boevey and Schr\"oer was a decomposition theorem for the irreducible components of each $\rep_\Lambda^d$, proving that
\begin{enumerate}
\item every irreducible component can be written as a direct sum of irreducible components whose general representation is indecomposable, and this decomposition is essentially unique.
\item the direct sum $\overline{X\oplus Y}$ of two irreducible components is again an irreducible component if and only if the general representations for $X$ and $Y$ have no extensions with each other.
\end{enumerate}

Of course, there are many other interesting schemes arising from the representation theory of algebras. For example one can consider quiver Grassmannians, or more generally flags of submodules of a fixed module, which in turn have applications to cluster algebras \cite{CC} (the Laurent polynomial describing a cluster variable from a given acyclic seed can be written using Euler characteristics of quiver Grassmannians) as well as quantum groups and Ringel-Hall algebras \cite{Ringel} (when expanding a product of the generators in terms of the basis of root vectors, the coefficients can be written in terms of the number of rational points of quiver flag varieties over finite fields). One can also consider those flags with specified subquotients, and thus the schemes of $\Delta$-filtered modules over a quasi-hereditary algebra. Another example might be the zero sets of homogeneous semi-invariants for quivers \cite{RZ}.

Importantly, in all these examples, we again have an analogue of the direct sum morphism. It is therefore a natural question to ask whether the results of Crawley-Boevey and Schr\"oer can be extended to these other types of schemes.

In this article we answer this question for four such types of schemes, including the schemes $\rep_\Lambda^d$ studied in \cite{CBS}, Grassmannians of submodules in \hyperref[Sec:Grassmannians]{Section \ref*{Sec:Grassmannians}} and more general flags of submodules in \hyperref[Sec:flags]{Section \ref*{Sec:flags}}. Our methods are slightly different than those of Crawley-Boevey and Schr\"oer, though, in that they are based on the fact that the direct sum morphism $\Theta\colon\GL_{d+e}\times X\times Y\to\rep_\Lambda^{d+e}$ is always a separable map.

For Grassmannians of submodules, the theorem reads as follows.

\begin{MainThm}
Let $K$ be algebraically closed, $\Lambda$ a finitely-generated $K$-algebra, and write $\Lambda(2)\subset\bM_2(\Lambda)$ for the subalgebra of upper-triangular matrices. Given a $\Lambda$-module $M$ we define the projective scheme
\[ \Gr_\Lambda\binom Md(R) := \Big\{ U\in\Gr_K\binom Md(R) : U\leq M\otimes_KR\textrm{ is an $\Lambda\otimes_KR$-submodule} \Big\}. \]
Then
\begin{enumerate}
\item the direct sum of representations yields a separable morphism
\[ \Theta \colon \Aut_\Lambda(M\oplus N)\times\Gr_\Lambda\binom Md\times\Gr_\Lambda\binom Ne \to \Gr_\Lambda\binom{M\oplus N}{d+e}. \]
\item every irreducible component of $\Gr_\Lambda\binom Md$ can be written uniquely (up to reordering) as a direct sum $X=\overline{X_1\oplus\cdots\oplus X_n}$, where $M\cong M_1\oplus\cdots\oplus M_n$ and $d=d_1+\cdots+d_n$, and each $X_i\subset\Gr_\Lambda\binom{M_i}{d_i}$ is an irreducible component such that for all $U_i$ in an open dense subset of $X_i$, the corresponding $\Lambda(2)$-module $(U_i\subset M_i)$ is indecomposable.
\item the direct sum $\overline{X\oplus Y}\subset\Gr_\Lambda\binom{M\oplus N}{d+e}$ of two irreducible components $X\subset\Gr_\Lambda\binom Md$ and $Y\subset\Gr_\Lambda\binom Ne$ is again an irreducible component if and only if, for generic $(U,V)\in X\times Y$, every pull-back along some $U\to N/V$ yields a split extension (in the category of $\Lambda(2)$-modules) of $(U\subset M)$ by $(V\subset N)$, and similarly for every pull-back along some $V\to M/U$.
\end{enumerate}
\end{MainThm}

In fact, in the first part of the paper we discuss the following general situation. Suppose we have smooth, connected group schemes $H\leq G$, an action of $G$ on a scheme $Y$, and an $H$-stable irreducible subscheme $X\subset Y$. Associated to this we have the morphism $\Theta\colon G\times X\to Y$, and we are interested in understanding when the image of $\Theta$ is dense in an irreducible component of $Y$.

There is an easy sufficient criterion for this using deformation theory (compare \cite[Lemma 4.1]{CBS}): suppose we have an open, irreducible subset $U\subset G\times X$ such that, for every field $L$ and every commutative diagram
\[ \begin{CD}
\Spec L @>>> \Spec L[[t]]\\
@VVV @VVV\\
U @>>> Y
\end{CD} \]
there exists a morphism $\Spec L[[t]]\to U$ making the two triangles commute. Then the image of $\Theta\colon G\times X\to Y$ is dense in an irreducible component of $Y$.

We show in \hyperref[Thm:sep]{Theorem \ref*{Thm:sep}} that the converse holds whenever $\Theta$ is separable. Moreover, we provide in \hyperref[Thm:group-action]{Theorem \ref*{Thm:group-action}} some sufficient conditions under which such a morphism will be separable. This result covers all the cases we investigate in this paper, and should be applicable quite generally.

\subsection{Overview}

We now describe the various sections of the paper in more detail.

We begin by collating some results about schemes, including tangent spaces, jet spaces, and separable morphisms. In particular, we include Greenberg's Theorem, \hyperref[Thm:Greenberg]{Theorem \ref*{Thm:Greenberg}}, which together with the Theorem of Generic Smoothness allows us in \hyperref[Cor:reduced]{Corollary \ref*{Cor:reduced}} to compute the tangent spaces $T_xX_\red$ in terms of the jet spaces of $X$, at least generically. We also prove an interesting characterisation of when a field extension is separable in terms of iterated tensor products \hyperref[Thm:sep-field]{Theorem \ref*{Thm:sep-field}}; geometrically this corresponds to saying that all iterated fibre products are generically reduced \hyperref[Thm:separable]{Theorem \ref*{Thm:separable}}. Along the way we also give in \hyperref[Prop:differentials-for-regular-local]{Proposition \ref*{Prop:differentials-for-regular-local}} an elementary proof that the module of K\"ahler differentials for a regular local ring over a perfect field is finite free.

In \hyperref[Sec:irred-cpts]{Section \ref*{Sec:irred-cpts}} we study a morphism of schemes $f\colon X\to Y$ and provide in \hyperref[Thm:sep]{Theorem \ref*{Thm:sep}} a sufficient criterion involving jet spaces for the image of an irreducible component of $X$ to be dense in an irreducible component of $Y$; moreover, this condition is necessary whenever $f$ is separable. We also show in \hyperref[Prop:irred-comp]{Proposition \ref*{Prop:irred-comp}} that every irreducible component of a Noetherian scheme has a natural subscheme structure such that all the jet spaces agree generically on the irreducible component.

We finish the first part by discussing group scheme actions in \hyperref[Sec:group-schemes]{Section \ref*{Sec:group-schemes}}. Suppose we have a group scheme $G$ acting on a scheme $X$, and let $\pi\colon X\to Y$ be a morphism of schemes which is constant on $G$-orbits. We begin by showing in \hyperref[Prop:geom-quotient]{Proposition \ref*{Prop:geom-quotient}} that $\pi$ is a geometric quotient if and only if it is a quotient in the category of all locally-ringed spaces, thus generalising \cite[Proposition 0.1]{GIT}. We also show that if $\pi$ is a quotient in the category of faisceaux, then it is automatically a universal geometric quotient. We then recall several results about when such quotient faisceaux exist and what properties they inherit, such as smoothness, affineness, faithful flatness and reducedness. We also provide a result which allows us to identify associated fibrations, \hyperref[Lem:assoc-fib]{Lemma \ref*{Lem:assoc-fib}}, and prove that all principal $G$-bundles are necessarily quotient faisceaux, \hyperref[Lem:principal-G-bundle]{Lemma \ref*{Lem:principal-G-bundle}}.

We then discuss the particular situation we are interested in, where we have smooth, connected groups $H\leq G$, an action of $G$ on a scheme $Y$, and an $H$-stable subscheme $X\subset Y$. In this situation one may consider what are essentially the fibres of the conormal bundles to the orbits: given an $L$-valued point $x\in X(L)$ for some field $L$, set $N_{X,x}:=T_xX/T_x\Orb_H(x)$ and similarly $N_{Y,x}:=T_xY/T_x\Orb_G(x)$. The morphism $\Theta\colon G\times X\to Y$ then induces a linear map $\theta_x\colon N_{X,x}\to N_{Y,x}$. We describe in \hyperref[Thm:group-action]{Theorem \ref*{Thm:group-action}} sufficient conditions involving $\theta_x$ for $\Theta$ to be separable.

In the second part of the paper we apply these ideas to four types of schemes arising from representation theory. We begin in \hyperref[Sec:rep-schemes]{Section \ref*{Sec:rep-schemes}} by considering the schemes $\rep_\Lambda^d$, as studied by Crawley-Boevey and Schr\"oer. We also discuss the generalisation to the case where we fix a complete set of orthogonal idempotents in the algebra $\Lambda$. This is a slight extension of results in \cite{Bongartz,Gabriel}, but note that there it is assumed $\Lambda$ is finite dimensional, so splits as a semisimple subalgebra together with the Jacobson radical. This does not hold in general, so we need a different proof. Our approach is in fact based on a formula for the determinant of a sum of square matrices, \hyperref[Lem:det-sum]{Lemma \ref*{Lem:det-sum}}, which may also be of independent interest.

In \hyperref[Sec:rep-schemes-homs]{Section \ref*{Sec:rep-schemes-homs}} we consider subschemes of the $\rep_\Lambda^d$ parameterising those representations $X$ such that $\dim\Hom(X,M)=u$, for some given module $M$ and non-negative integer $u$. Extending this to more than one module $M$, and taking the dual result for $\dim\Hom(M,X)$, one obtains (stratifications of) many types of subschemes. For example, if $\Lambda$ has a preprojective component, then one may consider those modules whose preprojective summand is fixed; fixing syzygies of the simples, one may consider those modules having fixed projective dimension; for the path algebra of a quiver without oriented cycles, one may also construct the zero sets of semi-invariants in this way.

We study Grassmannians of submodules in \hyperref[Sec:Grassmannians]{Section \ref*{Sec:Grassmannians}}, via their construction as a quotient for an action of a general linear group. We also study the smooth and irreducible subschemes given by fixing the isomorphism type of the submodule, and provide a criterion for such a subscheme to be dense in an irreducible component of the Grassmannians, \hyperref[Lem:S-rho-irred-cpt]{Lemma \ref*{Lem:S-rho-irred-cpt}}. This can be seen as an analogue of the easy consequence of Voigt's Lemma, that if a module has no self-extensions, then the closure of its orbit is an irreducible component of the representation scheme, \hyperref[Lem:orbit-irred-cpt]{Lemma \ref*{Lem:orbit-irred-cpt}}. We also compute several examples of quiver Grassmannians, showing that they can be generically non-reduced, and providing the decomposition of their irreducible components.

We finish with a brief discussion of flags of submodules. Our main observation here is that we can regard flags of $\Lambda$-modules of length $m$ as a special case of a Grassmannian for $\Lambda(m)$-modules, where $\Lambda(m)$ is the subalgebra of upper-triangular matrices in $\bM_m(\Lambda)$. Thus all the results for flags follow immediately from the results for Grassmannians.

\subsection*{Acknowledgements}

This work was written whilst visiting the University of Bielefeld under the support of the SFB 701. I would especially like to thank  Henning Krause for inviting me to Bielefeld, Dieter Vossieck for bringing the article of Greenberg \cite{Greenberg} to my attention, and Julia Sauter, Greg Stevenson and Markus Perling for interesting discussions. I would also like to thank the referee for their careful reading of the paper and for pointing out the connection to Zwara's paper \cite{Zwara}.

\section{Schemes}

We will mostly use the functorial approach to schemes, regarding them as covariant functors from the category of $K$-algebras to sets, for some base field $K$. Our basic references will be the books by Demazure and Gabriel \cite{DG} and Matsumura \cite{Matsumura1}. Given a scheme $X$ and a $K$-algebra $R$ we denote by $X(R)$ the set of $R$-valued points of $X$, which by Yoneda's Lemma is identified with the natural transformations $\Spec R\to X$. We call $X(K)$ the set of rational points of $X$. If $L/K$ is a field extension, then we can change base to obtain an $L$-scheme $X^L=X\times_{\Spec K}\Spec L$; note that if $R$ is an $L$-algebra, then $X^L(R)=X(R)$.

Let $f\colon X\to Y$ be a morphism of schemes (so a natural transformation of functors). If $y\in Y(L)$, then the fibre $f^{-1}(y)$ is the fibre product of $f$ and $y\colon\Spec L\to Y$. This is naturally an $L$-scheme, having $R$-valued points (for an $L$-algebra $R$) those $x\in X(R)$ such that $f(x)=y^R\colon\Spec R\to\Spec L\to Y$, and is in fact a closed subscheme of $X^L$.

\subsection{The relative tangent space}

Let $K$ be a field and $X$ a $K$-scheme. For a field $L$ and point $x\in X(L)$ we define the (relative) tangent space at $x$ via the pull-back
\[ \begin{CD}
T_x X @>>> \{x\}\\
@VVV @VVV\\
X(D_1) @>>> X(L)
\end{CD} \]
where $D_1:=L[t]/(t^2)$ is the $L$-algebra of dual numbers.

To relate this to the description of $X$ as a locally-ringed space, recall first that if $(R,\mathfrak m)$ is a local algebra, then elements of $X(R)$ can be described as pairs consisting of a point $\ptx\in X$ in the underlying topological space together with a local homomorphism of $K$-algebras $\mathcal O_{X,\ptx}\to R$. The point $\ptx$ is thus the image of the map $\Spec R/\mathfrak m\to\Spec R\to X$. The local homomorphisms $\mathcal O_{X,\ptx}\to D_1$ are precisely the maps of the form $x+\xi t$, where $x\colon\mathcal O_{X,\ptx}\to L$ is a local homomorphism and $\xi\colon\mathcal O_{X,\ptx}\to L$ is a $K$-derivation (with respect to $x$). This gives the alternative description
\begin{align*}
T_x X &= \Der_K(\mathcal O_{X,\ptx},L)\\
&= \{\textrm{$K$-linear } \xi\colon\mathcal O_{X,\ptx}\to L : \xi(ab)=\xi(a)x(b)+x(a)\xi(b)\}.
\end{align*}
In particular, $T_x X$ is naturally an $L$-vector space. If $f\colon X\to Y$ is a morphism of schemes, then the map on $D_1$-valued points restricts to a differential
\[ d_x f \colon T_x X \to T_{f(x)}Y. \]
In terms of derivations, this is just given by composition with the local homomorphism $f^\#_\ptx\colon\mathcal O_{Y,f(\ptx)}\to\mathcal O_{X,\ptx}$.

Following \cite[I \S4, 2.10]{DG} write $\Omega_{X/K,\ptx}:=\Omega_{\mathcal O_{X,\ptx}/K}$ for the module of K\"ahler differentials and set
\[ \Omega_{X/K}(\ptx):=\Omega_{X/K,\ptx}\otimes_{\mathcal O_{X,\ptx}}\kappa(\ptx). \]
Then we can also express the tangent space as
\[ T_xX = \Hom_{\kappa(\ptx)}(\Omega_{X/K}(\ptx),L). \]
It follows that if $X$ is locally of finite type\footnote{
A morphism $f\colon X\to Y$ is locally of finite type if the preimage of every open affine $\Spec A$ in $Y$ can be covered by open affines $\Spec B$ in $X$ such that each corresponding algebra homomorphism $A\to B$ is of finite type. We say $f$ is of finite type if we can cover the preimage of $\Spec A$ by finitely many such open affines $\Spec B$.
}
over $K$, $x\in X(L)$, $M/L$ is a field extension, and $x':=x^M\in X(M)$, then $T_{x'}X\cong M\otimes_L T_xX$. For, $\Omega_{X/K}(\ptx)$ is a finite-dimensional $\kappa(\ptx)$-vector space, and hence the natural embedding
\[  M\otimes_L\Hom_{\kappa(\ptx)}(\Omega_{X/K}(\ptx),L) \hookrightarrow \Hom_{\kappa(\ptx)}(\Omega_{X/K}(\ptx),M) \]
is an isomorphism. Thus by \cite[I \S4 Proposition 2.10]{DG} and \cite[I \S3 Theorem 6.1]{DG} we have for all $x\in X(L)$ that
\[ \dim_LT_x X = \dim_{\kappa(\ptx)}\Omega_{X/K}(\ptx) \geq \dim_\ptx X := \dim\mathcal O_{X,\ptx}+\mathrm{tr.}\deg_K\kappa(\ptx) \]
with equality if and only if $\ptx$ is non-singular, in which case we also say that $x$ is non-singular. We recall also that $\dim_\ptx X=\max_i\dim X_i$, where the $X_i$ are the irreducible components of $X$ containing $\ptx$.

Finally, let us compare the relative tangent space to the Zariski tangent space
\[ T^{\mathrm{Zar}}_\ptx X := \Hom_{\kappa(\ptx)}(\mathfrak m_\ptx/\mathfrak m_\ptx^2,\kappa(\ptx)). \]
Given $\ptx\in X$, let $x\colon\mathcal O_{X,\ptx}\to\kappa(\ptx)$ be the canonical homomorphism. Each derivation $\xi\in\Der_K(\mathcal O_{X,\ptx},\kappa(\ptx))$ induces a $\kappa(\ptx)$-linear map $\mathfrak m_\ptx/\mathfrak m_\ptx^2\to\kappa(\ptx)$, so yields an exact sequence
\[ 0 \to \Der_K(\kappa(\ptx),\kappa(\ptx)) \to T_x X \to T^{\mathrm{Zar}}_\ptx X. \]
When $\kappa(\ptx)/K$ is separable, the surjective algebra homomorphism $\mathcal O_{X,\ptx}/\mathfrak m_\ptx^2\to\kappa(\ptx)$ splits (in other words, $\mathcal O_{X,\ptx}/\mathfrak m_\ptx^2$ has a coefficient field containing $K$ \cite[(28.I) Proposition]{Matsumura1}), and hence $T_x X\to T^{\mathrm{Zar}}_\ptx X$ is surjective. In this case also $\Der_K(\kappa(\ptx),\kappa(\ptx))\cong\kappa(\ptx)^d$, where $d=\mathrm{tr.deg}_K\kappa(\ptx)$, by \cite[(27.B) Theorem 59]{Matsumura1}, giving
\[ 0 \to \kappa(\ptx)^d \to T_x X \to T^{\mathrm{Zar}}_\ptx X \to 0. \]
Thus if $\kappa(\ptx)/K$ is separable, then $\ptx$ is non-singular if and only if $\dim (\mathfrak m_\ptx/\mathfrak m_\ptx^2)=\dim\mathcal O_{X,\ptx}$, which is if and only if $\mathcal O_{X,\ptx}$ is a regular local ring.

As a special case we see that $T_xX=T_\ptx^{\mathrm{Zar}}X$ whenever $\kappa(\ptx)/K$ is a separable algebraic extension (for example if $x$ is a rational point). Thus, given any $x\in X(L)$, we can regard $x$ as a rational point of $X^L$, say corresponding to $\ptx\in X^L$, in which case $T_xX=T^{\mathrm{Zar}}_\ptx X^L$.

Note that if $f\colon X\to Y$, $\ptx\in X$ and $\pty=f(\ptx)\in Y$, then $\kappa(\ptx)$ is a field extension of $\kappa(\pty)$. We obtain a $\kappa(\ptx)$-linear map
\[ (\mathfrak m_\pty/\mathfrak m_\pty^2)\otimes_{\kappa(\pty)}\kappa(\ptx) \to \mathfrak m_\ptx/\mathfrak m_\ptx^2, \]
and hence an induced map
\[ \Hom_{\kappa(\ptx)}(\mathfrak m_\ptx/\mathfrak m_\ptx^2,\kappa(\ptx)) \to \Hom_{\kappa(\pty)}(\mathfrak m_\pty/\mathfrak m_\pty^2,\kappa(\ptx)). \]
Thus we only get a differential on Zariski tangent spaces under the assumption that $\kappa(\ptx)=\kappa(\pty)$. In other words, the relative tangent space is functorial, unlike the Zariski tangent space.

For convenience we recall the following result (see for example \cite[I \S4, Corollary 4.13]{DG}).

\begin{Prop}\label{Prop:differentials-for-regular-local}
Let $(A,\mathfrak m,L)$ be a regular local $K$-algebra, essentially of finite type.\footnote{
A $K$-algebra $A$ is essentially of finite type if it is a localisation of a finitely-generated algebra.
}
If $L/K$ is separable, then $\Omega_{A/K}$ is a free $A$-module of rank $\dim A+\mathrm{tr.}\deg_KL$.

In particular, let $X$ be a scheme, locally of finite type over $K$, and $\ptx\in X$ a non-singular point. If $\kappa(\ptx)/K$ is separable, then $\Omega_{X/K,\ptx}$ is a free $\mathcal O_{X,\ptx}$-module of rank $\dim_\ptx X$.
\end{Prop}

\begin{proof}
Set $d:=\dim(A)$ and $n:=\mathrm{tr.}\deg_KL$. Since $A$ is essentially of finite type, we know that $\Omega_{A/K}$ is a finitely-generated $A$-module. Choose $t_1,\ldots,t_n\in A$ such that their images form a separating transcendence basis for $L/K$. Then $F:=K(t_1,\ldots,t_n)\subset A$ is a subfield and $L/F$ is separable algebraic and finitely-generated, so finite dimensional. Let $\alpha\in L$ be a primitive element, with minimal polynomial $f$ over $F$. Let $x_1,\ldots,x_d\in\mathfrak m$ be a regular system of parameters and set $B:=F[X_1,\ldots,X_d]_{(X_1,\ldots,X_d)}$. Then $B$ is a regular local $K$-algebra of dimension $d$, we have $\Omega_{B/K}\cong B^{d+n}$ (see for example \cite[(26.J) Example]{Matsumura1}), and we have a natural injective algebra homomorphism $B\to A$ sending $X_i\mapsto x_i$.

The first fundamental exact sequence for K\"ahler differentials \cite[(26.H) Theorem 57]{Matsumura1}
\[ \Omega_{B/K}\otimes_BA \xrightarrow{\delta} \Omega_{A/K} \to \Omega_{A/B} \to 0 \]
 yields for each $A$-module $M$ an exact sequence
\[ 0 \to \Der_B(A,M) \to \Der_K(A,M) \xrightarrow{\delta^\ast} \Der_K(B,M). \]
We will show that $\delta^\ast$ is an isomorphism whenever $M$ has finite length (so finitely generated and $\mathfrak m^rM=0$ for some $r$).

Consider a $K$-derivation $d\colon B\to M$, and take $r$ such that $\mathfrak m^{r-1}M=0$. It follows that $d(\mathfrak m^r)=0$. Since $L/F$ is finite separable, we can lift $\alpha$ to an element $y\in A$ such that $f(y)\in\mathfrak m^r$ (see for example \cite[(28.I) Proposition]{Matsumura1}), from which it follows that $A/\mathfrak m^r\cong L[Z_1,\ldots,Z_d]/(Z_1,\ldots,Z_d)^r$. Now $0=d(f(y))=(df)(y)+f'(y)dy$, where $df$ is the polynomial with coefficients in $M$ obtained by applying $d$ to each coefficient of $f$. Since $f'(y)\in A$ is invertible, we see that $dy$ exists and is unique. It follows that there is a unique lift of $d$ to a $K$-derivation $d\colon A\to M$. Hence $\Der_K(A,M)\cong\Der_K(B,M)$ for all $A$-modules $M$ of finite length.

Now $\Omega_{A/B}$ is finitely-generated and $\Hom_A(\Omega_{A/B},M)=0$ for all finite length modules $M$, so $\Omega_{A/B}=0$. Also, if the kernel $U$ of $\delta$ is non-zero, then by Krull's Intersection Theorem there exists $r$ such that $U\not\subset\mathfrak m^{r-1}A^{d+n}$. It follows that $M=A^{d+n}/\mathfrak m^{r-1}A^{d+n}$ has finite length and the composition $U\to A^{d+n}\to M$ is non-zero, contradicting the fact that $\Der_K(A,M)\to\Der_K(B,M)$ is onto. Thus $U=0$, so $\Omega_{A/K}\cong\Omega_{B/K}\otimes_BA\cong A^{d+n}$.
\end{proof}

\subsubsection{The tangent bundle}

\begin{Lem}\label{Lem:gluing}
Let $X$ be a scheme, and suppose we are given for each open affine $U\subset X$ an $\mathcal O_X(U)$-module $M_U$, compatible with localisation; in other words, for each open affine $U$ and each $f\in\mathcal O_X(U)$ we have an isomorphism $\mathrm{res}^U_{D(f)}\colon M_U\otimes_{\mathcal O_X(U)}\mathcal O_X(D(f))\xrightarrow\sim M_{D(f)}$ of $\mathcal O_X(D(f))$-modules such that $\mathrm{res}^U_{D(fg)}=\mathrm{res}^{D(f)}_{D(fg)}\circ\mathrm{res}^U_{D(f)}$. Then there is a quasi-coherent sheaf $\mathcal M$ on $X$ such that $\mathcal M(U)=M_U$ for all open affines $U\subset X$.
\end{Lem}

\begin{proof}
Given two open affines $U,V\subset X$ and an open subset $W\subset U\cap V$, we can find $f_i\in\mathcal O_X(U)$ and $g_i\in\mathcal O_X(V)$ such that $D(f_i)=D(g_i)\subset W$ and $W=\bigcup_iD(f_i)$. This, together with the compatibility condition, implies that for each $\ptx\in X$ the open affine neighbourhoods of $\ptx$ form a directed system, and hence we can define the stalk $M_\ptx$. We can now define $\mathcal M(W)$ on any open subset $W\subset X$ to be those $(s_\ptx)\in\prod_{\ptx\in U}M_\ptx$ which are locally given by an element in some $M_U$. It is easy to see that this defines a sheaf on $X$, and that we may identify $\mathcal M(U)=M_U$ for each open affine $U\subset X$.
\end{proof}

We can use this lemma to construct first the sheaf of K\"ahler differentials on $X$, and then the tangent bundle. For each open affine $U\subset X$, let $\Omega_{\mathcal O_X(U)/K}$ be the $\mathcal O_X(U)$-module of K\"ahler differentials, and let $\mathrm{Sym}\updot(\Omega_{\mathcal O_X(U)/K})$ be its symmetric algebra. By the respective universal properties it is clear that these constructions are compatible with localisation, in the sense of the lemma. It follows that there are quasi-coherent sheaves $\Omega_{X/K}$ and $\mathrm{Sym}\updot(\Omega_{X/K})$ satisfying $\Omega_{X/K}(U)=\Omega_{\mathcal O_X(U)/K}$ and $\mathrm{Sym}\updot(\Omega_{X/K})(U)=\mathrm{Sym}\updot(\Omega_{\mathcal O_X(U)/K})$ for all open affines $U\subset X$.

Since $\mathrm{Sym}\updot(\Omega_{X/K})$ is a sheaf of $\mathcal O_X$-algebras, we may define the tangent bundle $TX$ to be the associated scheme $TX:=\Spec\mathrm{Sym}\updot(\Omega_{X/K})$. This comes with a natural morphism $\pi\colon TX\to X$, whose restriction over an open affine $U$ corresponds to the structure morphism $\mathcal O_X(U)\to\mathrm{Sym}\updot(\Omega_{\mathcal O_X(U)/K})$, and hence $\pi$ is an affine morphism. If $L$ is a field, then $(TX)(L)=X(D_1)$, so consists of pairs $(x,\xi)$ such that $x\in X(L)$ and $\xi\in T_x X$. Hence $T_xX$ is just the fibre of $TX\to X$ over the point $x$.

Note that if $X=\Spec A$ is affine, then $TX=\Spec\mathrm{Sym}\updot(\Omega_{A/K})$, and so $TX(R)$ is the set of algebra homomorphisms $\mathrm{Sym}\updot(\Omega_{A/K})\to R$, which we can identify with the set of algebra homomorphisms from $A$ to $R[t]/(t^2)$.

We remark that $TX\to X$ is in general not locally trivial, but following \cite[16.5.12]{EGA4.4} we will still refer to it as the tangent bundle.

\subsubsection{Upper semi-continuity}

A function $f\colon X\to\bZ$ is upper semi-continuous provided that for all $d$ there exists an open subset $U\subset X$ such that $f(\ptx)\leq d$ if and only if $\ptx\in U$. Equivalently, for any $x\in X(L)$ we have $f(x)\leq d$ if and only if $x\in U(L)$.

\begin{Prop}[Chevalley {\cite[Theorem 13.1.3]{EGA4.3}}]\label{Prop:usc}
Let $f\colon X\to Y$ be a morphism of schemes, with $f$ locally of finite type. Then the function $\ptx\mapsto\dim_\ptx f^{-1}(f(\ptx))$ is upper semi-continuous on $X$.
\end{Prop}

\begin{Cor}\label{Cor:usc}
Let $X$ be a scheme and $V\subset X\times\bA^n$ a closed subscheme. Write $\pi\colon V\to X$ for the restriction of the projection onto the first co-ordinate, and write $V_x:=\pi^{-1}(x)\subset L^n$ for the fibre over $x\in X(L)$. If each irreducible component of $V_x$ contains $(x,0)$, then the function $x\mapsto\dim V_x$ is upper semi-continuous on $X$.
\end{Cor}

\begin{proof}
Combine with the morphism $X\to V$, $x\mapsto(x,0)$.
\end{proof}

For example, if $X$ is locally of finite type over $K$, then the function $x\mapsto\dim T_xX$ is upper semi-continuous. For, the question is local on $X$, so we may assume $X=\Spec(A)$ for some $K$-algebra $A$ of finite type, and if $A$ is generated by $n$ elements, then $TX\subset X\times\bA^n$ is closed.

\subsection{Jet spaces}

Jet spaces are generalisations of tangent spaces, where we now consider $D_r$-valued points for $D_r:=L[[t]]/(t^{r+1})$ and $r\in[0,\infty]$. Note that $D_0=L$, $D_1=L[t]/(t^2)$ is the algebra of dual numbers, and $D_\infty=L[[t]]$ is the power series algebra. We have canonical algebra homomorphisms $D_r\to D_s$ whenever $r\geq s$, and $D_\infty=\varprojlim_r D_r$.

The set of $D_r$-valued points $X(D_r)$ is called the space of $r$-jets, and we again consider the pull-back
\[ \begin{CD}
T_x^{(r)}X @>>> \{x\}\\
@VVV @VVV\\
X(D_r) @>>> X(L).
\end{CD} \]

As for tangent spaces, an element of $X(D_r)$ for $r\in[0,\infty]$ can be thought of as a pair consisting of a point $\ptx\in X$ together with a local $K$-algebra homomorphism $\mathcal O_{X,\ptx}\to D_r$. If $r$ is finite, then every such map vanishes on $\mathfrak m_\ptx^{r+1}$, so it is enough to give $\mathcal O_{X,\ptx}/\mathfrak m_\ptx^{r+1}\to D_r$. Moreover, if $X$ is locally of finite type over $K$, then $\mathcal O_{X,\ptx}/\mathfrak m_\ptx^{r+1}$ is a finite-dimensional $K$-algebra and so $T_x^{(r)}X$ is an affine variety over $L$.

Greenberg's Theorem below allows one to restrict attention from $D_\infty$-valued points to $D_r$-valued points for $r$ sufficiently large.

\begin{Thm}[Greenberg \cite{Greenberg}]\label{Thm:Greenberg}
Let $A$ be a $K$-algebra of finite type. Then there exist positive integers $c$, $s$ and $N\geq c(s+1)$ such that, for all $r\geq N$ and all $\xi\in(\Spec A)(D_r)$, the image of $\xi$ in $(\Spec A)(D_{[r/c]-s})$ lifts to a point in $(\Spec A)(D_\infty)$.
\end{Thm}

We will not need such tight bounds, so observe that by doubling $N$ and $c$ we may assume that $s=0$.

The natural map $D_r\to D_1$ induces a map $T_x^{(r)}X\to T_xX$. We write $\overline T_x^{(r)}X$ for its image. If $X$ is locally of finite type, then this map is the restriction to $T_x^{(r)}X$ of a linear map, and hence the $\overline T_x^{(r)}X$ form a decreasing chain of constructible subsets, even cones, of the tangent space. We will show that this chain stabilises, and generically on $X$ the limit will be $T_xX_\red$.

A commutative local $K$-algebra $(A,\mathfrak m)$ is said to be formally smooth over $K$ provided it satisfies the infinitesimal lifting property: given a commutative $K$-algebra $R$ and a nilpotent ideal $N$, every algebra homomorphism $\xi\colon A\to R/N$ whose kernel contains some power of $\mathfrak m$ can be lifted to an algebra homomorphism $\hat\xi\colon A\to R$. If $A$ is Noetherian, then formally smooth implies regular, with the converse holding whenever $K$ is perfect \cite[(28.M) Proposition]{Matsumura1}.

\begin{Prop}
Let $X$ be a $K$-scheme and $x\in X(L)$. Then
\[ \overline T_x^{(\infty)}X \subset \bigcap_r\overline T_x^{(r)}X \subset T_x X_\red \subset T_x X. \]
If $x\colon\mathcal O_{X,\ptx}\to L$ and $(\mathcal O_{X,\ptx})_\red$ is formally smooth over $K$, then we have equalities
\[ \overline T_x^{(\infty)}X = \bigcap_r\overline T_x^{(r)}X = T_x X_\red. \]
\end{Prop}

\begin{proof}
Since the constructions are local it is enough to prove the corresponding statements for a local ring $(A,\mathfrak m)$. Let $\xi\colon A\to D_1=L[t]/(t^2)$ be a local algebra homomorphism, so $\xi(\mathfrak m)\subset(t)$.

If $\xi$ can be lifted to $\hat\xi\colon A\to D_{\infty}$, then \textit{a fortiori} it can be lifted to each $A\to D_r$. This proves that $\overline T_x^{(\infty)}X\subset\bigcap_r\overline T_x^{(r)}X$.

Now suppose that, for each $r\geq2$, the map $\xi\colon A\to D_1$ can be lifted to some $\xi_r\colon A\to D_r$. We want to show that $\xi$ factors over the reduced ring $A/\mathrm{nil}(A)$; that is, $\xi(a)=0$ for all nilpotent elements $a\in A$. Suppose $a^m=0$. We must have $\xi(a)=\alpha t$ for some $\alpha\in L$. Then $\xi_m(a)=\alpha t+\cdots$, so $0=\xi_m(a^m)=\alpha^mt^m$. Hence $\alpha=0$, so $a\in\Ker(\xi)$ as required. This proves that $\bigcap_r\overline T_x^{(r)}X\subset T_x X_\red$.

Finally, suppose that $A_\red$ is formally smooth. Then we can lift any $\xi\colon A_\red\to D_1$ to $D_2$, then to $D_3$, and so on, and thus to an algebra homomorphism $\hat\xi\colon A_\red\to D_\infty=\varprojlim D_r$. This proves that $T_x X_\red\subset\overline T_x^{(\infty)}X_\red$ whenever $(\mathcal O_{X,\ptx})_\red$ is formally smooth.
\end{proof}

\begin{Cor}
Let $X$ be a scheme, locally of finite type over $K$. Then for each $x$ there exists an $N$ such that $\overline T_x^{(N)}X=\overline T_x^{(\infty)}X$. In particular each $\overline T_x^{(\infty)}X$ is constructible inside the tangent space.

If $X$ is of finite type over $K$, then we can take the same $N$ for all $x$.
\end{Cor}

\begin{proof}
Take an open affine neighbourhood $\Spec A$ of $x$. Take $N$ and $c$ as in Greenberg's Theorem. If $r\geq N$ and $\xi\in\overline T_x^{(r)}X$, then we can lift $\xi$ to an element $\xi_r\in T_x^{(r)}X$, and we can find $\hat\xi\in T_x^{(\infty)}X$ such that $\hat\xi=\xi_r\in T_x^{([r/c])}X$. Since $[r/c]\geq1$ we have $\hat\xi=\xi\in T_xX$. Thus $\xi\in\overline T_x^{(\infty)}X$.

If $X$ is of finite type over $K$, then we can cover $X$ by finitely many open affine subschemes $\Spec A_i$. For each of these we have $N_i$ and $c_i$ as in the theorem. Now take $N=\max\{N_i\}$.
\end{proof}

We note that some finiteness condition on $X$ is essential. For, consider
\[ A := K[s_1,s_2,s_3,\ldots]/(s_1^1,s_2^2,s_3^3,\ldots) \]
and take $x\colon A\to K$, $s_i\mapsto 0$. Then
\[ \overline T_x^{(r)}\Spec A = \{(\xi_i)\in K^{(\mathbb N)}:\xi_i=0\textrm{ for }i\leq r\}, \]
so the $\overline T_x^{(r)}\Spec A$ form a strictly decreasing chain.

In a similar vein we have the following result.

\begin{Lem}
Let $X$ be a scheme, locally of finite type over $K$. Then the set $\overline T_x^{(2)}X$ is always a closed subvariety of $T_x X$, whereas $\overline T_x^{(3)}X$ is in general only constructible.
\end{Lem}

\begin{proof}
After changing base we may assume that $x$ is a rational point. Restricting to an open affine neighbourhood, we may then take $A=K[s_1,\ldots,s_m]/(f_1,\ldots,f_r)$ and $x\colon A\to K$, $s_i\mapsto 0$. Write $f_i$ in terms of its homogeneous parts
\[ f_i = \sum_pg_{ip}s_p + \sum_{p\leq q}h_{ipq}s_ps_q + \cdots. \]
Set $G=(g_{ip})$, an $r\times m$ matrix over $K$. Moreover, choosing an ordering for the pairs $(p,q)$ with $p\leq q$, we can consider $H=(h_{ipq})$ as an $r\times\binom{m+1}{2}$ matrix.

Now, an algebra homomorphism $\xi\colon A\to D_1$ extending $x$ is determined by $s_p\mapsto\xi_pt$ such that $\xi(f_i)=0$, so by a vector $(\xi_p)\in K^m$ such that $G(\xi_p)=0$. This identifies $T_x\Spec A$ with $\Ker(G)$. Similarly, $\xi\in\overline T_x^{(2)}\Spec A$ if and only if there is a vector $(\eta_p)\in K^m$ such that $s_p\mapsto \xi_pt+\eta_pt^2$ defines an algebra homomorphism $A\to D_2$. This is if and only if $G(\eta_p)+H(\xi_p\xi_q)=0$, and we can solve this inhomogeneous system if and only if $\mathrm{rank}(G,H(\xi_p\xi_q))\leq\mathrm{rank}(G)$. This proves that $\overline T_x^{(2)}\Spec A$ is closed; in fact, it is cut out by quadrics.

On the other hand, consider $A=K[s,t,u]/(st-u^3)$ and the point $x\colon A\to K$, $(s,t,u)\mapsto(0,0,0)$. Then $T_x\Spec A=K^3$ and $\overline T_x^{(2)}\Spec A=\{(\xi,\eta,\zeta):\xi\eta=0\}$, but
\[ \overline T_x^{(3)}X = \{(\xi,\eta,\zeta) : \xi\eta=0,\ (\xi,\eta)\neq(0,0)\}\cup\{(0,0,0)\}. \qedhere \]
\end{proof}

We now show how the spaces $\overline T_x^{(r)}X$ can be used to compute the tangent space $T_xX_\red$. This relies on the Theorem of Generic Smoothness \cite[I \S4 Corollary 4.10]{DG}.

\begin{Thm}[Generic Smoothness]
Let $X$ be a reduced scheme, locally of finite type over a perfect field $K$. Then the set of non-singular points in $X$ is open and dense.
\end{Thm}

\begin{Cor}\label{Cor:reduced}
Let $X$ be a scheme, locally of finite type over a perfect field $K$. Then for all $x$ in an open dense subset of $X$ we have $T_xX_\red=\overline T_x^{(r)}X$ for all $r$ sufficiently large.

In particular, $X$ is generically reduced if and only if, for all $x$ in an open dense subset of $X$, we have $\overline T_x^{(r)}X=T_xX$ for all $r$ sufficiently large.
\end{Cor}

\begin{proof}
For all $x$ we have $\overline T_x^{(r)}X=\overline T_x^{(\infty)}X$ when $r$ is sufficiently large, whereas Generic Smoothness implies that $T_x X_\red=\overline T_x^{(\infty)}X$ for all $x$ in an open dense subset. Finally, $X$ is generically reduced if and only if $T_x X=T_x X_\red$ on an open dense subset.
\end{proof}

\subsection{Separable morphisms}

A scheme $X$ is called integral provided it is both irreducible and reduced, and a dominant morphism $f\colon X\to Y$ between integral schemes is called separable whenever $K(X)/K(Y)$ is a separable field extension. More generally, if $X$ is integral and $f\colon X\to Y$ is a morphism, then $f$ is separable provided $f\colon X\to\overline{f(X)}$ is separable, where $\overline{f(X)}$ is the scheme-theoretic image of $f$.

\begin{Thm}\label{Thm:surj}
Let $f\colon X\to Y$ be a dominant morphism between integral schemes which are locally of finite type over $K$. Write $n:=\mathrm{tr.deg}\,K(X)/K(Y)$ for the relative degree. Then there exists an open dense $U\subset X$ and integers $d\geq n$ and $e$ such that
\[ \dim_L\,\Ker(d_x f)=d \quad\textrm{and}\quad \dim_L\,\Coker(d_x f)=e \quad\textrm{for all }x\in U(L). \]
Moreover,
\begin{enumerate}
\item $f$ is separable if and only if $d=n$.
\item $f$ separable implies $e=0$, and the converse holds whenever $K$ is perfect.
\end{enumerate}
\end{Thm}

\begin{proof}
The question is local, so we may assume that we have an embedding of domains $A\hookrightarrow B$, both of finite type over $K$. The first fundamental exact sequence for K\"ahler differentials \cite[(26.H) Theorem 57]{Matsumura1}
\[ \Omega_{A/K}\otimes_AB \to \Omega_{B/K} \to \Omega_{B/A} \to 0 \]
 yields, for each $x\colon B\to L$ with $L$ a field, an exact sequence
\[ 0 \to \Der_A(B,L) \to \Der_K(B,L) \to \Der_K(A,L), \]
and the right hand map can be identified with the differential
\[ d_x f\colon T_x\Spec B\to T_{f(x)}\Spec A. \]

The modules of K\"ahler differentials are all finitely generated over $B$, so by the Theorem of Generic Freeness \cite[(22.A) Lemma 1]{Matsumura1} we can localise $B$ to assume that each of these is free of finite rank over $B$. The dimensions of
\[ \Ker(d_x f), \quad T_x\Spec B \quad\textrm{and}\quad T_{f(x)}\Spec A, \]
and hence also of $\Coker(d_x f)$, are then all constant on an open dense subset of $\Spec B$.

It follows that we can compute these generic values by calculating them for the generic point of $\Spec B$. Let $L$ and $M$ be the fields of fractions of $A$ and $B$, respectively. Tensoring the first fundamental exact sequence with $M$ yields
\[ M\otimes_L\Omega_{L/K} \to \Omega_{M/K} \to \Omega_{M/L} \to 0. \]
Thus the kernel of the differential has dimension $\dim_M\Omega_{M/L}$ and the cokernel has dimension
\[ \dim_M\Omega_{M/L}-\dim_M\Omega_{M/K}+\dim_L\Omega_{L/K}. \]
By \cite[(27.B) Theorem 59]{Matsumura1} we know that
\[ d := \dim_M\Omega_{M/L} \geq \mathrm{tr.deg}\,M/L = n \]
and
\[ \dim_M\Omega_{M/K}\geq\dim_L\Omega_{L/K}+n \]
with equality in either case if and only if $M/L$ is separable. This proves (1) and the first part of (2).

Conversely suppose that $K$ is perfect. Then $M/K$ is separable, so $\dim\Omega_{M/K}=\mathrm{tr.deg}\,M/K$, and similarly for $L/K$. If moreover $e=0$, then
\[ \dim\Omega_{M/L} = \mathrm{tr.deg}\,M/K-\mathrm{tr.deg}\,L/K = \mathrm{tr.deg}\,M/L = n, \]
whence $M/L$ is separable.
\end{proof}

We observe that the converse in (2) is not true in general. For example, let $k$ have characteristic $p>0$, and set $K=k(t)$ and $L=k(u)$, together with an embedding $K\hookrightarrow L$, $t\mapsto u^p$. Then $X=\Spec L$ and $Y=\Spec K$ are both integral and $L/K$ is purely inseparable, but $T_xX=L$ and $T_yY=0$ so the differential is surjective.

Also, the dimension of the cokernel is not upper semi-continuous. For, consider the projection from a cuspidal cubic to a line in characteristic two
\[ f \colon X=\Spec K[s,t]/(s^2-t^3) \to \mathbb A^1 = \Spec K[t], \quad (a,b)\mapsto b. \]
Then $T_{(a,b)}X = \{(\xi,\eta):b\eta=0\}$ and $df$ is the projection onto the second co-ordinate. Thus $df$ is surjective at the origin and zero elsewhere (so $e=1$).

We next prove a nice result about separability of field extensions.

\begin{Thm}\label{Thm:sep-field}
A field extension $L/K$ is separable if and only if the $n$-fold tensor product $L^{\otimes\nfold Kn}=L\otimes_K\cdots\otimes_KL$ is reduced for all $n$. If $L/K$ is finitely generated, then there exists $N$ such that $L/K$ is separable if and only if $L^{\otimes\nfold KN}$ is reduced. If $L/K$ is algebraic, then we can take $N=2$.
\end{Thm}

\begin{proof}
If $L/K$ is separable, then by definition we know that $L\otimes_KA$ is reduced for all reduced $K$-algebras $A$. Conversely, suppose that $L/K$ is not separable. Then there exists an intermediate field $L'$ which is a primitive extension of a finitely-generated, purely transcendental extension of $K$ and such that $L'/K$ is inseparable. It is enough to prove that $(L')^{\otimes\nfold Kn}$ is not reduced for some $n$, and hence we may assume that $L=L'$.

Let $p=\mathrm{char}(K)>0$ and write $L=K(x_1,\ldots,x_d)$, where $F=K(x_1,\ldots,x_{d-1})$ is a purely transcendental extension of $K$. Let $x_d$ have minimal polynomial $f$ over $F$. Clearing denominators we obtain an equation $\sum_{i=1}^n\lambda_i\alpha_i=0$ with $\lambda_i\in K$ and each $\alpha_i$ a monomial in $x_1,\ldots,x_d$. Note that, by Gauss' Lemma, this polynomial is irreducible over each field $K(x_1,\ldots,\hat x_j,\ldots,x_d)$. Thus, if some $\alpha_i$ is not a $p$-th power, then some $x_j$ occurs in $\alpha_i$ with exponent not divisible by $p$, and hence $x_j$ is separable algebraic over $K(x_1,\ldots,\hat x_j,\ldots,x_d)$, a contradiction. Therefore we can write $\alpha_i=\overline \alpha_i^p$ for all $i$.

Set $R:=F^{\otimes\nfold Kn}$ and $S:=L^{\otimes\nfold Kn}$. Given $x\in L$, write
\[ x^{(j)} := 1^{\otimes{(j-1)}}\otimes x\otimes1^{\otimes{(n-j)}}\in S. \]
Then $R$ is an integral domain and $S\cong R[T_1,\ldots,T_n]/(f(T_1),\ldots,f(T_n))$ via the map $T_j\mapsto x_d^{(j)}$. Thus $S$ is a free $R$-module with basis the monomials $(x_d^{(1)})^{m_1}\cdots (x_d^{(n)})^{m_n}$ for $0\leq m_i<\deg(f)$.

Consider now the $n\times n$ matrices $M=\big(\alpha_i^{(j)}\big)$ and $\overline M=\big(\overline\alpha_i^{(j)}\big)$, having coefficients in $S$. If $\chi=\det\overline M$, then $\det M=\chi^p$. Also, since $\sum_i\lambda_i\alpha_i=0$, we know that $\det M=0$. On the other hand, using the Leibniz formula for $\det\overline M$, together with the basis for $S$ as a free $R$-module given above, we know that $\det\overline M\neq0$. Thus $\chi\in S$ is a non-trivial nilpotent element, so $S=L^{\otimes\nfold Kn}$ is not reduced.

If $L/K$ is finitely generated, then we can write $L=F(x_1,\ldots,x_e)$ with $F/K$ purely transcendental. Then $L/K$ is separable if and only if each $F(x_i)$ is separable over $K$, and hence there exists some $N$ such that $L/K$ is separable if and only if $L^{\otimes\nfold KN}$ is reduced. If $L/K$ is algebraic, then it is well-known that $L/K$ is separable if and only if $L\otimes_KL$ is reduced (see for example \cite[\S5.5 Exercise 11]{Cohn}, but note that it is falsely claimed in that exercise that $N=2$ works even for non-algebraic extensions).
\end{proof}

As an example, let $\alpha,\beta\in K\setminus K^p$ such that $\alpha^{1/p},\beta^{1/p}$ are $p$-independent over $K$, so $[K(\alpha^{1/p},\beta^{1/p}):K]=p^2$. Then $K$ is relatively algebraically closed in $L=K(x)[y]/(y^p+\alpha x^p+\beta)$. Now $L\otimes_KL\cong L\otimes_KK(\alpha^{1/p},x)$, where $\alpha^{1/p}=-(y^{(2)}-y^{(1)})/(x^{(2)}-x^{(1)})$, so is reduced, whereas $L\otimes_KL\otimes_KL\cong L\otimes_KK(\alpha^{1/p},x)\otimes_KK(\alpha^{1/p},x)$ is not reduced.

The theorem above leads to the following geometric criterion for separability.

\begin{Thm}\label{Thm:separable}
Let $f\colon X\to Y$ be a morphism, locally of finite type, between locally-Noetherian schemes\footnote{
We call a scheme locally-Noetherian if every open affine is the spectrum of a Noetherian algebra; it is Noetherian if it is locally-Noetherian and quasi-compact, so has a finite open affine cover by spectra of Noetherian algebras.
}
and assume that $X$ is integral. Then $f$ is separable if and only if, for each $n$, we can find an open dense $U\subset X$ such that $U^{\times\nfold Yn}$ is reduced.
\end{Thm}

\begin{proof}
If $Z$ is the scheme-theoretic image of $f$, then $X\times_YX\cong X\times_ZX$, so we may assume $Y=Z$, and hence that $f$ is a dominant morphism between integral schemes. Moreover, both conditions are invariant if we replace $X$ and $Y$ by open affines. Thus we have an inclusion of Noetherian domains $A\hookrightarrow B$, where $B$ is finitely-generated as an $A$-algebra, and with respective fields of fractions $K(Y)\hookrightarrow K(X)$.

Note that $K(X)^{\otimes\nfold{K(Y)}n}$ is the localisation of $B^{\otimes\nfold An}$ with respect to the multi\-plicatively-closed set $S=\{b\otimes\cdots\otimes b:0\neq b\in B\}$. Thus $\mathrm{nil}\big(K(X)^{\otimes\nfold{K(Y)}n}\big)=S^{-1}\mathrm{nil}\big(B^{\otimes\nfold An}\big)$, and since $B^{\otimes\nfold An}$ is Noetherian, we see that this vanishes if and only if there exists some non-zero $b\in B$ such that $b^{\otimes n}x=0$ for all $x\in\mathrm{nil}\big(B^{\otimes\nfold An}\big)$. Geometrically this says that we have a distinguished open $U=D(b)\subset X$ such that $U^{\times\nfold Yn}$ is reduced.

The result now follows from the previous theorem: $f$ is separable if and only, for each $n$, the nilradical of $K(X)^{\otimes\nfold{K(Y)}n}$ is zero.
\end{proof}

The next proposition shows that jet spaces behave well for separable morphisms.

\begin{Prop}\label{Prop:sep}
Let $f\colon X\to Y$ be a separable, dominant morphism, locally of finite type, between integral schemes. Then $T_x^{(r)}X\to T_{f(x)}^{(r)}Y$ is surjective for all $r\in[1,\infty]$ and all $x$ in an open dense subset of $X$.
\end{Prop}

\begin{proof}
The question is local on $X$, so we may assume that we have a monomorphism $A\hookrightarrow B$ of finite type between integral domains. Let $A$ and $B$ have quotient fields $L$ and $M$ respectively, so that $M/L$ is finitely generated and separable. Any set of generators for $B$ over $A$ must also generate $M$ over $L$, so contains a separating transcendence basis \cite[Theorem 26.2 and subsequent remark]{Matsumura2}. We may therefore assume that $B=A[u_1,\ldots,u_m,v_1,\ldots,v_n]$ with the $u_i$ forming a separating transcendence basis for $M/L$. Write $A[u]=A[u_1,\ldots,u_m]$, with quotient field $L(u)$, and let $h_i$ be the minimal polynomial of $v_i$ over $L(u)[v_1,\ldots,v_{i-1}]$, so a separable polynomial. Viewing the coefficients of $h_i$ as polynomials in $v_1,\ldots,v_{i-1}$ with coefficients in $L(u)$, we may take $\alpha\in A[u]$ to be a common denominator for all these coefficients, for all of the $h_i$. Thus $h_i$ is a polynomial in $A[u,\alpha^{-1},v_1,\ldots,v_{i-1}]$. Also, let $\beta\in B$ be the product of all $h_i'(v_i)$. Since $B$ is a domain and each $h_i$ is separable, $\beta$ is non-zero.

Suppose now that we have a commutative square
\[ \begin{CD}
A @>{\eta}>> D_r\\
@VVV @VVV\\
B @>{x}>> L
\end{CD} \]
with $x(\alpha),x(\beta)\neq0$. We can extend $\eta$ to $\xi_0\colon A[u]\to D_r$ by choosing $\xi_0(u_i)\in D_r$ to be any lift of $x(u_i)$. Since $D_r$ is local and $\xi_0(\alpha)$ is a lift of $x(\alpha)\neq0$, we know $\xi_0(\alpha)$ is invertible, so induces $\xi_0\colon A[u,\alpha^{-1}]\to D_r$.

Assume we have constructed $\xi_{i-1}\colon A[u,\alpha^{-1},v_1,\ldots,v_{i-1}]\to D_r$ extending $\eta$ and lying over $x$. In order to extend $\xi_{i-1}$ to a homomorphism $\xi_i\colon A[u,\alpha^{-1},v_1,\ldots,v_i]\to D_r$ lying over $x$, we need to construct $\xi_i(v_i)=x(v_i)+\sum_j\lambda_jt^j\in D_r$ such that $h_i(\xi_i(v_i))=0$. The coefficient of $t^j$ in $h_i(\xi_i(v_i))$ is a sum of $x(h_i'(v_i))\lambda_j$ together with things involving $\lambda_1,\ldots,\lambda_{j-1}$. Since $h_i'(v_i)$ is a factor of $\beta$ and $x(\beta)\neq0$, the coefficient of $\lambda_j$ is non-zero, so we can solve for $\lambda_j$ iteratively.

This shows that we can extend $\eta$ to a homomorphism $\xi\colon B\to D_r$ lying over $x$. Hence $T_x^{(r)}X\to T_{f(x)}^{(r)}Y$ is surjective for all $r$ on the open dense subset where $x(\alpha),x(\beta)\neq0$.
\end{proof}

\section{Detecting Irreducible Components}\label{Sec:irred-cpts}

We begin with the following easy observation.

\begin{Lem}\label{Lem:irred-cpt-non-sing}
Let $Y$ be a scheme, locally of finite type over $K$, and $X\subset Y$ an irreducible subscheme. If $T_xX=T_xY$ for some non-singular point $x\in X(L)$, then $\overline X$ is an irreducible component of $Y$.
\end{Lem}

\begin{proof}
Let $\ptx\in X$ be the corresponding point in the underlying topological space. Then $x$ non-singular and $X$ irreducible imply $\dim_LT_xX=\dim_\ptx X=\dim X$. On the other hand, $\dim_\ptx Y\leq\dim_LT_xY$, so $T_xX=T_xY$ implies $\dim X=\dim_\ptx Y$, whence $\overline X$ is an irreducible component of $Y$.
\end{proof}

In general, the subscheme $X$ will not be reduced, and so will not have any non-singular points. We therefore need to consider more general jet spaces and not just tangent spaces. 

\begin{Thm}
Let $K$ be a perfect field and $(A,\mathfrak m,L)$ a local $K$-algebra which is a Noetherian domain of dimension 1. Then there exists a field extension $L'/L$ and an injective algebra homomorphism $A\to L'[[t]]$ lifting the canonical map $A\to A/\mathfrak m=L\hookrightarrow L'$. If $A$ is finitely generated over $K$, then we may take $L'/L$ to be finite.
\end{Thm}

\begin{proof}
Let $B$ be the integral closure of $A$, so a Dedekind algebra by the Krull-Akizuki Theorem \cite[\S11 Theorem 11.7 and its corollary]{Matsumura2}. Let $\mathfrak n$ be a maximal ideal of $B$, so $B_{\mathfrak n}$ is a DVR, say with residue field $L'$. By the Cohen Structure Theorem \cite[(28.M) Proposition]{Matsumura1} its completion is isomorphic (as a $K$-algebra) to $L'[[t]]$. Finally, the natural maps $A\to B\to L'[[t]]$ are all injective.

If $A$ is finitely generated over a field, then $B$ is a finite $A$-module \cite[Corollary 8.11]{Reid}, hence $L'/L$ is finite.
\end{proof}

\begin{Cor}\label{Cor:min-gen}
Let $X$ be a scheme, locally of finite type over a perfect field $K$. Let $\ptx\in X$ and let $\ptx'$ be a minimal generalisation of $\ptx$. Then there exists a finite field extension $L/\kappa(\ptx)$ and a local homomorphism $\xi\colon\mathcal O_{X,\ptx}\to L[[t]]$ with kernel corresponding to $\ptx'$.
\end{Cor}

\begin{proof}
Apply the theorem to the local ring $\mathcal O_{X,\ptx}/\mathfrak p$, where $\mathfrak p$ is the prime ideal corresponding to $\ptx'$.
\end{proof}

We can now prove the following proposition.

\begin{Prop}\label{Prop:irred-cpt}
Let $Y$ be a scheme, locally of finite type over $K$, and let $X\subset Y$ be an irreducible subscheme. Then the following are equivalent.
\begin{enumerate}
\item $\overline X$ is an irreducible component of $Y$.
\item There exists an open dense subset $U\subset X$ such that $T_x^{(\infty)}X=T_x^{(\infty)}Y$ for all $x\in U(L)$.
\item There exists an open dense subset $U\subset X$ such that, for all $x\in U(L)$, we can find positive integers $N,c$ with the property that, when $r\geq N$, any $\eta\in T_x^{(r)}Y$ restricts to $\eta\in T_x^{([r/c])}X$.
\end{enumerate}
\end{Prop}

\begin{proof}
By shrinking $Y$ we may assume that $Y$ is irreducible and affine (so of finite type over $K$), and that $X\subset Y$ is closed.

$(1)\Rightarrow(3)\colon$ Take $N$ and $c$ as in Greenberg's Theorem for the scheme $Y$. If $r\geq N$ and $\eta\in T_x^{(r)}Y$, then we can find $\hat\eta\in T_x^{(\infty)}Y$ such that $\eta=\hat\eta\in T_x^{([r/c])}Y$. Since $X_\red=Y_\red$ by assumption we have $T_x^{(\infty)}Y=T_x^{(\infty)}X$, and so $\eta=\hat\eta\in T_x^{([r/c])}X$.

$(3)\Rightarrow(2)\colon$ Take $x\in U(L)$ and $\eta\in T_x^{(\infty)}Y$. For all $r\geq N$ we have $\eta\in T_x^{(rc)}Y$, so $\eta\in T_x^{(r)}X$. Thus $\eta\in T_x^{(\infty)}X$.

$(2)\Rightarrow(1)\colon$ Assume first that $K$ is a perfect field. Let $\ptx\in Y$ be the generic point of $X$. If this is not the generic point of $Y$, then we can take a minimal generalisation $\ptx'$. By \hyperref[Cor:min-gen]{Corollary \ref*{Cor:min-gen}} we can find a field $L$ and a local homomorphism $\mathcal O_{Y,\ptx}\to L[[t]]$ with kernel corresponding to $\ptx'$. Setting $x\colon\mathcal O_{Y,\ptx}\to L$ we get a point of $T_x^{(\infty)}Y$ which does not lie in $T_x^{(\infty)}X$, a contradiction.

In general let $\overline K$ be the algebraic closure of $K$, and recall that $X^{\overline K}$ is pure of dimension $\dim X$, and similarly for $Y$ \cite[I \S3 Corollary 6.2]{DG}. Then $U^{\overline K}$ is open and dense in $X^{\overline K}$ and clearly property (2) holds for $U^{\overline K}\subset X^{\overline K}$. Using the result for perfect fields we conclude that $\dim X=\dim Y$, so that $X_\red=Y_\red$.
\end{proof}

\subsection{Irreducible components via morphisms}

We now want to determine when the image of an irreducible scheme is dense in an irreducible component. We begin with the following sufficient criterion when the domain is reduced.

\begin{Lem}\label{Lem:sep-non-sing}
Let $f\colon X\to Y$ be a morphism between schemes which are locally of finite type over a perfect field $K$, and assume that $X$ is integral. If $d_xf\colon T_xX\to T_{f(x)}Y$ is surjective on an open dense subset of $X$, then $f$ is separable and $\overline{f(X)}$ is an irreducible component of $Y$.
\end{Lem}

\begin{proof}
Let $Y':=\overline{f(X)}$ be the scheme-theoretic image of $X$, so an integral subscheme of $Y$. Then $d_xf\colon T_xX\to T_{f(x)}Y'\hookrightarrow T_{f(x)}Y$. Our hypothesis implies that $T_xX\to T_{f(x)}Y'$ is surjective on an open dense subset $U$ of $X$, so $f$ is separable by \hyperref[Thm:surj]{Theorem \ref*{Thm:surj}}. On the other hand, Chevellay's Theorem tells us that $f(U)$ contains a dense open subset of $Y'$, so by Generic Smoothness it contains some non-singular point $\pty\in Y'$. Now, for some field $L$, we can find $x\in U(L)$ such that $f(x)\in Y'(L)$ corresponds to $\pty$, so is non-singular. Now $x\in U(L)$ implies $T_{f(x)}Y'=T_{f(x)}Y$, so $Y'$ is an irreducible component by \hyperref[Lem:irred-cpt-non-sing]{Lemma \ref*{Lem:irred-cpt-non-sing}}.
\end{proof}

If we know more about the points, then we can get away with fewer assumptions (c.f. \cite[I \S4 Corollary 4.14]{DG}).

\begin{Lem}
Let $f\colon X\to Y$ be a morphism between schemes which are locally of finite type over a perfect field $K$, and assume that $X$ is irreducible. If $x\in X(L)$ and $y=f(x)\in Y(L)$ are non-singular and $d_xf$ is surjective, then $f\colon X_\red\to Y$ is separable and $\overline{f(X)}$ is an irreducible component of $Y$.
\end{Lem}

\begin{proof}
The set of non-singular points of $X$ is always open, and it is non-empty by assumption. We now replace $X$ by its open dense subset of non-singular points.

Let $x$ and $y$ correspond to the points $\ptx$ and $\pty$ respectively.  We know from \hyperref[Prop:differentials-for-regular-local]{Proposition \ref*{Prop:differentials-for-regular-local}} that $\Omega_{X/K,\ptx}$ is free of rank $\dim_\ptx X=\dim X$, and similarly $\Omega_{Y/K,\pty}$ is free of rank $\dim_\pty Y$. The natural homomorphism $\theta\colon\mathcal O_{X,\ptx}\otimes_{\mathcal O_{Y,\pty}}\Omega_{Y/K,\pty}\to\Omega_{X/K,\ptx}$ now goes between finite free $\mathcal O_{X,\ptx}$-modules, and by assumption $\Hom(\theta,L)$ is surjective, whence $\Hom(\theta,\kappa(\ptx))$ is surjective. Representing $\theta$ by a matrix, we see that some minor of full rank is invertible in $\kappa(\ptx)$, so is invertible in $\mathcal O_{X,\ptx}$, proving  that $\theta$ is a split monomorphism. We deduce that the cokernel of $\theta$, denoted $\Omega_{X/Y,\ptx}$, is projective (and hence free) of rank $d:=\dim X-\dim_\pty Y$.

Let $\xi$ be the generic point of $X$, and set $\eta:=f(\xi)$, the generic point of $\overline{f(X)}$. Localising $\theta$ we deduce that the kernel of $d_\xi f$ has dimension $d$, which by \hyperref[Thm:surj]{Theorem \ref*{Thm:surj}} is at least $\dim X-\dim\overline{f(X)}$. It follows that $\dim\overline{f(X)}\geq\dim_\pty Y$, so we must have equality, and hence $\overline{f(X)}$ is an irreducible component of $Y$. Finally $f$ is separable by \hyperref[Thm:surj]{Theorem \ref*{Thm:surj}} once more.
\end{proof}

We now consider morphisms between non-reduced schemes.

\begin{Thm}\label{Thm:sep}
Let $f\colon X\to Y$ be a morphism, where $X$ and $Y$ are locally of finite type over $K$. Let $X'$ be an irreducible component of $X$, with the reduced subscheme structure, and let $Y'=\overline{f(X')}$ be the scheme-theoretic image. Consider the following statements.
\begin{enumerate}
\item $Y'$ is an irreducible component of $Y$.
\item There exists an open dense subset $U\subset X'$ such that $T_x^{(\infty)}X\to T_{f(x)}^{(\infty)}Y$ is surjective for all $x\in U(L)$.
\item There exists an open dense subset $U\subset X'$ such that, for all $x\in U(L)$, we can find positive integers $N,c$ with the property that, when $r\geq N$, any $\eta\in T_{f(x)}^{(r)}Y$  restricts to something in the image of $T_x^{([r/c])}X$.
\end{enumerate}
Then $(2)\Rightarrow(3)\Rightarrow(1)$, and $(1)\Rightarrow(2)$ whenever $f\colon X'\to Y$ is separable.
\end{Thm}

\begin{proof}
$(2)\Rightarrow(3)\colon$ Take $U\subset X'$ satisfying (2). Given $x\in U(L)$, choose $N$ and $c$ satisfying Greenberg's Theorem for $f(x)\in Y(L)$. Thus, given $r\geq N$ and $\eta\in T_{f(x)}^{(r)}Y$, we can find $\hat\eta\in T_{f(x)}^{(\infty)}Y$ such that $\eta=\hat\eta\in T_{f(x)}^{([r/c])}Y$. By assumption we have $\hat\xi\in T_x^{(\infty)}X$ mapping to $\hat\eta$, so its restriction $\xi\in T_x^{([r/c])}X$ maps to the restriction of $\eta$.

$(3)\Rightarrow(1)\colon$ By hypothesis, and applying \hyperref[Prop:irred-cpt]{Proposition \ref*{Prop:irred-cpt}} to the irreducible component $X'\subset X$, there exists an open dense $U\subset X'$ such that, for all $x\in U(L)$, we can find $N$ and $c$ with the property that, if $r\geq N$ and $\eta\in T_{f(x)}^{(r)}Y$, then $\eta$ restricts to something in the image of $T_x^{([r/c])}X'$, and hence to something in $T_{f(x)}^{([r/c])}Y'$. Since $f(U)\subset Y'$ is constructible and dense, it contains an open dense subset of $Y'$. Applying \hyperref[Prop:irred-cpt]{Proposition \ref*{Prop:irred-cpt}} once more, we deduce that $Y'\subset Y$ is an irreducible component.

$(1)\Rightarrow(2)$ when $f\colon X'\to Y$ is separable: Applying \hyperref[Prop:irred-cpt]{Proposition \ref*{Prop:irred-cpt}} to the irreducible components $X'\subset X$ and $Y'\subset Y$, and applying \hyperref[Prop:sep]{Proposition \ref*{Prop:sep}} to the morphism $f\colon X'\to Y'$, we see that we can find an open dense $U\subset X$ such that, for all $x\in U(L)$, we have both $T_x^{(\infty)}X'=T_x^{(\infty)}X$ and $T_{f(x)}^{(\infty)}Y' = T_{f(x)}^{(\infty)}Y$, and $d_xf\colon T_x^{(\infty)}X'\to T_{f(x)}^{(\infty)}Y'$ is surjective.
\end{proof}

\subsection{Subscheme structure on irreducible components}

\begin{Prop}\label{Prop:irred-comp}
Let $Y$ be a locally Noetherian scheme and $X$ an irreducible component of $Y$. Then there is a subscheme structure on $X$ such that $T_x^{(r)}X=T_x^{(r)}Y$ for all $r$ and all $x$ in an open dense subset of $X$.
\end{Prop}

\begin{proof}
Let $\xi$ be the generic point of $X$. For each open affine $U\subset Y$ the point $\xi$ determines an ideal $\mathfrak p=\mathfrak p_U$ of $\mathcal O_Y(U)$. If $\xi\in U$, then this is a minimal prime, in which case we know \cite[Corollary 4.11]{AM} that there is a unique $\mathfrak p$-primary ideal $\mathfrak q=\mathfrak q_U$ in the primary decomposition of the zero ideal in $\mathcal O_Y(U)$. We can therefore write $0=\mathfrak q\cap\mathfrak q'$ with $\mathfrak q'\not\subset\mathfrak p$. (In fact, if $a\in\mathfrak q'\setminus\mathfrak p$, then $\mathfrak q=\{b:ab=0\}$.) If $\xi\not\in U$, then set $\mathfrak p_U=\mathfrak q_U=\mathcal O_Y(U)$.

These ideals are compatible with localisation, in the sense of \hyperref[Lem:gluing]{Lemma \ref*{Lem:gluing}}. For, this clearly holds for the prime ideals $\mathfrak p_U$, so it necessarily holds for the primary ideals $\mathfrak q_U$ by uniqueness. Thus, by that lemma, there is an ideal sheaf $\mathcal I\lhd\mathcal O_Y$ such that $\mathcal I(U)=\mathfrak q_U$ for all open affines $U$. Since the support of the quotient sheaf $\mathcal O_Y/\mathcal I$ is just the irreducible component $X$, the ideal sheaf $\mathcal I$ determines a closed subscheme structure on $X$.

For the result about jet spaces, it is enough to prove it for affine schemes. We therefore have a Noetherian $K$-algebra $A$, a minimal prime $\mathfrak p$, and the $\mathfrak p$-primary ideal $\mathfrak q$ in the primary decomposition of the zero ideal. Writing $0=\mathfrak q\cap\mathfrak q'$ with $\mathfrak q'\not\subset\mathfrak p$ as above we get $A\hookrightarrow(A/\mathfrak q)\times(A/\mathfrak q')$. Then, for any $a\in\mathfrak q'\setminus\mathfrak p$ we have $A_a\cong A_a\otimes_A(A/\mathfrak q)$. In other words, the schemes $X$ and $Y$ agree on the non-empty distinguished open $D(a)$, so they have the same jet spaces.
\end{proof}

Note that the subscheme structure on irreducible components is not uniquely determined by this property on jet spaces. For, consider $A=K[X,Y]/(X^2,Y^3)$ and its proper quotient $B=A/(XY^2)$. Then $T^{(r)}\Spec B=T^{(r)}\Spec A$ for all $r\in[1,\infty]$.

\section{Group schemes}\label{Sec:group-schemes}

In our applications we will be interested in morphisms of the form $G\times X\to Y$, where $G$ is a group scheme acting on a scheme $Y$, and $X\subset Y$ is a subscheme. In particular, we want to know when such a morphism is separable.

Recall that a $K$-scheme $X$ is called geometrically irreducible provided that $X^{\overline K}$ is irreducible, or equivalently if $X\times Y$ is irreducible for all irreducible $K$-schemes $Y$. Similarly $X$ is called geometrically reduced provided that $X^{\overline K}$ is reduced, or equivalently if $X\times Y$ is reduced for all reduced $K$-schemes $Y$. Finally, $X$ is called geometrically integral provided that it is both geometrically irreducible and geometrically reduced, so that $X\times Y$ is integral for all integral $K$-schemes $Y$.

We note that an integral $K$-scheme $X$ is geometrically irreducible if and only if $K(X)/K$ is a primary field extension, is geometrically reduced if and only if $K(X)/K$ is a separable field extension, and geometrically integral if and only if $K(X)/K$ is a regular field extension.

Let $G$ be a group scheme, locally of finite type over $K$. We say that $G$ is connected provided the scheme is irreducible, in which case $G$ is geometrically irreducible \cite[II \S5 1.1]{DG}. Moreover, $G$ is smooth (so every point is non-singular) if and only if the identity element is a non-singular point, in which case $G$ is geometrically reduced \cite[II \S5 2.1]{DG}. Thus if $G$ is smooth and connected, then it is geometrically integral. Conversely, if $K$ is perfect, then $G$ reduced implies $G$ smooth \cite[II \S5 Corollary 2.3]{DG}. We also know that $G$ is pure, so all irreducible components have the same dimension \cite[II \S5 Theorem 1.1]{DG}.

We say that a group scheme $G$ acts on a scheme $X$ if there is a morphism $\mu\colon G\times X\to X$ inducing for all $R$ an action of the group $G(R)$ on the set $X(R)$. When working with group actions one runs into the problem that the category of schemes, although complete, is not cocomplete, so arbitrary coproducts or coequalisers need not exist. Thus orbits and more general quotients will in general not exist. For this reason it is sometimes convenient to embed the category of schemes into a cocomplete category and consider quotients in this larger category.

One way of doing this is to consider the category of faisceaux.\footnote{
A functor $X$ is a faisceau if it satisfies the sheaf property with respect to the fppf topology; that is, it respects finite direct products, so $X(R\times S)\cong X(R)\times X(S)$, and if $S$ is a faithfully-flat and finitely-presented $R$-algebra, then the map $X(R)\to X(S)$ identifies $X(R)$ with the equaliser of the two maps $X(\mathrm{id}_S\otimes1),X(1\otimes\mathrm{id}_S)\colon X(S)\to X(S\otimes_RS)$.
}
This is complete \cite[III \S1 1.12]{DG} and cocomplete \cite[III \S1 1.14]{DG}. In fact, faisceaux form an exact reflective subcategory of the category of all functors, so the inclusion has a left adjoint which preserves finite limits \cite[III \S1 Theorem 1.8]{DG}. Also, every scheme is a faisceau \cite[III \S1 Corollary 1.3]{DG}, and any morphism of schemes which is faithfully flat and locally of finite presentation is an epimorphism of faisceaux \cite[III \S1 Corollary 2.10]{DG}.

So, given a group action $\mu\colon G\times X\to X$, one may consider the pair of morphisms $\mu,\mathrm{pr}_2\colon G\times X\to X$ and take their coequaliser $X/G$ in the category of faisceaux. This is the faisceau associated to the functor $R\mapsto X(R)/G(R)$. 

On the other hand, the category of all locally-ringed spaces is also both complete \cite[I \S1 Remark 1.8]{DG} and cocomplete \cite[I \S1 Proposition 1.6]{DG}. As for schemes we can associate to any locally-ringed space a functor from $K$-algebras to sets. This determines a functor from the category of locally-ringed spaces to the category of faisceaux \cite[III \S1 Proposition 1.3]{DG}, which in turn has a left adjoint called the geometric realisation \cite[I \S1 Proposition 4.1]{DG}. It follows that the geometric realisation commutes with colimits. In particular, the geometric realisation of $X/G$ is automatically the coequaliser in the category of locally-ringed spaces.

Following \cite{GIT} a morphism of schemes $\pi\colon X\to Y$ is called a geometric quotient if it is submersive (so surjective and $Y$ has the quotient topology), constant on $G$-orbits, the induced morphism $\Psi\colon G\times X\to X\times_YX$ is surjective, and $\mathcal O_Y=\pi_\ast(\mathcal O_X)^G$; it is called a universal geometric quotient if for all $Z\to Y$, the base change $X_Z\to Z$ is still a geometric quotient for the action $G_Z\times X_Z\to X_Z$.

The next result is an improvement on \cite[Proposition 0.1]{GIT}. 

\begin{Prop}\label{Prop:geom-quotient}
Let $G$ act on $X$, and let $\pi\colon X\to Y$ be constant on $G$-orbits. Then $\pi$ is a geometric quotient if and only if it is the coequaliser of $\mu,\mathrm{pr}_2\colon G\times X\to X$ in the category of locally-ringed spaces.
\end{Prop}

\begin{proof}
Recall from  the proof of \cite[I \S1 Proposition 1.6]{DG} that $Y$ is the coequaliser in the category of all locally ringed spaces if and only if $Y$ is the coequaliser in the category of topological spaces, and $\mathcal O_Y$ is the equaliser (in the category of sheaves of rings on $Y$) of the two morphisms $\mu^\ast,\mathrm{pr}_2^\ast\colon\pi_\ast(\mathcal O_X)\to\psi_\ast(\mathcal O_{G\times X})$, where $\psi=\pi\mu=\pi\,\mathrm{pr}_2$.

The second property is clearly equivalent to saying that $\mathcal O_Y\cong\pi_\ast(\mathcal O_X)^G$, and $Y$ is the coequaliser in the category of topological spaces if and only if it is the coequaliser in the category of sets and $\pi$ is submersive. Thus it is enough to prove that when $\pi$ is submersive, $Y$ is the coequaliser in the category of sets if and only if $\Psi$ is surjective.

Suppose first that $\Psi$ is surjective, and let $\ptx,\ptx'\in\pi^{-1}(\pty)$. We know that the set of points in $X\times_YX$ projecting to $\ptx,\ptx'$ is given by $\Spec\big(\kappa(\ptx)\otimes_{\kappa(\pty)}\kappa(\ptx')\big)$ \cite[I \S1 Corollary 5.2]{DG}, so this is non-empty. Let $\xi$ be any such point. Since $\Psi$ is surjective there exists some $\ptz\in G\times X$ mapping to $\xi$, and hence $\mathrm{pr}_2(\ptz)=\ptx$ and $\mu(\ptz)=\ptx'$. As $\pi$ is surjective, this shows that $Y$ is the coequaliser in the category of sets.

Conversely, let $Y$ be the coequaliser in the category of sets and suppose $\ptx,\ptx'\in X$ map to the same point $\pty\in Y$. By definition there exists a sequence $\ptz_1,\ldots,\ptz_n$ in $G\times X$ such that
\[ \mathrm{pr}_2(\ptz_1)=\ptx, \quad \mu(\ptz_i)=\mathrm{pr}_2(\ptz_{i+1}), \quad \mu(\ptz_n)=\ptx'. \]
Now take a sufficiently large algebraically-closed field $L$ and points $(g_i,x_i)\in G(L)\times X(L)$ corresponding to $\ptz_i$. Set $x=x_1$ and $g:=g_n\cdots g_1\in G(L)$. Then $(g,x)\in G(L)\times X(L)$, so corresponds to some $\ptz\in G\times X$. Since $x=x_1$ and $g\cdot x=g_n\cdot x_n$ we have both $\mathrm{pr}_2(\ptz)=\mathrm{pr}_2(\ptz_1)=\ptx$ and $\mu(\ptz)=\mu(\ptz_n)=\ptx'$. Thus $\Psi\colon G\times X\to X\times_YX$ is surjective.
\end{proof}

\begin{Lem}\label{Lem:geom-quotient}
Let $G$ be a group scheme acting on a scheme $X$, and let $\pi\colon X\to Y$ be a morphism of schemes which is constant on $G$-orbits.
\begin{enumerate}
\item If $Y\cong X/G$ in the category of faisceaux, then $\pi$ is faithfully flat and a universal geometric quotient.
\item If $\pi$ is faithfully flat and quasi-compact, then it is a universal geometric quotient.
\end{enumerate}
\end{Lem}

\begin{proof}
(1) Suppose $Y\cong X/G$. Since $Y$ is a scheme it coincides with its geometric realisation, so $\pi$ is a geometric quotient by the proposition above and the preceding remarks. Moreover, given any $Z\to X/G$, the base change $X_Z\to Z$ is still a coequaliser for the action $G_Z\times X_Z\to X_Z$ \cite[III \S1 Example 2.5]{DG}, so that $\pi$ is in fact a universal geometric quotient. We also know that $\pi$ is an epimorphism of faisceaux and that its pull-back along itself is just the projection $G\times X\to X$, so flat. Thus $\pi$ is flat by \cite[III \S1 Corollary 2.11]{DG}. Since $\pi$ is surjective, it is faithfully flat.

(2)  If $\pi$ is faithfully flat and quasi-compact, then it is the coequaliser in the category of locally-ringed spaces \cite[I \S2 Theorem 2.7]{DG}, and hence is a geometric quotient. Moreover, since being faithfully flat and quasi-compact is preserved under base change \cite[I \S2 Propositions 2.2 and 2.5]{DG}, we see that it is a universal geometric quotient.
\end{proof}

Regarding the surjectivity of $\pi$, we remark that $X(L)/G(L)\to(X/G)(L)$ is a bijection for all algebraically-closed fields $L$ \cite[III \S1 Remark 1.15]{DG}.

We say that $G$ acts freely on $X$ if $G(R)$ acts freely on $X(R)$ for all $R$. It follows that the natural map $X(R)/G(R)\to(X/G)(R)$ is injective for all $K$-algebras $R$ \cite[I \S5.5]{RAGS}. In other words, if $x,x'\in X(R)$, then $\pi(x)=\pi(x')$ in $(X/G)(R)$ if and only if there exists $g\in G(R)$ such that $x'=g\cdot x$. We deduce that $\Psi\colon G\times X\to X\times_{X/G}X$, $(g,x)\mapsto(x,g\cdot x)$ is an isomorphism, so $\pi$ is a $G$-torsor \cite[III \S4 Corollary 1.7]{DG}. 

\begin{Cor}\label{Cor:geom-quotient}
Let $G$ be a group scheme acting freely on a scheme $X$. Let $\pi\colon X\to Y$ be a faithfully flat morphism of schemes which is constant on $G$-orbits, and assume $G$, $X$ and $Y$ are all locally of finite type over $K$.
\begin{enumerate}
\item $\dim X=\dim Y+\dim G$ and $X$ is pure if and only if $Y$ is pure.
\item $\pi$ is a smooth (respectively affine) morphism if and only if $G$ is smooth (respectively affine).
\item $X$ smooth implies $Y$ smooth, with the converse holding when $G$ is smooth.
\item If $X$ is irreducible and $G$ smooth, then $\pi\colon X_\red\to Y$ is separable.
\end{enumerate}
\end{Cor}

\begin{proof}
(1) Since every fibre is isomorphic to $G$, which is pure by \cite[II \S5 Theorem 1.1]{DG}, we can apply \cite[I \S3 Corollary 6.3]{DG} to get $\dim_\ptx X=\dim_{\pi(\ptx)}Y+\dim G$ for all $\ptx\in X$. Since $\pi$ is surjective, the result follows.

(2) A morphism $f\colon X\to Y$ is called smooth if it is flat, locally of finite presentation and all fibres are smooth schemes; it is called affine if $X\times_Y\Spec R$ is an affine scheme for all $\Spec R\to Y$. Next, a morphism is smooth if and only if its pull-back along a faithfully-flat morphism is smooth \cite[I \S4 4.1]{DG}. Similarly a morphism is affine if and only if its pull-back along a faithfully-flat and quasi-compact morphism is affine (one direction is immediate from the definition of affine morphisms, the other is \cite[I \S2 Corollary 3.9]{DG}).

Since $\pi$ is faithfully-flat and all schemes are locally of finite type over $K$, we see that $\pi$ is smooth (respectively affine) if and only if its pull-back along itself is smooth (respectively affine). This pull-back is just the projection $G\times X\to X$, which is smooth (respectively affine) if and only if $G$ is a smooth (respectively affine) scheme.

(3) If $X$ is smooth, then the faithful flatness of $\pi$ implies that $Y$ is smooth \cite[Proposition 17.7.7]{EGA4.4}. Conversely, if $G$ and $Y$ are smooth, then $\pi$ is smooth, and hence $X$ is smooth, since a composition of smooth morphisms is again smooth \cite[I \S4 Corollary 4.4]{DG}.

(4) As $G$ is smooth, it is geometrically reduced, so the subscheme $G\times X_\red$ is reduced and hence $\mu$ induces a morphism $G\times X_\red\to X_\red$. Thus $G$ acts on $X_\red$, so we may assume that $X$ is reduced, hence integral. We now observe that for all $n\geq1$ there is an isomorphism
\[ G^{n-1}\times X \xrightarrow{\sim} X^{\times\nfold Yn}, \quad (g_2,\ldots,g_n,x) \mapsto (x,g_2\cdot x,\ldots,g_n\cdot x). \]
The case $n=1$ is trivial, whereas for $n=2$ this is just the statement that $\Psi$ is an isomorphism. The general case follows by induction. Again, since $G$ is geometrically reduced, we conclude that $X^{\times\nfold Yn}$ is reduced for all $n$, whence $\pi$ is separable by \hyperref[Thm:separable]{Theorem \ref*{Thm:separable}}.
\end{proof}

For each $x\in X(L)$ we have the orbit map $\mu_x\colon G^L\to X^L$; this sends $g\in G(R)$ (for an $L$-algebra $R$) to $g\cdot x^R$. We also have the stabiliser $\Stab_G(x)$, defined to be the fibre of $G^L\to X^L$ over $x$, so a closed subgroup scheme of $G^L$. Note that the stabiliser acts freely on $G^L$ (on the right), and we define the orbit faisceau to be $\Orb_G(x):=G^L/\Stab_G(x)$.

\begin{Prop}\label{Prop:group-action}
Let a group scheme $G$ act on a scheme $X$, both of finite type over $K$.
\begin{enumerate}
\item For all $x\in X(L)$ the orbit $\Orb_G(x)$ is a subscheme of $X$. A point $y\in X(R)$ lies in $\Orb_G(x)(R)$ if and only if there exists a faithfully-flat and finitely-presented $R$-algebra $S$ and an element $g\in G(S)$ such that $y^S=g\cdot x^S$.
\item The orbit map $\mu_x\colon G^L\to\Orb_G(x)$ is faithfully flat and $\Orb_G(x)$ is pure of dimension $\dim G-\dim\Stab_G(x)$. 
\item The function $x\mapsto\dim\Stab_G(x)$ is upper semi-continuous on $X$.
\item $\mu_x$ is a smooth (respectively affine) morphism if and only if $\Stab_G(x)$ is smooth (respectively affine).
\item If $G$ is smooth, then so too is $\Orb_G(x)$, so is given by the reduced subscheme structure on the corresponding locally-closed subset of $X$.
\item If $G$ is connected and $\Stab_G(x)$ is smooth, then the orbit map $\mu_x\colon G^L_\red\to X^L$ is separable.
\end{enumerate}
\end{Prop}

\begin{proof}
(1) That the orbit $\Orb_G(x)\subset X^L$ is a subscheme follows from \cite[III \S3 Proposition 5.2]{DG}. The description of the points of $\Orb_G(x)(R)$ is given in \cite[I \S5.5]{RAGS}.

(2) The morphism $G^L\to G^L/\Stab_G(x)$ is faithfully flat by \hyperref[Lem:geom-quotient]{Lemma \ref*{Lem:geom-quotient}}, whereas by \hyperref[Cor:geom-quotient]{Corollary \ref*{Cor:geom-quotient}} (1) we know that the orbit $\Orb_G(x)$ is pure and $\dim G=\dim G^L=\dim\Stab_G(x)+\dim G^L/\Stab_G(x)$.

(3) Let $Z$ be the pull-back of $G\times X\to X\times X$ along the diagonal and consider the morphism $\pi\colon Z\to X$, so $\pi^{-1}(x)\cong\Stab_G(x)$. By the upper semi-continuity of fibre dimension we know that $z\mapsto\dim_z\pi^{-1}(\pi(z))=\dim_z\Stab_G(\pi(z))$ is upper semi-continuous on $Z$, and since $\Stab_G(\pi(z))$ is pure, the map $z\mapsto\dim\Stab_G(\pi(z))$ is upper semi-continuous. Now compose $\pi$ with the section $X\to Z$, $x\mapsto(1,x)$, to conclude that $x\mapsto\dim\Stab_G(x)$ is upper semi-continuous on $X$.

(4), (5) and (6) follow from \hyperref[Cor:geom-quotient]{Corollary \ref*{Cor:geom-quotient}} (2), (3) and (4) respectively.
\end{proof}

Let $G$ act on two schemes $X$ and $X'$. We then have a diagonal action on $X\times X'$, and hence we may consider the quotient faisceau $X\times^GX':=(X\times X')/G$, called the associated fibration. Note that, even if $X/G$ is a scheme, it does not automatically follow that $X\times^GX'$ is a scheme. We observe that if $G$ acts freely on $X$, then the diagonal action on $X\times X'$ is also free.

The following is useful for identifying associated fibrations.

\begin{Lem}\label{Lem:assoc-fib}
Let $G$ be a group faisceau acting on faisceaux $X$ and $X'$, and such that the action of $G$ on $X$ is free. Then $Z\cong X\times^GX'$ if and only if there exists a commutative diagram
\[ \begin{CD}
X\times X' @>{\phi}>> Z\\
@VV{\mathrm{pr}_1}V @VV{q}V\\
X @>{\pi}>> X/G
\end{CD} \]
such that $\phi$ is constant on $G$-orbits and the induced morphism $X\times X'\to X\times_{X/G}Z$ is an isomorphism.
\end{Lem}

\begin{proof}
We begin by showing that the associated fibration $X\times^GX'$ satisfies the conditions. Let $\phi'\colon X\times X'\to X\times^GX'$ be the coequaliser associated to the diagonal $G$-action on $X\times X'$, and observe that this is necessarily constant on $G$-orbits. The morphism $\pi\,\mathrm{pr}_1\colon X\times X'\to X/G$ clearly factors through $\phi'$, giving $q'\colon X\times^GX'\to X/G$ such that $q'\phi'=\pi\,\mathrm{pr}_1$. Finally, since $G\times X\xrightarrow{\sim}X\times_{X/G}X$, the induced morphism $X\times X'\to X\times_{X/G}(X\times^GX')$ is an isomorphism by \cite[III \S4 3.1]{DG}.

Now suppose we have $Z$, $\phi$ and $q$. Since $\phi$ is constant on $G$-orbits it induces a morphism $f\colon X\times^GX'\to Z$ such that $\phi=f\phi'$, so
\[ q'\phi' = \pi\,\mathrm{pr}_1 = q\phi = qf\phi'. \]
Since $\phi'$ is an epimorphism we must have $q'=qf$. Finally, taking the pull-back of $q'=qf$ along the epimorphism $\pi\colon X\to X/G$ gives the morphism
\[ X \times X' \xrightarrow{\sim} X\times_{X/G}(X\times^GX') \to X\times_{X/G}Z, \quad (x,x') \mapsto \big(x,\phi(x,x')\big), \]
which is an isomorphism by hypothesis. It follows from \cite[III \S1 Example 2.6]{DG} that $f$ is an isomorphism.
\end{proof}

Let $G$ act freely on $X$. A morphism of schemes $\pi\colon X\to Y$ is called a principal $G$-bundle provided it is locally trivial in the Zariski topology; that is, we have an open cover $Y=\bigcup_iU_i$ and isomorphisms $\phi_i\colon G\times U_i\xrightarrow{\sim}\pi^{-1}(U_i)$ such that $\pi\phi_i=\mathrm{pr}_2$ is the second projection, and $\phi_i(g,u)=g\cdot\phi_i(1,u)$.

\begin{Lem}\label{Lem:principal-G-bundle}
Let $\pi\colon X\to Y$ be a principal $G$-bundle. Then $Y\cong X/G$ is isomorphic to the quotient faisceau. In particular, $\pi$ is a universal geometric quotient.
\end{Lem}

\begin{proof}
Set $U:=\coprod_iU_i$, and let $q\colon U\to Y$ be the induced morphism, which is clearly faithfully flat and locally of finite presentation. If $p\colon X\to X/G$ denotes the faisceau quotient, then there is a morphism $f\colon X/G\to Y$ such that $\pi=fp$. Set $V:=(X/G)\times_YU$. Then by \cite[III \S1 Example 2.5]{DG} we have $V\cong(X\times_{X/G}V)/G$, using the natural action of $G$ on $X\times_{X/G}V$. On the other hand, since $\pi$ is locally trivial we have $X\times_{X/G}V\cong X\times_YU\cong G\times U$, under which the natural $G$-action corresponds to left multiplication on itself. We deduce that $V\cong U$, so that the pull-back of $f$ along $q$ is an isomorphism. Since $q$ is an epimorphism in the category of faisceaux \cite[III \S1 Corollary 2.10]{DG}, we conclude from \cite[III \S1 Example 2.6]{DG} that $f$ is an isomorphism, so $Y\cong X/G$.
\end{proof}

\begin{Lem}\label{Lem:principal-G-bundle-subscheme}
Let $\pi\colon X\to Y$ be a principal $G$-bundle. If $X'\subset X$ is a $G$-stable subscheme, then there exists a subscheme $Y'\subset Y$ such that $\pi\colon X'\to Y'$ is again a principal $G$-bundle.
\end{Lem}

\begin{proof}
Since $\pi$ is locally trivial there is an open covering $Y=\bigcup_iY_i$ and $G$-equivariant isomorphisms $\phi_i\colon G\times Y_i\xrightarrow{\sim}X_i:=\pi^{-1}(Y_i)$ satisfying $\pi\phi_i=\mathrm{pr}_2$. In particular, $Y$ is formed by gluing the schemes $Y_i$ along the open subschemes $U_{ij}:=\iota_i^{-1}(X_i\cap X_j)$, where $\iota_i\colon Y_i\to X_i$ is the section of $\pi$ given by $y\mapsto\phi_i(1,y)$.

Now, given a $G$-stable subscheme $X'\subset X$, set $X'_i:=X_i\cap X'$ and $Y'_i:=\iota_i^{-1}(X'_i)$. Then we can glue the $Y'_i$ along the subschemes $U'_{ij}:=\iota_i^{-1}(X'_i\cap X'_j)$, and hence obtain a subscheme $Y'\subset Y$ satisfying $Y'\cap Y_i=Y'_i$. It is then clear that $\phi_i$ restricts to an isomorphism $G\times Y'_i\xrightarrow\sim X'_i$, so that $\pi\colon X'\to Y'$ is again a principal $G$-bundle.
\end{proof}

\subsection{A separability criterion}

Let $G$ be a group scheme acting on a scheme $Y$, and let $X\subset Y$ be a subscheme. We can restrict the action of $G$ to get a morphism $\Theta\colon G\times X\to Y$, and we write $\overline{G\cdot X}$ for its scheme-theoretic image. Associated to $y\in Y(L)$ we have the fibre $\Theta^{-1}(y)$, which is the closed subscheme of $G^L\times X^L$ having $R$-valued points the pairs $(g,x)$ with $g\cdot x=y^R$. On the other hand we have the transporter of $y$ into $X$, denoted $\Transp(y,X)$, which is the fibre product of the orbit map $\mu_y\colon G^L\to Y^L$ with the immersion $X^L\to Y^L$. Thus $\Transp(y,X)$ is a subscheme of $G^L$ having $R$-valued points those $g\in G(R)$ such that $g\cdot y^R\in X(R)$. (This is a special case of \cite[II, \S1, Definition 3.4]{DG}.) There is clearly an isomorphism $\Theta^{-1}(y)\xrightarrow{\sim}\Transp(y,X)$, $(g,x)\mapsto g^{-1}$.

This leads to the following sufficient criterion for separability.

\begin{Thm}
Let $G$ be a smooth, connected group scheme, locally of finite type over $K$. Suppose $G$ acts on a scheme $Y$ and let $X\subset Y$ be a locally Noetherian integral subscheme. Then $\Theta\colon G\times X\to Y$ is separable whenever there exists, for each $n$, an open dense $U\subset X$ such that $\mathrm{Transp}(x,U)^n$ is reduced for all $x\in U$.
\end{Thm}

\begin{proof}
We want to apply \hyperref[Thm:separable]{Theorem \ref*{Thm:separable}}. Note therefore that, since $G$ is smooth and connected, it is geometrically integral, so $G\times X$ is integral; since both $G$ and $X$ are locally Noetherian, so too is $G\times X$; finally, as $G\times Y\to Y$ is locally of finite type, so is $G\times X\to Y$.

Given $n$, choose $U$ such that $\Transp(x,U)^{n-1}$ is reduced for all $x\in U$, and consider the projection $\pi\colon (G\times U)^{\times\nfold Yn}\to G\times U$ onto the first component. Then $\pi^{-1}(g,x)$ is isomorphic to $\Transp(x,U)^{n-1}$ via
\begin{align*}
\Transp(x,U)^{n-1} &\to (G\times U)^{\times\nfold Yn}\\
(h_2,\ldots,h_n) &\mapsto \big((g,x),(gh_2^{-1},h_2\cdot x),\ldots,(gh_n^{-1},h_n\cdot x)\big).
\end{align*}
In particular, each fibre is reduced. Since $G\times U$ is integral we can apply the Theorem of Generic Flatness \cite[I \S3 Theorem 3.7]{DG} to assume further (after possibly shrinking $U$) that $\pi$ is faithfully flat. Thus $\pi$ is faithfully flat with reduced fibres and reduced image, so has reduced domain \cite[(21.E) Corollary (iii)]{Matsumura1}. This proves that $(G\times U)^{\times\nfold Yn}$ is reduced, so $\Theta$ is separable by \hyperref[Thm:separable]{Theorem \ref*{Thm:separable}}.
\end{proof}

\begin{Lem}\label{Lem:orbit-transporter}
Let $G$ and $Y$ be of finite type over $K$, $H\leq G$ a closed subgroup and $y\in Y(L)$. Then the morphism
\[ H^L\times\Stab_G(y) \to \Transp(y,\Orb_H(y)), \quad (h,s) \mapsto hs, \]
induces an isomorphism between the associated fibration $H^L\times^{\Stab_H(y)}\Stab_G(y)$ and the transporter $\Transp(y,\Orb_H(y))$. In particular, $\Transp(y,\Orb_H(y))$ is pure of dimension $\dim H+\dim\Stab_G(y)-\dim\Stab_H(y)$.
\end{Lem}

\begin{proof}
We have a commutative diagram
\[ \begin{CD}
H^L\times\Stab_G(y) @>>> \Transp(y,\Orb_H(y))\\
@VVV @VVV\\
H^L @>>> \Orb_H(y)
\end{CD} \]
where the upper morphism $(h,s)\mapsto hs$ is constant on $\Stab_H(y)$-orbits. Moreover, this is a pull-back diagram. For, if $(h,g)$ is an $R$-valued point of the fibre product, then $h\cdot y^R=g\cdot y^R\in Y(R)$, so $s:=h^{-1}g\in\Stab_G(y)(R)$ and $(h,g)=(h,hs)$ as required. The isomorphism with the associated fibration now follows from \hyperref[Lem:assoc-fib]{Lemma \ref*{Lem:assoc-fib}}, whereas the purity and the dimension formula follow from \hyperref[Cor:geom-quotient]{Corollary \ref*{Cor:geom-quotient}}.
\end{proof}

Assume now that we have group schemes $H\leq G$, a $G$-action on $Y$, an $H$-stable subscheme $X\subset Y$, and that all schemes are of finite type over $K$. For each $x\in X(L)$ we have the orbits $\Orb_H(x)\subset X^L$ and $\Orb_G(x)\subset Y^L$, so we define $N_{X,x}:=T_xX/T_x\Orb_H(x)$ and similarly $N_{Y,x}:=T_xY/T_x\Orb_G(x)$. Writing $\Theta\colon G\times X\to Y$ as above for the restriction of the $G$-action on $Y$, we observe that the differential $d_x\Theta:=d_{(1,x)}\Theta$ induces a map $\theta_x\colon N_{X,x}\to N_{Y,x}$.

\begin{Thm}\label{Thm:group-action}
With the notation as above, assume that $G$ and $H$ are smooth and connected, and that $X$ is integral. Suppose further that there exists an open dense $U\subset X$ such that, for all $x\in U(L)$, both $\Stab_G(x)$ and $\Stab_H(x)$ are smooth, and that there are only finitely many $H(\bar L)$-orbits on $\big(\!\Orb_G(x)\cap X\big)(\bar L)$. Then $\Theta$ is separable if and only if $\theta_x$ is injective on an open dense subset of $X$. On the other hand, $d_x\Theta$ is surjective if and only if $\theta_x$ is surjective.
\end{Thm}

\begin{proof}
Write $n$ for the relative degree of $\Theta$, so $n:=\dim G+\dim X-\dim\overline{G\cdot X}$. Since $G\times X$ is integral and of finite type we can apply \hyperref[Thm:surj]{Theorem \ref*{Thm:surj}} to deduce that $\Theta$ is separable if and only if $\dim\Ker(d_{(g,x)}\Theta)=n$ on an open dense subset of $G\times X$. Using the $G$-action we see that $\dim\Ker(d_{(g,x)}\Theta)$ is independent of $g$, so it is enough to show $\dim\Ker(d_x\Theta)=n$ on an open dense subset of $X$. Also, since the dimensions remain unchanged after base change, we may assume that $K$ is algebraically closed, in which case it is enough to consider rational points $x\in X(K)$.

By upper semi-continuity of the dimensions of the stabiliser groups, \hyperref[Prop:group-action]{Proposition \ref*{Prop:group-action}}, and shrinking $U$ if necessary, we may assume that $\dim\Stab_G(x)=a$ and $\dim\Stab_H(x)=b$ are constant for all $x\in U(K)$, and that these values are the generic, or minimal, values on all of $X$. We have already seen that the fibre $\Theta^{-1}(x)$ is isomorphic to $\Transp(x,X)=\Transp(x,\Orb_G(x)\cap X)$. By assumption we can write $\Orb_G(x)\cap X$ as a finite union $\bigcup_i\Orb_H(x_i)$, so that $\Transp(x,X)=\bigcup_i\Transp(x,\Orb_H(x_i))$. It then follows from \hyperref[Lem:orbit-transporter]{Lemma \ref*{Lem:orbit-transporter}} that
\[ \dim\Theta^{-1}(x) = \max_i\{\dim\Transp(x,\Orb_H(x_i))\} = \dim H+a-b. \]
For, $x_i\in\Orb_G(x)$, so $\dim\Stab_G(x_i)=\dim\Stab_G(x)=a$, and $\dim\Stab_H(x)=b=\min_i\{\dim\Stab_H(x_i)\}$. Thus all fibres of $\Theta$ have dimension $\dim H+a-b$, and this is necessarily the relative degree $n$.

Consider now the commutative squares
\[ \begin{CD}
G\times H @>{(1,\mu_x)}>> G\times\Orb_H(x)\\
@VV{\alpha}V @VV{\mu}V\\
G @>{\mu_x}>> \Orb_G(x)
\end{CD}
\qquad\textrm{and}\qquad
\begin{CD}
G\times\Orb_H(x) @>>> G\times X\\
@VV{\mu}V @VV{\Theta}V\\
\Orb_G(x) @>>> Y
\end{CD} \]
where $\alpha\colon(g,h)\mapsto gh$ is just the group multiplication. Since both $\Stab_G(x)$ and $\Stab_H(x)$ are smooth we know from \hyperref[Prop:group-action]{Proposition \ref*{Prop:group-action}} that the orbit maps are separable, so their differentials are generically surjective by \hyperref[Prop:sep]{Proposition \ref*{Prop:sep}}. Using the group actions we conclude that the differentials are always surjective. So, writing  $\mathfrak g=\mathrm{Lie}(G)=T_1G$, and similarly $\mathfrak h=\mathrm{Lie}(H)=T_1H$, and setting $d_x\mu:=d_{(1,x)}\mu$, we get two exact commutative diagrams
\[ \begin{CD}
\mathfrak g\times\mathfrak h @>>> \mathfrak g\times T_x\Orb_H(x) @>>> 0\\
@V{d\alpha}VV @V{d_x\mu}VV\\
\mathfrak g @>>> T_x\Orb_G(x) @>>> 0
\end{CD} \]
and
\[ \begin{CD}
0 @>>> \mathfrak g\times T_x\Orb_H(x) @>>> \mathfrak g\times T_xX @>>> N_{X,x} @>>> 0\\
@. @V{d_x\mu}VV @V{d_x\Theta}VV @V{\theta_x}VV\\
0 @>>> T_x\Orb_G(x) @>>> T_xY @>>> N_{Y,x} @>>> 0
\end{CD} \]
Now, $\alpha$ factors as the isomorphism $G\times H\xrightarrow\sim G\times H$, $(g,h)\mapsto(gh,h)$, followed by the projection onto the first co-ordinate. Thus $d\alpha$ is surjective, so $d_x\mu$ is also surjective, and hence $d_x\Theta$ is surjective if and only if $\theta_x$ is surjective. Moreover, since the groups $G$ and $H$ are smooth, as are the orbits, we can compute
\begin{align*}
\dim\Ker(d_x\mu) &= \dim G + \dim\Orb_H(x) - \dim\Orb_G(x)\\
&= \dim H + \dim\Stab_G(x) - \dim\Stab_H(x) = n.
\end{align*}
Using that $\dim\Ker(d_x\Theta)=\dim\Ker(d_x\mu)+\dim\Ker(\theta_x)$ we conclude from \hyperref[Thm:surj]{Theorem \ref*{Thm:surj}} that $\Theta$ is separable if and only if $\theta_x$ is injective on an open dense subset of $X$.
\end{proof}

We summarise these considerations in the following corollary, which will be sufficient for all our applications.

\begin{Cor}\label{Cor:irred-cpt}
Suppose we have smooth, connected group schemes $H\leq G$, a $G$-action on $Y$, an $H$-stable subscheme $X\subset Y$, and that all schemes are of finite type over $K$. Assume further that there is an open dense subset $U\subset X$ such that, for all $x\in U(L)$,
\begin{enumerate}
\item the stabilisers $\Stab_G(x)$ and $\Stab_H(x)$ are smooth.
\item there are only finitely many $H(\bar L)$-orbits on $\big(\!\Orb_G(x)\cap X\big)(\bar L)$.
\item $\theta_x$ is injective.
\item $T_1^{(\infty)}G\times T_x^{(\infty)}X\to T_x^{(\infty)}Y$ is surjective if and only if $\theta_x$ is surjective.
\end{enumerate}
Then for any irreducible component $X'\subset X$ we have that $\overline{G\cdot X'}$ is an irreducible component of $Y$ if and only if $\theta_x$ is surjective for all $x$ in an open dense subset of $X'$.
\end{Cor}

\begin{proof}
Let $X'\subset X$ be an irreducible component, endowed with its reduced subscheme structure. Since $H$ is connected, it must act on $X'$. In particular, if $x\in X'(L)$, then $\Orb_H(x)\subset X'$, and hence $N_{X',x}\subset N_{X,x}$. Thus since $\theta_x$ is injective for all $x\in U(L)\cap X'(L)$, so too is $\theta'_x\colon N_{X',x}\to N_{Y,x}$. There are also only finitely many $H(\bar L)$-orbits on $\big(\!\Orb_G(x)\cap X'\big)(\bar L)$, and so by the previous theorem the morphism $\Theta'\colon G\times X'\to Y$ is separable.

Now, by \hyperref[Thm:sep]{Theorem \ref*{Thm:sep}}, the scheme-theoretic image $\overline{G\cdot X'}$ is an irreducible component of $Y$ if and only if $T_1^{(\infty)}G\times T_x^{(\infty)}X\to T_x^{(\infty)}Y$ is surjective on an open dense subset of $X'$, which by (4) is if and only if $\theta_x$ is surjective on an open dense subset of $X'$.
\end{proof}

\section{Schemes of representations}\label{Sec:rep-schemes}

In the following four sections we will apply \hyperref[Cor:irred-cpt]{Corollary \ref*{Cor:irred-cpt}} to various schemes arising from the representation theory of finitely-generated algebras. The schemes we will consider all have a natural description as functors, but are in general non-reduced, even generically non-reduced (so the local rings at the generic points are non-reduced).

We will work over a fixed perfect field $K$.

We first consider the case studied by Crawley-Boevey and Schr\"oer in \cite{CBS}.

\subsection{Schemes parameterising representations}

Let $\Lambda=K\langle x_1,\ldots,x_N\rangle/I$ be a finitely generated (but not necessarily commutative) $K$-algebra. For each $d$ there is an affine scheme $\rep_\Lambda^d$ parameterising all $d$-dimensional $\Lambda$-modules, whose $R$-valued points (for a commutative $K$-algebra $R$) are given by
\[ \rep_\Lambda^d(R) := \{(X_1,\ldots,X_N)\in\bM_d(R)^N : f(X)=0 \textrm{ for all }f\in I\}. \]
Thus $\rep_\Lambda^d$ is a closed subscheme of $\bM_d^N$, and hence is of finite type over $K$. It is clear that the points of $\rep_\Lambda^d(R)$ are in bijection with the set of $K$-algebra homomorphisms $\Lambda\to\bM_d(R)$, and we will usually identify these two sets.

Associated to $\rho\in\rep_\Lambda^d(R)$ is a $\Lambda$-module $M_\rho$, having underlying vector space $R^d$ and action determined by the algebra homomorphism $\rho\colon\Lambda\to\bM_d(R)$; in fact $M_\rho$ is naturally a module over $\Lambda^R:=\Lambda\otimes_KR$, free of rank $d$ as an $R$-module. Also, given $\sigma\in\rep_\Lambda^e(R)$, we can identify $\Hom_{\Lambda^R}(M_\rho,M_\sigma)$ with
\begin{align*}
\Hom(\rho,\sigma) &:= \{ f\in\bM_{e\times d}(R) : f\rho=\sigma f \}\\
&\phantom{:}= \{ f\in\bM_{e\times d}(R) : f\rho_i=\sigma_if\textrm{ for all }i \}.
\end{align*}
The group scheme $\GL_d$ acts on $\rep_\Lambda^d$ by `change of basis'
\[ \GL_d(R)\times\rep_\Lambda^d(R) \to \rep_\Lambda^d(R), \  g\cdot(\rho_1,\ldots,\rho_N)=(g\rho_1g^{-1},\ldots,g\rho_Ng^{-1}). \]
It follows that for $L$-valued points the orbits are in bijection with the set of isomorphism classes of $d$-dimensional $\Lambda^L$-modules. Note that if $\rho\in\rep_\Lambda^d(L)$, then $\Stab_{\GL_d}(\rho)\cong\Aut_{\Lambda^L}(M_\rho)$, which is open in the vector space $\End_{\Lambda^L}(M_\rho)$ and hence is smooth and irreducible.

Given two representations $\rho\in\rep_\Lambda^d(R)$ and $\sigma\in\rep_\Lambda^e(R)$, their direct sum is the representation
\[ \rho\oplus\sigma\in\rep_\Lambda^{d+e}(R), \quad (\rho\oplus\sigma)_i := \begin{pmatrix}\rho_i&0\\0&\sigma_i\end{pmatrix}. \]
In terms of algebra homomorphisms we have
\[ \rho\oplus\sigma \colon \Lambda \to \bM_{d+e}(R), \quad a \mapsto \begin{pmatrix}\rho(a)&0\\0&\sigma(a)\end{pmatrix}. \]
This induces a closed immersion
\[ \rep_\Lambda^d\times\rep_\Lambda^e \to \rep_\Lambda^{d+e}, \]
which we can combine with the action of $\GL_{d+e}$ on $\rep_\Lambda^{d+e}$ to obtain a morphism of schemes
\[ \Theta \colon \GL_{d+e}\times\rep_\Lambda^d\times\rep_\Lambda^e \to \rep_\Lambda^{d+e}, \quad (g,\rho,\sigma) \mapsto g\cdot(\rho\oplus\sigma). \]
In the notation of \hyperref[Cor:irred-cpt]{Corollary \ref*{Cor:irred-cpt}} we have the smooth, connected groups $G=\GL_{d+e}$ and $H=\GL_d\times\GL_e$, acting on the schemes $Y=\rep_\Lambda^{d+e}$ and $X=\rep_\Lambda^d\times\rep_\Lambda^e$. Note that all stabilisers are smooth, so condition (1) is satisfied, and the following lemma proves that condition (2) also holds.

\begin{Lem}\label{Lem:fin-many-orbits}
For all fields $L$ and all $\rho\oplus\sigma\in\rep_\Lambda^d(L)\times\rep_\Lambda^e(L)$, there are only finitely many $\GL_d(L)\times\GL_e(L)$-orbits on $\Orb_{\GL_{d+e}(L)}(\rho\oplus\sigma)\cap\big(\rep_\Lambda^d(L)\times\rep_\Lambda^e(L)\big)$.
\end{Lem}

\begin{proof}
By the Krull-Remak-Schmidt Theorem we can find representations $\tau_i$ such that each $M_i:=M_{\tau_i}$ is indecomposable as a $\Lambda^L$-module, and $M_\rho\oplus M_\sigma\cong\bigoplus_iM_i$. Now $\rho'\oplus\sigma'$ lies in $\Orb_{\GL_{d+e}(L)}(\rho\oplus\sigma)\cap\big(\rep_\Lambda^d(L)\times\rep_\Lambda^e(L)\big)$ if and only if there exists a set $I$ such that $M_{\rho'}\cong\bigoplus_{i\in I}M_i$ and $M_{\sigma'}\cong\bigoplus_{i\not\in I}M_i$, in which case $\rho'\oplus\sigma'$ lies in the same $\GL_d(L)\times\GL_e(L)$-orbit as $\big(\bigoplus_{i\in I}\tau_i\big)\oplus\big(\bigoplus_{i\not\in I}\tau_i\big)$. We therefore see that the $\GL_d(L)\times\GL_e(L)$-orbits in the intersection are parameterised by certain subsets $I$, and hence there are only finitely many.
\end{proof}

\subsubsection{Example}

Consider the algebra
\[ \Lambda = \begin{pmatrix}\bC&\bC\\0&\bR\end{pmatrix}. \]
Using the presentation
\[ \bR\langle x,y\rangle/\big(yx,(1+x^2)x,(1+x^2)y\big)\xrightarrow{\sim}\Lambda, \quad x\mapsto\begin{pmatrix}i&0\\0&0\end{pmatrix}, \quad
y\mapsto\begin{pmatrix}0&1\\0&0\end{pmatrix} \]
one can easily compute some small examples.

The scheme $\rep_\Lambda^1$ is given by the spectrum of the algebra
\[ \bR[X,Y]/\big(YX,(1+X^2)X,(1+X^2)Y\big) \cong \bR[X]/(X+X^3) \cong \bR\times\bC. \]
Thus there is a unique $\bR$-valued point, and this corresponds to the simple injective $\Lambda$-module $I$. There are also three $\bC$-valued points
\[ \bR[X]/(X+X^3)\to\bC, \quad X\mapsto 0,\pm i. \]
The first corresponds to $I\otimes_\bR\bC$, which is the simple injective $\Lambda^\bC$-module, whereas the latter two correspond to the two simple projective $\Lambda^\bC$-modules. Note that, via restriction of scalars, the latter two both induce the simple projective $\Lambda$-module $P$, but $P$ should rather be thought of as the following $\bR$-rational point of $\rep_\Lambda^2$
\[ \Lambda \mapsto \bM_2(\bR), \quad x \mapsto \begin{pmatrix}0&1\\-1&0\end{pmatrix}, \quad y \mapsto \begin{pmatrix}0&0\\0&0\end{pmatrix}. \]

\subsection{Tangent spaces and derivations}

Define a closed subscheme $\Der(d,e)\subset\rep_\Lambda^{d+e}$ by taking those representations in block form having a zero block of size $d\times e$ in the bottom left corner. Thus $\begin{pmatrix}\sigma&\xi\\0&\rho\end{pmatrix}$ lies in $\Der(d,e)(R)$ if and only if $\rho\in\rep_\Lambda^d(R)$ and $\sigma\in\rep_\Lambda^e(R)$, and $\xi\in\Der(\rho,\sigma):=\Der_K\big(\Lambda,\Hom_R(M_\rho,M_\sigma)\big)$ is a $K$-derivation, or crossed homomorphism, so a $K$-linear map
\[ \xi \colon \Lambda\to\Hom_R(M_\rho,M_\sigma) \quad\textrm{satisfying}\quad \xi(ab) = \xi(a)\rho(b)+\sigma(a)\xi(b). \]
The associated module $M_\xi$ fits naturally into a short exact sequence
\[ 0 \to M_\sigma \to M_\xi \to M_\rho \to 0, \]
and hence induces an extension class in $\Ext^1_{\Lambda^R}(M_\rho,M_\sigma)$.

Note that the projection $\pi\colon\Der(d,e)\to\rep_\Lambda^d\times\rep_\Lambda^e$ has fibres $\pi^{-1}(\rho,\sigma)=\Der(\rho,\sigma)$.

Next, since $\rep_\Lambda^d$ is an affine scheme, the tangent bundle $T\rep_\Lambda^d$ is given by $T\rep_\Lambda^d(R)=\rep_\Lambda^d\big(R[t]/(t^2)\big)$, so consists of algebra homomorphisms $\rho+\xi t\colon\Lambda\to\bM_d\big(R[t]/(t^2)\big)$; equivalently, $\rho\in\rep_\Lambda^d(R)$ and $\xi\in\Der(\rho,\rho)$. Thus $T\rep_\Lambda^d$ is isomorphic to the fibre product of $\pi\colon\Der(d,d)\to\rep_\Lambda^{2d}$ with the diagonal map $\rep_\Lambda^d\to\rep_\Lambda^{2d}$, so is a closed subscheme of $\rep_\Lambda^{2d}$. In particular, given $\rho\in\rep_\Lambda(L)$, we can identify the tangent space $T_\rho\rep_\Lambda^d$ with $\Der(\rho,\rho)$, and this comes with a natural morphism to $\Ext^1_{\Lambda^L}(M_\rho,M_\rho)$.

Finally observe that if $\rho\in\rep_\Lambda^d(R)$ and $\sigma\in\rep_\Lambda^e(R)$, then the natural decomposition of $\End_R(M_\rho\oplus M_\sigma)$ induces a corresponding decomposition of $\Der(\rho\oplus\sigma,\rho\oplus\sigma)$.

\subsubsection{Voigt's Lemma and Hochschild cohomology}

Recall that we have the group action $\GL_d\times\rep_\Lambda^d \to \rep_\Lambda^d$. On $D_1:=L[t]/(t^2)$-valued points this gives
\[ (1+\gamma t)\cdot(\rho+\xi t) = \rho+(\xi+\gamma\rho-\rho\gamma)t, \]
so taking the differential at $(1,\rho)$ yields the additive group action
\[ \bM_d(L)\times T_\rho\rep_\Lambda^d \to T_\rho\rep_\Lambda^d, \quad (\gamma,\xi) \mapsto \xi+(\gamma\rho-\rho\gamma). \]
In particular this restricts to a linear map
\[ \delta_\rho \colon \bM_d(L) \to T_\rho\rep_\Lambda^d=\Der(\rho,\rho), \quad \delta_\rho(\gamma) = \gamma\rho-\rho\gamma. \]
More generally, given $\sigma\in\rep_\Lambda^e(L)$, we have a linear map
\[ \delta_{\rho,\sigma} \colon \bM_{e\times d}(L) \to \Der(\rho,\sigma), \quad \delta_{\rho,\sigma}(\gamma) := \gamma\rho-\sigma\gamma. \]

\begin{Lem}[Voigt \cite{Gabriel}]\label{Lem:Voigt}
Given $\rho\in\rep_\Lambda^d(L)$ and $\sigma\in\rep_\Lambda^e(L)$ there exists an exact sequence
\[ 0 \longrightarrow \Hom_{\Lambda^L}(M_\rho,M_\sigma) \longrightarrow \bM_{e\times d}(L)
\overset{\delta_{\rho,\sigma}}{\longrightarrow} \Der(\rho,\sigma) \overset{\varepsilon_{\rho,\sigma}}{\longrightarrow} \Ext^1_{\Lambda^L}(M_\rho,M_\sigma) \longrightarrow 0, \]
where $\varepsilon_{\rho,\sigma}$ takes $\xi$ to the class of the natural extension
\[ 0 \longrightarrow M_\sigma \longrightarrow M_\xi \longrightarrow M_\rho \longrightarrow 0. \]
\end{Lem}

\begin{proof}
We begin by making the connection to Hochschild cohomology. Consider the complex
\[ d^n \colon \Hom_K\big(\Lambda^{\otimes n},\Hom_L(M_\rho,M_\sigma)\big) \to \Hom_K\big(\Lambda^{\otimes(n+1)},\Hom_L(M_\rho,M_\sigma)\big) \]
given by
\begin{multline*}
d^n(f)(a_1\otimes\cdots\otimes a_{n+1}) := \sigma(a_1)f(a_2\otimes\cdots\otimes a_{n+1})\\
+ \smash[b]{\sum_{1\leq i\leq n}(-1)^if(a_1\otimes\cdots\otimes a_ia_{i+1}\otimes\cdots\otimes a_{n+1})}\\
+ (-1)^{n+1}f(a_1\otimes\cdots\otimes a_n)\rho(a_{n+1}).
\end{multline*}
By \cite[IX Corollary 4.3]{CE} this has cohomology
\[ HH^n\big(\Lambda,\Hom_L(M_\rho,M_\sigma)\big) \cong \Ext^n_{\Lambda^L}(M_\rho,M_\sigma). \]
Now observe that
\[ \Ker(d^1) = \Der_K\big(\Lambda,\Hom_L(M_\rho,M_\sigma)\big) \]
and
\begin{align*}
d^0 \colon \Hom_L(M_\rho,M_\sigma) &\to \Der_K\big(\Lambda,\Hom_L(M_\rho,M_\sigma)\big)\\
d^0(\gamma)(a) &= \sigma(a)\gamma-\gamma\rho(a),
\end{align*}
so $\delta_{\rho,\sigma}=-d^0$. The result follows.
\end{proof}

We make the following remark. When both representations equal $\rho\oplus\sigma$, then all four terms of the sequence in Voigt's Lemma naturally decompose, and the maps preserve the block structures, so in fact the whole sequence decomposes. For example, $\delta_{\rho\oplus\sigma}$ acts as
\[ \begin{pmatrix}\gamma_{11}&\gamma_{12}\\\gamma_{21}&\gamma_{22}\end{pmatrix} \mapsto \begin{pmatrix}\gamma_{11}\rho-\rho\gamma_{11} & \gamma_{12}\sigma-\rho\gamma_{12}\\\gamma_{21}\rho-\sigma\gamma_{21}&\gamma_{22}\sigma-\sigma\gamma_{22}\end{pmatrix}
= \begin{pmatrix}\delta_\rho(\gamma_{11}) & \delta_{\sigma,\rho}(\gamma_{12})\\\delta_{\rho,\sigma}(\gamma_{21})&\delta_\sigma(\gamma_{22})\end{pmatrix}. \]

Recall that the direct sum induces the morphism
\[ \Theta \colon \GL_{d+e}\times\rep_\Lambda^d\times\rep_\Lambda^e \longrightarrow \rep_\Lambda^{d+e}, \quad (g,\rho,\sigma) \longmapsto g\cdot(\rho\oplus\sigma), \]
whose differential at the $L$-valued point $(1,\rho,\sigma)$ is
\begin{align*}
d_{\rho,\sigma}\Theta \colon \bM_{d+e}(L)\times\big(T_\rho\rep_\Lambda^d\big)\times\big(T_\sigma\rep_\Lambda^e\big) &\longrightarrow T_{\rho\oplus\sigma}\rep_\Lambda^{d+e}\\
(\gamma,\xi,\eta) &\longmapsto (\xi\oplus\eta)+\delta_{\rho\oplus\sigma}(\gamma).
\end{align*}
Using Voigt's Lemma we deduce that, in the notation of \hyperref[Cor:irred-cpt]{Corollary \ref*{Cor:irred-cpt}},
\[ N_{X,\rho\oplus\sigma}\cong\Ext_{\Lambda^L}^1(M_\rho,M_\rho)\times\Ext_{\Lambda^L}^1(M_\sigma,M_\sigma) \]
and
\[ N_{Y,\rho\oplus\sigma}\cong\Ext_{\Lambda^L}^1(M_\rho\oplus M_\sigma,M_\rho\oplus M_\sigma). \]
Moreover, the map $\theta_{\rho,\sigma}\colon N_{X,\rho\oplus\sigma}\to N_{Y,\rho\oplus\sigma}$ induced by $d_{\rho,\sigma}\Theta$ sends $([\xi],[\eta])$ to $[\xi\oplus\eta]$, which is precisely the canonical embedding
\[ \theta_{\rho,\sigma} \colon \Ext_{\Lambda^L}^1(M_\rho,M_\rho)\times\Ext_{\Lambda^L}^1(M_\sigma,M_\sigma) \hookrightarrow \Ext_{\Lambda^L}^1(M_\rho\oplus M_\sigma,M_\rho\oplus M_\sigma). \]
It follows that condition (3) of the corollary is also satisfied, so $\Theta$ is separable on each irreducible component (with its reduced subscheme structure). Moreover, $\theta_{\rho,\sigma}$ is surjective if and only if $\Ext_{\Lambda^L}^1(M_\rho,M_\sigma)=0=\Ext_{\Lambda^L}^1(M_\sigma,M_\rho)$.

\subsubsection{Example}

If $\Lambda$ is commutative, then $\rep_\Lambda^1=\Spec \Lambda$. On the other hand, whenever $\rho\in\rep_\Lambda^1(L)$ we always get $\delta_\rho=0$, and so $T_\rho\rep_\Lambda^1=\Ext^1_{\Lambda^L}(M_\rho,M_\rho)$. Putting these together we recover the well-known result that for $\Lambda$ commutative and finitely generated and $\rho\in\Spec \Lambda(L)$, say $\rho\colon \Lambda\twoheadrightarrow \Lambda/\mathfrak p\hookrightarrow L$, we have
\[ T_\rho\Spec \Lambda = L\otimes_\Lambda\Ext^1_\Lambda(\Lambda/\mathfrak p,\Lambda/\mathfrak p). \]

\subsubsection{Upper semi-continuity}

\begin{Prop}[{\cite[Lemma 4.3]{CBS}}]\label{Prop:usc-hom-ext}
The functions $\rep_\Lambda^d\times\rep_\Lambda^e\to\mathbb N$ which send an $L$-valued point $(\rho,\sigma)$ to the dimensions of
\[ \Hom_{\Lambda^L}(M_\rho,M_\sigma), \ \Der_K(\Lambda,\Hom_L(M_\rho,M_\sigma)) \textrm{ or } \Ext^1_{\Lambda^L}(M_\rho,M_\sigma) \]
are all upper semi-continuous.
\end{Prop}

\begin{proof}
For homomorphisms we consider the closed subscheme
\[ \rep_{\Lambda(2)}^{(d,e)} \subset \rep_\Lambda^d\times\rep_\Lambda^e\times\bM_{e\times d}, \quad \rep_{\Lambda(2)}^{(d,e)}(R) := \{ (\rho,\sigma,f) : f\rho=\sigma f \}. \]
The projection $\pi\colon\rep_{\Lambda(2)}^{(d,e)}\to\rep_\Lambda^d\times\rep_\Lambda^e$ has fibres $\pi^{-1}(\rho,\sigma)\cong\Hom_{\Lambda^R}(M_\rho,M_\sigma)$. The result therefore follows from \hyperref[Cor:usc]{Corollary \ref*{Cor:usc}}.

For derivations we apply \hyperref[Cor:usc]{Corollary \ref*{Cor:usc}} to the closed subscheme $\Der(d,e)\subset\rep_\Lambda^{d+e}$ and the projection $\pi\colon \Der(d,e)\to\rep_\Lambda^d\times\rep_\Lambda^e$, having fibres $\pi^{-1}(\rho,\sigma)=\Der(\rho,\sigma)$.

Finally, for extensions, Voigt's Lemma tells us that
\[ \dim_L\Ext^1_{\Lambda^L}(M_\rho,M_\sigma) = \dim_L\Hom_{\Lambda^L}(M_\rho,M_\sigma)+\dim_L\Der(\rho,\sigma)-de, \]
and a sum of upper semi-continuous functions is again upper semi-continuous.
\end{proof}

If $K$ is algebraically-closed, and $X\subset\rep_\Lambda^d$ and $Y\subset\rep_\Lambda^e$ are irreducible, then $X\times Y$ is again irreducible. We denote by $\hom(X,Y)$ and $\ext(X,Y)$ the generic, or minimal, values on $X\times Y$ of the functions $\dim_L\Hom_{\Lambda^L}(M_\rho,M_\sigma)$ and $\dim_L\Ext^1_{\Lambda^L}(M_\rho,M_\sigma)$.

\subsection{Jet space computations}\label{Section:schemes}

Setting $D_r:=L[t]/(t^{r+1})$ as usual, we see that the points of the jet space $T_\rho^{(r)}\rep_\Lambda^d$ are given by algebra homomorphisms $\hat\rho:=\rho+\xi_1t+\cdots+\xi_rt^r\colon\Lambda\to\bM_d(D_r)$, which, by restriction of scalars, we can also regard as the subset of $\rep_\Lambda^{(r+1)d}(L)$ given by those algebra homomorphisms of the form
\begin{equation}\label{eq:r-lifting}
\tilde\rho \colon \Lambda \to \bM_{(r+1)d}(L), \quad a \mapsto
\begin{pmatrix}
\rho(a)&\xi_1(a)&\xi_2(a)&\cdots&\xi_r(a)\\0&\rho(a)&\xi_1(a)&\ddots&\vdots\\
\vdots&\ddots&\ddots&\ddots&\xi_2(a)\\\vdots&&\ddots&\rho(a)&\xi_1(a)\\0&\cdots&\cdots&0&\rho(a)
\end{pmatrix}.
\end{equation}

Consider the direct sum morphism
\[ \Theta \colon \GL_{d+e}\times\rep_\Lambda^d\times\rep_\Lambda^e \to \rep_\Lambda^{d+e} \]
and in particular the induced morphisms on jet spaces
\[ d_{\rho,\sigma}^{(r)}\Theta \colon T_1^{(r)}\GL_{d+e} \times T_\rho^{(r)}\rep_\Lambda^d \times T_\sigma^{(r)}\rep_\Lambda^e \to T_{\rho\oplus\sigma}^{(r)}\rep_\Lambda^{d+e}. \]

\begin{Prop}\label{Prop:surj-diff1}
The following are equivalent for $\rho\in\rep_\Lambda^d(L)$ and $\sigma\in\rep_\Lambda^e(L)$.
\begin{enumerate}
\item $d_{\rho,\sigma}^{(r)}\Theta$ is surjective for all $r\in[1,\infty]$.
\item $d_{\rho,\sigma}^{(r)}\Theta$ is surjective for some $r\in[1,\infty]$.
\item $\theta_{\rho,\sigma}$ is surjective.
\end{enumerate}
In particular, condition (4) of \hyperref[Cor:irred-cpt]{Corollary \ref*{Cor:irred-cpt}} holds true.
\end{Prop}

\begin{proof}
$(1)\Rightarrow(2)\colon$ Trivial.

$(2)\Rightarrow(3)\colon$ Assume $\Ext_{\Lambda^L}^1(M_\rho,M_\sigma)\neq0$ and take some $\xi\in\Der(\rho,\sigma)$ representing a non-split extension. Then for any $r\in[1,\infty]$ the map
\[ \begin{pmatrix}\rho&0\\0&\sigma\end{pmatrix} + \begin{pmatrix}0&0\\\xi&0\end{pmatrix}t \colon \Lambda \to \bM_{d+e}(D_r) \]
is an algebra homomorphism, so corresponds to a point in $T_{\rho\oplus\sigma}^{(r)}\rep_\Lambda^{d+e}$. On the other hand, it cannot lie in the image of the morphism on jet spaces. For, if it did, then $\xi$ would necessarily lie in the image of the differential, and so the extension class $[\xi]$ would have to lie in the image of $\theta_{\rho,\sigma}$, which it does not.

$(3)\Rightarrow(1)\colon$ Suppose $\Ext_{\Lambda^L}^1(M_\rho,M_\sigma)=0=\Ext_{\Lambda^L}^1(M_\sigma,M_\rho)$. Given $r\in[1,\infty]$ take an element of $T_{\rho\oplus\sigma}^{(r)}\rep_\Lambda^{d+e}$, say
\[ \widehat{\rho\oplus\sigma} := \begin{pmatrix}\rho&0\\0&\sigma\end{pmatrix} + \sum_{i=1}^r\begin{pmatrix}\xi_i&y_i\\x_i&\eta_i\end{pmatrix}t^i \in \rep_\Lambda^{d+e}(D_r). \]

Let $s\geq1$ be minimal such that $x_i,y_i$ are not both zero. Restriction of scalars gives $\widetilde{\rho\oplus\sigma}\in\rep_{\Lambda}^{(s+1)(d+e)}(L)$ satisfying
\[ \widetilde{\rho\oplus\sigma} = \begin{pmatrix}A_0&A_1&\cdots&A_s\\0&\ddots&\ddots&\vdots\\\vdots&\ddots&\ddots&A_1\\0&\cdots&0&A_0\end{pmatrix}, \]
where
\[ A_0 := \begin{pmatrix}\rho&0\\0&\sigma\end{pmatrix} \quad\textrm{and}\quad A_i := \begin{pmatrix}\xi_i&y_i\\x_i&\eta_i\end{pmatrix}. \]
Let $N$ be the corresponding module. Next define submodules $U\leq V\leq N$, where $U$ is given by taking the odd-numbered rows and columns up to $2s-1$, whereas for $V$ we take additionally the second and $(2s+1)$-st rows and columns, but reordered so that these appear last. Note that $U$ and $V$ correspond respectively to the representations
\[ U \leftrightarrow \begin{pmatrix}
\rho&\xi_1&\cdots&\xi_{s-1}\\
&\ddots&\ddots&\vdots\\
&&\ddots&\xi_1\\
&&&\rho
\end{pmatrix}
\quad\textrm{and}\quad
V\leftrightarrow\begin{pmatrix}
\rho&\xi_1&\cdots&\xi_{s-1}&0&\xi_s\\
&\ddots&\ddots&\vdots&\vdots&\vdots\\
&&\ddots&\xi_1&0&\xi_2\\
&&&\rho&0&\xi_1\\
&&&&\sigma&x_s\\
&&&&&\rho
\end{pmatrix} \]
Clearly $U\hookrightarrow V$ with quotient $V/U$ corresponding to the representation $\begin{pmatrix}\sigma&x_s\\0&\rho\end{pmatrix}$, so the natural extension $0\to M_\sigma\to V/U\to M_\rho\to 0$ lies in $\Ext_{\Lambda^L}^1(M_\rho,M_\sigma)$. Since this is zero we know from \hyperref[Lem:Voigt]{Lemma \ref*{Lem:Voigt}} that $x_s=\gamma\rho-\sigma\gamma$ for some $\gamma\in\bM_{e\times d}(L)$.

An analogous argument yields $\delta\in\bM_{d\times e}(L)$ such that $y_s=\delta\sigma-\rho\delta$. Finally, conjugating by $1-\begin{pmatrix}0&\delta\\\gamma&0\end{pmatrix}t^s\in\GL_{d+e}(D_r)$ yields another representation in $\rep_\Lambda^{d+e}(D_r)$, but now with $x_i=0=y_i$ for all $i\leq s$.

So, starting from $\widehat{\rho\oplus\sigma}$, we can define inductively $g_s=1-\begin{pmatrix}0&\delta_s\\\gamma_s&0\end{pmatrix}t^s$ such that conjugation by $g_s\cdots g_2g_1$ yields an element of $\rep_\Lambda^{d+e}(D_r)$ having zero in the off-diagonal entries of the coeffiecient of $t^i$ for all $i\leq s$. Now set $g$ to be the product of all the $g_s$. Note that this also makes sense for $r=\infty$, so $g\in\GL_{d+e}(D_r)$, and by construction $g\cdot\widehat{\rho\oplus\sigma}$ has zero in the off-diagonal entries of the coefficient of all $t^i$. This element clearly lies in the image of $d^{(r)}_{\rho,\sigma}\Theta$ as required.
\end{proof}

\subsubsection{Relationship to Massey products}

We note that the condition on algebra homomorphisms given by \eqref{eq:r-lifting} is reminiscent of the condition that the $r$-fold Massey product contains zero, which we now recall.

Let $M_i$ be a family of $\Lambda^L$-modules and $\eta_i\in\Ext^1_{\Lambda^L}(M_{i+1},M_i)$ a family of extensions. Choose representations $\rho_i\colon\Lambda\to\End_L(M_i)$ corresponding to $M_i$ and derivations $\xi_i\in\Der(\rho_{i+1},\rho_i)$ corresponding to $\eta_i$, so $\begin{pmatrix}\rho_i&\xi_i\\0&\rho_{i+1}\end{pmatrix}$ is a representation. We say that the $r$-fold Massey product $\langle\eta_1,\ldots,\eta_r\rangle$ contains zero if and only if there is a representation of the form
\[ a \mapsto \begin{pmatrix}\rho_1(a)&\xi_1(a)&\star&\cdots&\star\\0&\rho_2(a)&\xi_2(a)&\ddots&\vdots\\
\vdots&\ddots&\ddots&\ddots&\star\\\vdots&&\ddots&\rho_r(a)&\xi_r(a)\\
0&\cdots&\cdots&0&\rho_{r+1}(a)\end{pmatrix}. \]

As a special case let $\eta=[\xi]$ be the extension class coming from $\xi\in T_\rho\rep_\Lambda^d$. Then $\xi\in T_\rho^{(r)}\rep_\Lambda^d$ implies that there exists an algebra homomorphism as in \eqref{eq:r-lifting}, and so the $r$-fold Massey product $\langle\eta,\cdots,\eta\rangle$ contains zero.

In \cite[IV \S2 Exercise 3 (f)]{GM} it is falsely claimed that the tangent cone of $\Spec A$ at a point $\rho\in\Spec A(L)$ is given by those extensions $\eta\in\Ext_{A^L}^1(M_\rho,M_\rho)$ such that, for all $r$, the $r$-fold Massey product $\langle\eta,\cdots,\eta\rangle$ contains zero.

If $\rho\colon A\to K$ has kernel $\mathfrak m$, then
\[ T_\rho\Spec A = \Hom_K(\mathfrak m/\mathfrak m^2,K) \cong \Ext^1_A(A/\mathfrak m,A/\mathfrak m) \]
whereas the tangent cone $TC_\rho\Spec A$ is by definition given by the $K$-valued points of $\Spec\big(\bigoplus_r\mathfrak m^r/\mathfrak m^{r+1}\big)$, viewed as a quotient of the ring of functions on $T_\rho\Spec A$. The claim would be that the tangent cone equals
\[ \{\eta\in\Ext^1_A(A/\mathfrak m,A/\mathfrak m) : 0\in\langle\underbrace{\eta,\ldots,\eta}_{r}\rangle\textrm{ for all }r\}. \]
This is easily seen to fail, for example by considering the cuspidal cubic at the origin, so
\[ A = K[X,Y]/(X^3-Y^2), \quad \rho \colon A \to K, \ \rho(X)=\rho(Y)=0, \  \mathfrak m = (X,Y). \]
Then $\bigoplus_r\mathfrak m^r/\mathfrak m^{r+1}\cong K[X,Y]/(Y^2)$, so
\[ T_\rho\Spec A = K^2 \quad\textrm{and}\quad TC_\rho\Spec A = \{(x,0):x\in K\}. \]
A point $\xi=(x,y)\in T_\rho\Spec A$ corresponds to the class $\eta$ of the extension with middle term
\[ A \to \bM_2(K), \quad X \mapsto \begin{pmatrix}0&x\\0&0\end{pmatrix}, \quad Y \mapsto \begin{pmatrix}0&y\\0&0\end{pmatrix}. \]
We can lift this to a four-dimensional module
\[ A \to \bM_4(K), \quad X \mapsto \begin{pmatrix}0&x&x_2&x_4\\0&0&x&x_3\\0&0&0&x\\0&0&0&0\end{pmatrix}, \quad
Y \mapsto \begin{pmatrix}0&y&y_2&y_4\\0&0&y&y_3\\0&0&0&y\\0&0&0&0\end{pmatrix} \]
if and only if $x=y=0$ (just compute the image of $X^3-Y^2$). We therefore see that $0\in\langle\eta,\eta,\eta\rangle$ if and only if $\xi=(0,0)$.

Keeping with the same algebra $A=K[X,Y]/(X^3-Y^2)$, consider the two-dimensional representation
\[ \rho\colon A\to\bM_2(K), \quad X \mapsto \begin{pmatrix}0&1\\0&0\end{pmatrix}, \quad Y \mapsto \begin{pmatrix}0&0\\0&0\end{pmatrix}. \]
Then $T_\rho\rep_A^2$ consists of pairs of matrices $\left(\begin{pmatrix}x_1&x_2\\0&x_4\end{pmatrix},\begin{pmatrix}y_1&y_2\\y_3&y_4\end{pmatrix}\right)$ and the map $\delta_\rho$ from Voigt's Lemma is given by
\[ \delta_\rho \colon \bM_2(K) \to T_\rho\rep_A^2, \quad \begin{pmatrix}a&b\\c&d\end{pmatrix}\mapsto\left(\begin{pmatrix}-c&a-d\\0&c\end{pmatrix},\begin{pmatrix}0&0\\0&0\end{pmatrix}\right). \]
We can also compute
\[ \overline T_\rho^{(2)}\rep_A^2 = \left\{\left(\begin{pmatrix}x_1&x_2\\0&x_4\end{pmatrix},\begin{pmatrix}y_1&y_2\\y_3&-y_1\end{pmatrix}\right):y_1^2+y_2y_3=0\right\}, \]
whereas $\overline T_\rho^{(3)}\rep_A^2$ is the union of
\[ \left\{\left(\begin{pmatrix}x_1&x_2\\0&x_4\end{pmatrix},\begin{pmatrix}y_1&y_2\\y_3&y_4\end{pmatrix}\right):y_1^2+y_2y_3=0,\ (y_2,y_3)\neq(0,0)\right\} \]
and
\[\left\{\left(\begin{pmatrix}x_1&x_2\\0&-x_1\end{pmatrix},\begin{pmatrix}0&0\\0&0\end{pmatrix}\right)\right\}. \]

\subsection{Irreducible components}

The next two theorems were proved by Crawley-Boevey and Schr\"oer using slightly different methods, and generalise Kac's `canonical decomposition' for representations of quivers to all finitely-generated algebras. Note that the second result appears as \cite[Lemma (1.3)]{delaPena}. We remark that the restriction to algebraically-closed fields is just to ensure that the direct product of two irreducible schemes is again irreducible.

\begin{Thm}[\cite{CBS}]\label{Thm:CBS1}
Let $\Lambda$ be a finitely-generated algebra over an algebraically-closed field $K$. Let $X\subset\rep_\Lambda^d$ and $Y\subset\rep_\Lambda^e$ be irreducible components. Then the closure $\overline{X\oplus Y}$ of the image of $\GL_{d+e}\times X\times Y$ is an irreducible component of $\rep_\Lambda^{d+e}$ if and only if $\ext(X,Y)=0=\ext(Y,X)$.
\end{Thm}

\begin{proof}
We have shown that the direct sum map
\[ \Theta \colon \GL_{d+e}\times\rep_\Lambda^d\times\rep_\Lambda^e \to \rep_\Lambda^{d+e} \]
satisfies the conditions of \hyperref[Cor:irred-cpt]{Corollary \ref*{Cor:irred-cpt}}, and also that $\theta_{\rho,\sigma}$ is surjective if and only if $\Ext_{\Lambda^L}^1(M_\rho,M_\sigma)=0=\Ext_{\Lambda^L}^1(M_\sigma,M_\rho)$. Thus $\overline{X\oplus Y}$ is an irreducible component if and only if $\theta_{\rho,\sigma}$ is surjective on an open dense subset of $X\times Y$, which is if and only if $\ext(X,Y)=0=\ext(Y,X)$.
\end{proof}

We call an irreducible component $X\subset\rep_\Lambda^d$ generically indecomposable provided that $X$ contains a dense open subset, all of whose points correspond to indecomposable $\Lambda$-modules.

\begin{Thm}[\cite{delaPena,CBS}]\label{Thm:CBS2}
Every irreducible component $X\subset\rep_\Lambda^d$ can be written uniquely (up to reordering) as a direct sum $X=\overline{X_1\oplus\cdots\oplus X_n}$ of generically indecomposable components $X_i\subset\rep_\Lambda^{d_i}$.
\end{Thm}

\begin{proof}
If $X$ is not generically indecomposable, then $X$ lies in the union of the closed sets $\overline{\rep_\Lambda^e\oplus\rep_\Lambda^{d-e}}$ for $0<e<d$, and hence for some $e$ and some irreducible components $X_1\subset\rep_\Lambda^e$ and $X_2\subset\rep_\Lambda^{d-e}$ we can write $X=\overline{X_1\oplus X_2}$. By induction on dimension we see that every irreducible component can be written as a direct sum of generically indecomposable components.

Suppose now that $X=\overline{X_1\oplus\cdots\oplus X_m}=\overline{Y_1\oplus\cdots\oplus Y_n}$, where $X_i$ and $Y_j$ are generically indecomposable components in some schemes of representations. Then on an open dense subset $U\subset X$ every representation can be written as $\rho\cong\rho_1\oplus\cdots\oplus\rho_m\cong\sigma_1\oplus\cdots\oplus\sigma_n$, where $\rho_i\in X_i$ and $\sigma_j\in Y_j$ are indecomposable representations. By the Krull-Remak-Schmidt Theorem we deduce that $m=n$ and, after reordering, $\rho_i\cong\sigma_i$. Thus $X_i,Y_i\subset\rep_\Lambda^{d^i}$ and $\rho_i\in X_i\cap Y_i$. On the other hand, the projection $\GL_d\times X_1\times\cdots\times X_n\to X_i$ is an open map, so $X_i\cap Y_i$ contains an open set. Thus $X_i=Y_i$.
\end{proof}

\subsubsection{Orbits and irreducible components}

We know that if $\rho\in\rep_\Lambda^d(K)$, then its orbit $\Orb_{\GL_d}(\rho)$ is a smooth, irreducible subscheme of $\rep_\Lambda^d$; also, the morphism $\GL_d\to\Orb_{\GL_d}(\rho)$ is smooth, affine and separable, and a universal geometric quotient for the action of $\Aut_\Lambda(M_\rho)$.

In general it is not easy to describe the $R$-valued points of the orbit: \textit{a priori} we know that $\sigma\in\rep_\Lambda^d(R)$ lies in the orbit if and only if there exists a finitely-presented and faithfully-flat $R$-algebra $S$ and $g\in\GL_d(S)$ such that $\sigma=g\cdot\rho^S$. In particular, $\GL_d(L)$ acts transitively on $\Orb_{\GL_d}(\rho)(L)$ for all algebraically closed fields $L$.

In fact we can do better.

\begin{Lem}\label{Lem:orbit-points}
Let $\rho\in\rep_\Lambda^d(K)$. If $D_r=L[[t]]/(t^{r+1})$ for some field $L$ and some $r\in[0,\infty]$, then the morphism $\GL_d(D_r)\to\Orb_{\GL_d}(\rho)(D_r)$ is onto.
\end{Lem}

\begin{proof}
Consider first the case $r=0$, so $D_r=L$. Let $\sigma\in\Orb_{\GL_d}(\rho)(L)$ and let $F=\bar L$ be the algebraic closure of $L$. We know that $\GL_d(F)$ acts transitively on $\Orb_{\GL_d}(\rho)(F)$, so $M_\sigma^F\cong M_\rho^F$ as $\Lambda^F$-modules. Now let $N$ be any $\Lambda^L$-module. Then $\Hom_{\Lambda^F}(M_\sigma^F,N^F)\cong\Hom_{\Lambda^L}(M_\sigma,N)^F$, so that
\[ \dim_L\Hom_{\Lambda^L}(M_\sigma,N) = \dim_L\Hom_{\Lambda^L}(M_\rho^L,N). \]
It follows from \cite{AR} that $M_\sigma\cong M_\rho^L$ as $\Lambda^L$-modules, so there exists $g\in\GL_d(L)$ such that $\sigma=g\cdot\rho^L$.

Next, by \hyperref[Prop:sep]{Proposition \ref*{Prop:sep}}, the morphism $T_g^{(r)}\GL_d\to T_{g\cdot\rho}^{(r)}\Orb_{\GL_d}(\rho)$ is surjective for all $r\in[1,\infty]$ and all $g$ in an open dense subset of $\GL_d$. Since $\GL_d$ is a group we see that this holds for all $g$, and we have just shown that every $\sigma\in\Orb_{\GL_d}(\rho)(L)$ is of the form $g\cdot\rho^L$ for some $g\in\GL_d(L)$. This proves that the morphism $\GL_d(D_r)\to\Orb_{\GL_d}(\rho)(D_r)$ is always onto.
\end{proof}

Using \hyperref[Lem:sep-non-sing]{Lemma \ref*{Lem:sep-non-sing}} we recover the well-known sufficient criterion for an orbit closure to be an irreducible component.

\begin{Lem}\label{Lem:orbit-irred-cpt}
Let $\rho\in\rep_\Lambda^d(K)$. If $\Ext_\Lambda^1(M_\rho,M_\rho)=0$, then $\overline{\Orb_{\GL_d}(\rho)}$ is an irreducible component of $\rep_\Lambda^d$.
\end{Lem}

\begin{proof}
Using Voigt's Lemma we see that our condition on $\rho$ implies that the differential $T_1\GL_d\to T_\rho\rep_\Lambda^d$ is surjective, whence $T_g\GL_d\to T_{g\cdot\rho}\rep_\Lambda^d$ is surjective for all $g$, and so $\overline{\Orb_{\GL_d}(\rho)}$ is an irreducible component of $\rep_\Lambda^d$ by \hyperref[Lem:sep-non-sing]{Lemma \ref*{Lem:sep-non-sing}}.
\end{proof}

\subsection{A refinement to dimension vectors}\label{Section:dim-vectors}

We begin with some notation which will also be useful in the next section. Recall that there is a closed subscheme $\mathrm{rank}_{<s}\subset\bM_{e\times d}$ given by
\[ \mathrm{rank}_{<s}(R) := \{ f\in\bM_{e\times d}(R) : \textrm{all $s$ minors vanish} \}. \]
Its complement $\mathrm{rank}_{\geq s}$ is the union of the distinguished open subschemes given by inverting the $s$ minors, so a matrix $f\in\bM_{e\times d}(R)$ lies in $\mathrm{rank}_{\geq s}(R)$ precisely when the $s$ minors of $f$ generate the unit ideal in $R$. Note that if $R$ is a local ring, then we can simplify this to
\[ \mathrm{rank}_{\geq s}(R) := \{ f\in\bM_{e\times d}(R) : \textrm{at least one $s$ minor of $f$ is invertible} \}. \]
The scheme of matrices of rank precisely $s$ is the subscheme
\[ \mathrm{rank}_s := \mathrm{rank}_{<s+1}\cap\mathrm{rank}_{\geq s} \subset \bM_{e\times d}. \]
We have the following useful lemma.

\begin{Lem}[{\cite[Lemma 3.3]{Zwara}}]\label{Lem:rank}
Let $(R,\mathfrak m)$ be a local $L$-algebra of dimension $r$. Let $\hat f\in\bM_{e\times d}(R)$, let $f\in\bM_{e\times d}(R/\mathfrak m)$ be its image under $R\to R/\mathfrak m$, and let $\tilde f\in\bM_{er\times dr}(L)$ be given by restriction of scalars. Then $\hat f\in\mathrm{rank}_s(R)$ if and only if $f\in\mathrm{rank}_s(R/\mathfrak m)$ and $\tilde f\in\mathrm{rank}_{sr}(L)$.
\end{Lem}

Consider now a complete set of orthogonal idempotents $e^i\in\Lambda$, so $e^ie^j=\delta_{ij}e^i$ and $\sum_ie^i=1$. Given a decomposition, or dimension vector, $d\updot=(d^1,\ldots,d^n)$ and setting $d=\sum_id^i$, we can regard matrices in $\bM_d(R)$ as being in block form according to this dimension vector. Write $E_i$ for the block diagonal matrix having the identity $1_{d^i}$ in block $(i,i)$ and zeros elsewhere. We define the closed subscheme $\rep_\Lambda^{d\updot}$ of $\rep_\Lambda^d$ by
\[ \rep_\Lambda^{d\updot}(R) := \{ \rho\in\rep_\Lambda^d(R) : \rho(e^i)=E_i \textrm{ for all }i \}, \]
the locally-closed subscheme $Y_\Lambda^{d\updot}\subset\rep_\Lambda^d$ by
\[ Y_\Lambda^{d\updot}(R) := \{ \rho\in\rep_\Lambda^d : \rho(e^i)\in\mathrm{rank}_{d^i}(R) \textrm{ for all }i \}, \]
and the closed subgroup $\GL_{d\updot}:=\prod_i\GL_{d^i}\leq\GL_d$ by taking the block-diagonal matrices. Observe that the action of $\GL_d$ on $\rep_\Lambda^d$ restricts to an action of $\GL_{d\updot}$ on $\rep_\Lambda^{d\updot}$.

The following is a slight generalisation of results in \cite{Bongartz,Gabriel}.

\begin{Thm}\label{Thm:direct-product}
\begin{enumerate}
\item The $Y_\Lambda^{d\updot}$ are both open and closed subschemes of $\rep_\Lambda^d$, and $\rep_\Lambda^d=\coprod_{d\updot}Y_\Lambda^{d\updot}$ is the disjoint union over all decompositions $d\updot$ of $d$.
\item If $\Lambda\cong K^n$ and the $e^i$ are the standard basis, then $Y_\Lambda^{d\updot}\cong\GL_d/\GL_{d\updot}$ is a homogeneous space, so is smooth and irreducible.
\item More generally, $Y_\Lambda^{d\updot}\cong\GL_d\times^{\GL_{d\updot}}\rep_\Lambda^{d\updot}$.
\end{enumerate}
\end{Thm}

Given a $\Lambda^L$-module $M$, we may therefore define the dimension vector of $M$ to be $\underline\dim\,M=d\updot$, where $M\cong M_\rho$ for some $\rho\in\rep_\Lambda^{d\updot}(L)$.

The proof of the theorem uses the following lemma, which although easy to prove, does not seem to be well-known. Given an ordered partition $I=(I_1,\ldots,I_r)$ of $\{1,\ldots,d\}$, we say that $\sigma\in S_d$ is an $I$-shuffle provided it preserves the ordering of the elements in each $I_i$. Also, given two subsets $I,J$ of $\{1,\ldots,d\}$ of size $r$ and a matrix $M\in\bM_d(R)$, we define $\Delta_{IJ}(M)$ to be the $r$-minor formed from the rows and columns indexed by $I$ and $J$, respectively.

\begin{Lem}\label{Lem:det-sum}
For matrices $M_1,\ldots,M_n\in\bM_d(R)$ we have
\[ \det(M_1+\cdots+M_n) = \sum_{I,\sigma}\mathrm{sgn}(\sigma)\Delta_{I_1\sigma(I_1)}(M_1)\cdots\Delta_{I_n\sigma(I_n)}(M_n), \]
where the sum is taken over all ordered partitions $I=(I_1,\ldots,I_n)$ of $\{1,\ldots,d\}$ and all $I$-shuffles $\sigma$.

In particular, when $n=2$ we can write this as
\[ \det(M+N) = \sum_{\substack{I=\{i_1,\ldots,i_a\}\\J=\{j_1,\ldots,j_a\}}}(-1)^{i_1+j_1+\cdots+i_a+j_a}\Delta_{IJ}(M)\Delta_{I^cJ^c}(N), \]
where $I^c$ denotes the complement of $I$.
\end{Lem}

\begin{proof}
The first formula follows easily by induction, once we have proved the case for $n=2$. Consider now the second formula. Since both sides are polynomial functions and $\bM_d$ is irreducible, it is enough to prove this when $N$ is invertible. In this case we have
\[ \det(M+N) = \det(MN^{-1}+1_d)\det(N) = \sum_{I}\Delta_{II}(MN^{-1})\det(N). \]
For, let $I_r$ be the matrix having 1s in the first $r$ places on the diagonal and zeros elsewhere. Then an easy induction, expanding out the $r$-the column, gives that $\det(A+1_r)=\sum_I\Delta_{II}(A)$, where the sum is taken over subsets $I$ containing $\{r+1,\ldots,d\}$.

Using the Cauchy-Binet Formula (see for example \cite[Equation(16)]{BS}) we can write this as
\[ \det(M+N) = \sum_a\sum_{|I|=|J|=a}\Delta_{IJ}(M)\Delta_{JI}(N^{-1})\det(N). \]
Finally, using Jacobi's Identity \cite[Equation (12)]{BS} we have
\[ \det(M+N) = \sum_a\sum_{|I|=|J|=a}(-1)^{S(I)+S(J)}\Delta_{IJ}(M)\Delta_{I^cJ^c}(N), \]
where $S(I)=\sum_{i\in I}i$. This proves the second formula.

Now let $I=\{i_1,\ldots,i_a\}$ and $I^c=\{i'_1,\ldots,i'_{d-a}\}$, assumed to be in increasing order, and similarly for $J$. Let $\sigma$ be the corresponding shuffle, so that $\sigma(i_t)=j_t$ and $\sigma(i'_t)=j'_t$. Clearly $\sigma=\sigma_J\sigma_I^{-1}$, where $\sigma_I(t)=i_t$ for $1\leq t\leq a$ and $\sigma_I(a+t)=i'_t$ for $1\leq t\leq d-a$. Using the formula
\[ \mathrm{sgn}(\sigma) = (-1)^{\mathrm{inv}(\sigma)}, \quad \mathrm{inv}(\sigma) := |\{(i,j):i<j, \ \sigma(i)>\sigma(j)\}| \]
we see that
\[ \mathrm{inv}(\sigma_I) = |\{(i,i')\in I\times I^c : i'<i\}| = S(I)-\binom{|I|+1}{2}, \]
and hence that
\[ \mathrm{sgn}(\sigma) = \mathrm{sgn}(\sigma_I)\mathrm{sgn}(\sigma_J) = (-1)^{S(I)+S(J)}. \]
Thus the two formulae agree in the case $n=2$.
\end{proof}

We observe that writing a matrix $M=(m_{ij})$ as the sum $M=\sum_{i,j}m_{ij}E_{ij}$, where the $E_{ij}$ are the usual elementary matrices, one recovers the well-known Leibniz Formula
\[ \det(M) = \sum_{\sigma\in S_d}\mathrm{sgn}(\sigma)m_{1\sigma(1)}m_{2\sigma(2)}\cdots m_{d\sigma(d)}. \]

\begin{Prop}
The open subscheme
\[ Y_\Lambda^{\geq d\updot}(R) := \{\rho\in\rep_\Lambda^d(R) : \rho(e^i)\in\mathrm{rank}_{\geq d^i}(R) \textrm{ for all }i\} \]
and the closed subscheme
\[ Y_\Lambda^{\leq d\updot}(R) := \{\rho\in\rep_\Lambda^d(R) : \rho(e^i)\in\mathrm{rank}_{\leq d^i}(R) \textrm{ for all }i\}, \]
coincide, and both are equal to $Y_\Lambda^{d\updot}$. Moreover, $\rep_\Lambda^d=\coprod_{d\updot}Y_\Lambda^{d\updot}$.
\end{Prop}

\begin{proof}
Since the $Y_\Lambda^{\geq d\updot}$ are open subschemes, by looking at $L$-valued points for fields $L$ we can prove that the $Y_\Lambda^{\geq d\updot}$ are disjoint, that they cover $\rep_\Lambda^d$, and that $Y_\Lambda^{\leq d\updot}\subset Y_\Lambda^{\geq d\updot}$ is a closed subscheme. To prove that we have equality, take $\rho\in Y_\Lambda^{\leq d\updot}(R)$. Then $\rho(e^i)=f_i\in\mathrm{rank}_{\leq d^i}(R)$ and $\sum_if_i=1$. Applying the previous lemma to $\det(f_1+\cdots+f_n)$ shows that we can write 1 as a linear combination of terms of the form $\prod_i\Delta_{I_i\sigma(I_i)}(f_i)$ where $|I_i|=d^i$. It follows that for each $i$, the $d^i$ minors of $f_i$ generate the unit ideal, so $\rho\in Y_\Lambda^{\geq d\updot}(R)$.
\end{proof}

\begin{proof}[Proof of Theorem \ref*{Thm:direct-product}]
(1) This follows immediately from the proposition above.

(2) We have $\Lambda=K^n$, so $\rep_\Lambda^{d\updot}\cong\Spec K$. This consists of the unique representation $\bar\rho$ such that $\bar\rho(e^i)=E_i$ for all $i$. We also have the corresponding orbit map $\pi\colon\GL_d\to Y_\Lambda^{d\updot}$, $g\mapsto g\cdot\bar\rho$.

Let $L$ be an algebraically-closed field and $\sigma\in Y_\Lambda^{d\updot}(L)$. The matrices $\sigma(e^i)$ form a complete set of orthogonal idempotents; in particular they commute, so are simultaneously diagonalisable. It follows that there exists $g\in\GL_d(L)$ such that $g\sigma(e^i)g^{-1}=E_i$, so $g\cdot\sigma=\bar\rho^L$ and hence $Y_\Lambda^{d\updot}(L)=\Orb_{\GL_d}(\bar\rho)(L)$. Since the orbit is smooth and irreducible, we deduce that $\Orb_{\GL_d}(\bar\rho)=\big(Y_\Lambda^{d\updot}\big)_\red$.

It is easy to see that if $\sigma\in Y_\Lambda^{d\updot}(L)$, then $\End_{L^n}(M_\sigma)\cong\prod_i\bM_{d_i}(L)$. On the other hand, $\Lambda=K^n$ is a semisimple algebra, so all extension groups vanish. Thus by Voigt's Lemma
\[ \dim T_\sigma Y_\Lambda^{d\updot} = d^2-\sum_id_i^2 = \dim\GL_d-\dim\GL_{d\updot} = \dim\Orb_{\GL_d}(\bar\rho). \]
This proves that $Y_\Lambda^{d\updot}$ is smooth, so equals $\Orb_{\GL_d}(\bar\rho)\cong\GL_d/\GL_{d\updot}$.

(3) The idempotents $e^i$ induce an algebra homomorphism $K^n\to\Lambda$, and hence a morphism $\rep_\Lambda^d\to\rep_{K^n}^d$. In terms of representations, this is just restriction of scalars. This restricts to a morphism $q\colon Y_\Lambda^{d\updot}\to Y_{K^n}^{d\updot}\cong\GL_d/\GL_{d\updot}$. In particular, for $\rho\in\rep_\Lambda^d(R)$ we have $\rho\in\rep_\Lambda^{d\updot}(R)$ if and only if $q(\rho)=\bar\rho^R$. We therefore obtain a commutative diagram
\[ \begin{CD}
\GL_d\times\rep_\Lambda^{d\updot} @>{\phi}>> Y_\Lambda^{d\updot}\\
@VV{\mathrm{pr}_1}V @VV{q}V\\
\GL_d @>{\pi}>> Y_{K^n}^{d\updot}
\end{CD} \]
where $\phi(g,\rho):=g\cdot\rho$ is constant on $\GL_{d\updot}$-orbits and $\pi(g)=g\cdot\bar\rho$. It is easy to check that this is a pull-back diagram, and hence $Y_\Lambda^{d\updot}\cong\GL_d\times^{\GL_{d\updot}}\rep_\Lambda^{d\updot}$ by \hyperref[Lem:assoc-fib]{Lemma \ref*{Lem:assoc-fib}}.
\end{proof}

\subsubsection{Examples}

Well-known examples of this situation are when we take $\Lambda=KQ/I$ to be a quotient of the path algebra of a quiver, and take $e^i$ to be the (images of the) trivial paths in $Q$, so giving a complete set of primitive idempotents in $\Lambda$. We can also apply this to a tensor algebra $\Lambda\otimes_KKQ$, where we now use the idempotents $e^i$ in $KQ$ to get idempotents $1\otimes e^i$ in $\Lambda\otimes_KKQ$.

Since we will also need this later, we introduce the quiver $Q_n$ to be the linearly oriented quiver of Dynkin type $\mathbb A_n$, so that $Q_n\colon 1\to 2\to\cdots\to n$ and the path algebra $KQ_n$ is isomorphic to the subalgebra of upper-triangular matrices inside $\bM_n(K)$. We define $\Lambda(n)$ to be the tensor algebra $\Lambda\otimes_KKQ_n$, so isomorphic to the subalgebra of upper-triangular matrices inside $\bM_n(\Lambda)$, and $\Gamma\cong\Lambda^n$ is the subalgebra of diagonal matrices.

Observe that a $\Lambda(n)$-module can be thought of as a sequence $U^1\xrightarrow{f^1} U^2\xrightarrow{f^2}\cdots\xrightarrow{f^{n-1}}U^n$ in the category of $\Lambda$-modules, which we can write as a pair $(U\updot,f\updot)$. Thus fixing a dimension vector $d\updot$ corresponds to fixing the dimensions $\dim U^i=d^i$ and we have a closed immersion
\[ \rep_{\Lambda(n)}^{d\updot} \subset \prod_i\rep_\Lambda^{d^i}\times\prod_i\bM_{d^{i+1}\times d^i}, \]
where
\[ \rep_{\Lambda(n)}^{d\updot}(R) = \{(\rho\updot,f\updot):f^i\rho^i=\rho^{i+1}f^i\textrm{ for all }i\}. \]
Note that the projection
\[ \pi \colon \rep_{\Lambda(n)}^{d\updot} \to \prod_i\rep_\Lambda^{d^i} \]
has fibres $\pi^{-1}(\rho\updot)\cong\prod_i\Hom_{\Lambda^R}(M_{\rho^i},M_{\rho^{i+1}})$.

We remark that both \hyperref[Thm:CBS1]{Theorem \ref*{Thm:CBS1}} and \hyperref[Thm:CBS2]{Theorem \ref*{Thm:CBS2}} can be refined to the case when we consider dimension vectors (with respect to some complete set of orthogonal idempotents), as can \hyperref[Lem:orbit-irred-cpt]{Lemma \ref*{Lem:orbit-irred-cpt}}.

In fact, since $\rep_\Lambda^d=\coprod_{d\updot}Y_\Lambda^{d\updot}$, the set of irreducible components of $\rep_\Lambda^d$ is the union of the sets of irreducible components of the $Y_\Lambda^{d\updot}$. On the other hand, consider the morphism $\GL_d\times\rep_\Lambda^{d\updot}\to Y_\Lambda^{d\updot}$. All fibres are isomorphic to $\GL_{d\updot}$, so are irreducible of the same dimension. It follows that the preimage of an irreducible subset of $Y_\Lambda^{d\updot}$ is again irreducible (c.f. \cite[Theorem 11.14]{Harris}). We obtain in this manner a bijection between the irreducible $\GL_d$-invariant subsets of $Y_\Lambda^{d\updot}$ and the irreducible $\GL_{d\updot}$-invariant subsets of $\rep_\Lambda^{d\updot}$. Since each irreducible component is invariant under the group action, this restricts to a bijection between the irreducible components of $Y_\Lambda^{d\updot}$ and those of $\rep_\Lambda^{d\updot}$.

\section{Subschemes determined by homomorphisms}\label{Sec:rep-schemes-homs}

We can generalise the previous result to subschemes parameterising those representations having a fixed number of homomorphisms to a predetermined module. Examples of this situation will be given at the end of the section.

We keep the same assumptions that $K$ is a perfect field and $\Lambda:=K\langle x_1,\ldots,x_N\rangle/I$ is a finitely-generated algebra.

\subsection{Schemes of modules determined by homomorphisms}

Fix a representation $\tau\in\rep_\Lambda^m(K)$ and write $M=M_\tau$. For any $K$-algebra $R$ we thus obtain the representation $\tau^R\in\rep_\Lambda^m(R)$ and the corresponding $\Lambda^R$-module $M^R=M\otimes_KR$. Recall that
\[ \rep_{\Lambda(2)}^{(d,m)}(R) = \{(\rho,\sigma,\phi)\in\rep_\Lambda^d(R)\times\rep_\Lambda^m(R)\times\bM_{m\times d}(R) : \phi\rho=\sigma\phi\}. \] 
Thus the construction
\begin{align*}
\rep_{\Lambda(2)}^{(d,\tau)}(R) &= \{ (\rho,\sigma,\phi)\in\rep_{\Lambda(2)}^{(d,m)}(R):\sigma=\tau^R \}\\
&\cong \{ (\rho,\phi)\in\rep_\Lambda^d(R)\times\bM_{m\times d}(R) : \phi\rho=\tau^R\phi \}
\end{align*}
defines a closed subscheme of $\rep_{\Lambda(2)}^{(d,m)}$, and the projection $\pi\colon\rep_{\Lambda(2)}^{(d,\tau)}\to\rep_\Lambda^d$ is of finite type. Moreover, for $\rho\in\rep_\Lambda^d(R)$ we have $\pi^{-1}(\rho)\cong\Hom_{\Lambda^R}(M_\rho,M^R)$. Hence by \hyperref[Prop:usc-hom-ext]{Proposition \ref*{Prop:usc-hom-ext}} we have a locally-closed subset $X_{d,u}\subset\rep_\Lambda^d$ such that for any field $L$
\[ X_{d,u}(L) := \{\rho\in\rep_\Lambda^d(L) : \dim_L\Hom_{\Lambda^L}(M_\rho,M^L)=u\}. \]

In order to provide $X_{d,u}$ with the structure of a scheme we will use instead the following description. Given $\rho\in\rep_\Lambda^d(R)$ we obtain a linear map
\[ \Phi(\rho) \colon \bM_{m\times d}(R) \to \bM_{m\times d}(R)^N, \quad \phi \mapsto \big(\phi\rho_j-\tau_j\phi\big)_j, \]
which we could regard as a matrix  in $\bM_{dmN\times dm}(R)$. The coefficients of $\Phi(\rho)$ are polynomials of degree at most one in the coefficients of $\rho$, and hence we get a linear morphism of schemes
\[ \Phi \colon \rep_\Lambda^d \to \Hom\big(\bM_{m\times d},\bM_{m\times d}^N\big) \cong \bM_{dmN\times dm}, \quad \rho \mapsto \Phi(\rho). \]
We set $X_{d,u}$ to be the preimage of the subscheme $\mathrm{rank}_{dm-u}\subset\bM_{dmN\times dm}$.

It is clear that the action of $\GL_d$ on $\rep_\Lambda^d$ restricts to an action on $X_{d,u}$. Moreover, the direct sum map induces a morphism
\[ \Theta \colon \GL_{d+e}\times X_{d,u}\times X_{e,v} \to X_{d+e,u+v}, \quad (g,\rho,\sigma) \mapsto g\cdot(\rho\oplus\sigma). \]
To see this, we regard matrices in $\bM_{m\times(d+e)}$ in block form. Then for $\rho\in\rep_\Lambda^d(R)$ and $\sigma\in\rep_\Lambda^e(R)$ we have
\[ \Phi(\rho\oplus\sigma)(\phi,\psi) = \big((\phi\rho_j-\tau_j\phi,\psi\sigma_j-\tau_j\sigma)\big) = \big(\Phi(\rho)(\phi),\Phi(\sigma)(\psi)\big). \]
Thus $\Phi(\rho\oplus\sigma)$ and $\Phi(\rho)\oplus\Phi(\sigma)$ are equivalent matrices, so have the same rank.

In the notation of \hyperref[Cor:irred-cpt]{Corollary \ref*{Cor:irred-cpt}} we have the smooth, connected groups $G=\GL_{d+e}$ and $H=\GL_d\times\GL_e$, acting on the schemes $Y=X_{d+e,u+v}$ and $X=X_{d,u}\times X_{e,v}$, and conditions (1) and (2) again both hold.

\subsubsection{Remark}

There is an analogous construction for the functor $\Hom_\Lambda(M_\tau,-)$, yielding a scheme $X'_{d,u}$ such that
\[ X'_{d,u}(L) = \{\rho\in\rep_\Lambda^d(L) : \dim_L\Hom_{\Lambda^L}(M^L,M_\rho)=u\}. \]
In this case we use the map
\[ \Phi'(\rho) \colon \bM_{d\times m}(R) \to \bM_{d\times m}(R)^N, \quad \phi \mapsto \big(\phi\tau_j-\rho_j\phi\big)_j. \]
This construction has been considered by Zwara in \cite[\S3.3]{Zwara} (although he restricted to finite-dimensional algebras over an algebraically-closed field). For, consider the following free presentation of $M$
\[ \Lambda^N\otimes_KM \overset{\chi}{\to} \Lambda\otimes_K M \overset{\pi}{\to} M \to 0, \]
where $\pi(a\otimes m):=am$ and $\chi(a_jj\otimes m_j):=\sum_j(a_jx_j\otimes m_j-a_j\otimes x_jm_j)$.
Then for $\rho\in\rep_\Lambda^d(R)$ we have $\Hom_{\Lambda^R}(\chi^R,M_\rho)=\Phi'(\rho)$ as $R$-linear maps $\Hom_R(M^R,M_\rho)\to\Hom_R(M^R,M_\rho)^N$.

\subsection{Jet spaces}

\begin{Lem}[{c.f. \cite[Lemma 3.5]{Zwara}}]\label{Lem:jet-spaces}
As usual set $D_{r-1}:=L[t]/(t^r)$. Take $\hat\rho\in\rep_\Lambda^d(D_{r-1})$, let $\rho\in\rep_\Lambda^d(L)$ be its image under $D_{r-1}\to L$, and let $\tilde\rho\in\rep_\Lambda^{dr}(L)$ be the restriction of scalars. Then $\hat\rho\in X_{d,u}(D_{r-1})$ if and only if $\rho\in X_{d,u}(L)$ and $\tilde\rho\in X_{dr,ur}(L)$.
\end{Lem}

\begin{proof}
We first compute the image of $\Phi(\hat\rho)$ under the map $D_{r-1}\to L$, as well as its restriction of scalars to $L$.

Consider $\hat\phi=\sum_{i=0}^{r-1}\phi_it^i\in\bM_{m\times d}(D_{r-1})$. This is sent under $\Phi(\hat\rho)$ to $\big(\hat\phi\hat\rho_j-\tau_j\hat\phi\big)_j$. Writing $\hat\rho_j=\rho_j+\sum_i\xi_{ji}t^i$, then for a fixed $j=1,\ldots,N$ this equals
\[ \sum_{i=0}^{r-1}\Big((\phi_i\rho_j-\tau_j\phi_i)+\phi_{i-1}\xi_{j1}+\cdots+\phi_0\xi_{ji}\Big)t^i. \]
Under the map $D_{r-1}\to L$ this is just the map sending $\phi_0$ to $\big(\phi_0\rho_j-\tau_j\phi_0)_j$, so equals $\Phi(\rho)$. On the other hand, under restriction of scalars to $L$, we have the map sending $\tilde\phi=(\phi_0,\ldots,\phi_{r-1})$ to
\[ \Big(\phi_0\rho_j-\tau_j\phi_0,\phi_1\rho_j-\tau_j\phi_1+\phi_0\xi_{j1},\ldots,\phi_{r-1}\rho_j-\tau_j\phi_{r-1}+\phi_{r-2}\xi_{j1}+\cdots+\phi_0\xi_{jr-1}\Big)_j, \]
which is precisely $\Phi(\tilde\rho)(\tilde\phi)$. Thus the restriction of scalars of $\Phi(\hat\rho)$ is just $\Phi(\tilde\rho)$.

We can now apply \hyperref[Lem:rank]{Lemma \ref*{Lem:rank}} to deduce that $\hat\rho\in X_{d,u}(D_{r-1})$ if and only if $\rho\in X_{d,u}(L)$ and $\tilde\rho\in X_{dr,ur}(L)$.
\end{proof}

Note that, for $r=1$, the lemma implies that
\[ T_\rho X_{d,u} = X_{2d,2u}(L)\cap T_\rho\rep_\Lambda^d, \]
where we have used the identification of $T_\rho\rep_\Lambda^d$ with a subset of $\rep_\Lambda^{2d}(L)$. This gives
\[ T_\rho X_{d,u} = \Der_{2u}(\rho,\rho) := \{\xi\in\Der(\rho,\rho) : \dim\Hom_{\Lambda^L}(M_\xi,M^L)=2u\}. \]
More generally, given $\rho\in X_{d,u}(L)$ and $\sigma\in X_{e,v}(L)$, set
\[ \Der_{u+v}(\rho,\sigma) := \{\xi\in\Der(\rho,\sigma) : \dim\Hom_{\Lambda^L}(M_\xi,M^L)=u+v\}, \]
and let $E_{u+v}(\rho,\sigma)\subset\Ext_{\Lambda^L}^1(M_\rho,M_\sigma)$ be the image of $\Der_{u+v}(\rho,\sigma)$.

\subsubsection{Voigt's Lemma}

We begin with the following analogue of Voigt's Lemma.

\begin{Lem}\label{Lem:Voigt2}
Let $\rho\in X_{d,u}(L)$ and $\sigma\in X_{e,v}(L)$. Then $\Der_{u+v}(\rho,\sigma)\subset\Der(\rho,\sigma)$ is a subspace, and we have an exact sequence
\[ 0 \longrightarrow \Hom_{\Lambda^L}(M_\rho,M_\sigma) \longrightarrow \bM_{e\times d}(L) \overset{\delta_{\rho,\sigma}}{\longrightarrow} \Der_{u+v}(\rho,\sigma) \overset{\varepsilon_{\rho,\sigma}}{\longrightarrow} E_{u+v}(\rho,\sigma) \longrightarrow 0. \]
\end{Lem}

\begin{proof}
For each $\xi\in\Der(\rho,\sigma)$ we have a short exact sequence $0\to M_\sigma\to M_\xi\to M_\rho\to0$. Applying $\Hom_{\Lambda^L}(-,M^L)$ and comparing dimensions we see that $\xi\in\Der_{u+v}(\rho,\sigma)$ if and only if, for every $\Lambda^L$-homomorphism $\phi\colon M_\sigma\to M^L$, there exists an $L$-linear map $\psi\colon M_\rho\to M^L$ such that $\phi\xi=\tau\psi-\psi\rho=\Phi(\rho)(\psi)$. It is now clear that $\Der_{u+v}(\rho,\sigma)$ is a subspace of $\Der(\rho,\sigma)$. Moreover, it contains the image of $\delta_{\rho,\sigma}$, since if $\xi=\gamma\rho-\sigma\gamma$, then given $\phi\colon M_\sigma\to M^L$ we can take $\psi:=-\phi\gamma$.
\end{proof}

We next want to investigate how these maps interact with the direct sum morphism. For this we will need the following easy lemma.

\begin{Lem}\label{Lem:filtered}
Set $\dim\Hom_\Lambda(X,M)=x$ and $\dim\Hom_\Lambda(Y,M)=y$. Let $N$ be any module filtered by $a$ copies of $X$ and $b$ copies of $Y$ (so $N$ has a filtration with precisely these subquotients in some order). Then $\dim\Hom_\Lambda(N,M)\leq ax+by$. Moreover, if we have equality for $N$, then we have equality for all submodules of $U\leq N$
such that both $U$ and $N/U$ are filtered by copies of $X$ and $Y$.
\end{Lem}

\begin{proof}
Suppose we have a short exact sequence $0\to N_1\to N\to N_2\to 0$ such that $N_i$ is filtered by $a_i$ copies of $X$ and $b_i$ copies of $Y$, so that $a=a_1+a_2$ and $b=b_1+b_2$. By induction on dimension we know that $\dim\Hom_\Lambda(N_i,M)\leq a_ix+b_iy$, and since $\dim\Hom_\Lambda(N,M)\leq\dim\Hom_\Lambda(N_1,M)+\dim\Hom_\Lambda(N_2,M)$ we get the result for $N$. Conversely, if we have equality for $N$, then we must also have equality for both $N_1$ and $N_2$.
\end{proof}

\begin{Lem}
Let $\rho\in X_{d,u}(L)$ and $\sigma\in X_{e,v}(L)$. Then the standard decomposition
\[ T_{\rho\oplus\sigma}\rep_\Lambda^{d+e} \cong \Der(\rho,\rho) \times \Der(\sigma,\sigma) \times \Der(\rho,\sigma)\times \Der(\sigma,\rho) \]
restricts to a decomposition
\[ T_{\rho\oplus\sigma}X_{d+e,u+v} \cong \Der_{2u}(\rho,\rho) \times \Der_{2v}(\sigma,\sigma) \times \Der_{u+v}(\rho,\sigma) \times \Der_{u+v}(\sigma,\rho), \]
which in turn induces a decomposition of the extension groups
\[ E_{2(u+v)}(\rho\oplus\sigma,\rho\oplus\sigma) \cong E_{2u}(\rho,\rho) \times E_{2v}(\sigma,\sigma) \times E_{u+v}(\rho,\sigma) \times E_{u+v}(\sigma,\rho). \]
\end{Lem}

\begin{proof}
By comparing dimensions of homomorphisms to $M^L$ it is clear that each factor on the right is a subspace of the tangent space $T_{\rho\oplus\sigma}X_{d+e,u+v}$. Conversely, suppose we have a tangent vector $\xi=\begin{pmatrix}\xi_{11}&\xi_{12}\\\xi_{21}&\xi_{22}\end{pmatrix}$, so the module $N:=M_\xi$ corresponds to the representation
\[ \begin{pmatrix}\rho&0&\vline&\xi_{11}&\xi_{12}\\0&\sigma&\vline&\xi_{21}&\xi_{22}\\\hline0&0&\vline&\rho&0\\0&0&\vline&0&\sigma\end{pmatrix}. \]
Let $U\leq V\leq N$ be the submodules corresponding to the subrepresentations given by the first row and column, respectively the first three rows and columns. Then $V/U$ corresponds to the representation $\begin{pmatrix}\sigma&\xi_{21}\\0&\rho\end{pmatrix}$, so is isomorphic to $M_{\xi_{21}}$. Clearly each of $U$, $V/U$ and $N/V$ is filtered by copies of $M_\rho$ and $M_\sigma$, so we can apply \hyperref[Lem:filtered]{Lemma \ref*{Lem:filtered}} to deduce that $\dim_L\Hom_{\Lambda^L}(V/U,M)=u+v$. Thus $\xi_{21}\in\Der_{u+v}(\rho,\sigma)$ as required. The other cases are entirely analogous.
\end{proof}

In the notation of \hyperref[Cor:irred-cpt]{Corollary \ref*{Cor:irred-cpt}},
\[ N_{X,\rho\oplus\sigma}\cong E_{2u}(\rho,\rho)\times E_{2v}(\sigma,\sigma) \]
and
\begin{align*}
N_{Y,\rho\oplus\sigma} &\cong E_{2(u+v)}(\rho\oplus\sigma,\rho\oplus\sigma)\\
&\cong E_{2u}(\rho,\rho) \times E_{2v}(\sigma,\sigma) \times E_{u+v}(\rho,\sigma) \times E_{u+v}(\sigma,\rho).
\end{align*}
Moreover, the map $\theta_{\rho,\sigma}\colon N_{X,\rho\oplus\sigma}\to N_{Y,\rho\oplus\sigma}$ induced by the differential $d_{\rho,\sigma}\Theta$ is just the canonical embedding. It follows that condition (3) of the corollary is also satisfied, so $\Theta$ is separable on each irreducible component (with its reduced subscheme structure). Moreover, $\theta_{\rho,\sigma}$ is surjective if and only if $E_{u+v}(\rho,\sigma)=0=E_{u+v}(\sigma,\rho)$.

\subsubsection{Upper semi-continuity}

\begin{Lem}
The function sending $(\rho,\sigma)\in X_{d,u}(L)\times X_{e,v}(L)$ to  $\dim_LE_{u+v}(\rho,\sigma)$ is upper semi-continuous.
\end{Lem}

\begin{proof}
Consider the scheme $\Der_{u+v}(d,e):=\Der(d,e)\cap X_{d+e,u+v}$ together with the projection $\pi$ to $\rep_\Lambda^d\times\rep_\Lambda^e$. The fibre over a point $(\rho,\sigma)\in X_{d,u}(L)\times X_{e,v}(L)$ is $\Der_{u+v}(\rho,\sigma)$, so we can apply \hyperref[Cor:usc]{Corollary \ref*{Cor:usc}} to deduce that the function $(\rho,\sigma)\mapsto\dim_L\Der_{u+v}(\rho,\sigma)$ is upper semi-continuous on $X_{d,u}\times X_{e,v}$. By the analogue of Voigt's Lemma, \hyperref[Lem:Voigt2]{Lemma \ref*{Lem:Voigt2}},
\[ \dim_LE_{u+v}(\rho,\sigma) = \dim_L\Hom_{\Lambda^L}(M_\rho,M_\sigma) + \dim_L\Der_{u+v}(\rho,\sigma) - de, \]
so this function is also upper semi-continuous.
\end{proof}

As before, if $K$ is algebraically closed and $X\subset X_{d,u}$ and $Y\subset X_{e,v}$ are irreducible, then $e_{u+v}(X,Y)$ is the generic, or minimal, value of $\dim_L E_{u+v}(\rho,\sigma)$.

\subsubsection{Surjectivity of the differential}

Consider the direct sum morphism
\[ \Theta \colon \GL_{d+e}\times X_{d,u}\times X_{e,v} \longrightarrow X_{d+e,u+v} \]
and in particular the induced morphisms on jet spaces
\[ d_{\rho,\sigma}^{(r)}\Theta \colon T_1^{(r)}\GL_{d+e} \times T_\rho^{(r)}X_{d,u} \times T_\sigma^{(r)}X_{e,v} \longrightarrow T_{\rho\oplus\sigma}^{(r)}X_{d+e,u+v}. \]

\begin{Prop}\label{Prop:surj-diff2}
The following are equivalent for $\rho\in\rep_\Lambda^d(L)$ and $\sigma\in\rep_\Lambda^e(L)$.
\begin{enumerate}
\item $d_{\rho,\sigma}^{(r)}\Theta$ is surjective for all $r\in[1,\infty]$.
\item $d_{\rho,\sigma}^{(r)}\Theta$ is surjective for some $r\in[1,\infty]$.
\item $\theta_{\rho,\sigma}$ is surjective.
\end{enumerate}
In particular, condition (4) of \hyperref[Cor:irred-cpt]{Corollary \ref*{Cor:irred-cpt}} holds true.
\end{Prop}

\begin{proof}
$(1)\Rightarrow(2)\colon$ Trivial.

$(2)\Rightarrow(3)\colon$ As before, if $\xi\in\Der_{u+v}(\rho,\sigma)$ induces a non-split extension, then $(\rho\oplus\sigma)+\xi t\in T_{\rho\oplus\sigma}^{(r)}X_{d+e,u+v}$ but is not in the image of $d_{\rho,\sigma}^{(r)}\Theta$.

$(3)\Rightarrow(1)\colon$ We keep the same notation as in the proof of \hyperref[Prop:surj-diff1]{Proposition \ref*{Prop:surj-diff1}}, so we start with a representation $\widehat{\rho\oplus\sigma}\in X_{d+e,u+v}(D_r)$. We construct the modules $U\leq V\leq N$ in the same way, and observe that $N$, being given by restriction of scalars, satisfies $\dim\Hom_{\Lambda^L}(N,M^L)=(s+1)(u+v)$ by \hyperref[Lem:jet-spaces]{Lemma \ref*{Lem:jet-spaces}}. Also, the modules $U$, $V$ and $U/V$ are all filtered by copies of $M_\rho$ and $M_\sigma$, so $\dim\Hom_{\Lambda^L}(V/U,M^L)=u+v$ by \hyperref[Lem:filtered]{Lemma \ref*{Lem:filtered}}. In particular, the extension $0\to M_\sigma\to V/U\to M_\rho\to 0$ lies in $E_{u+v}(\rho,\sigma)$. We can therefore apply the analogue of Voigt's Lemma, \hyperref[Lem:Voigt2]{Lemma \ref*{Lem:Voigt2}} to obtain the matrix $g_s\in\GL_{d+e}(D_r)$.

Proceeding by induction as before we obtain $g\in\GL_{d+e}(D_r)$ such that $g\cdot\widehat{\rho\oplus\sigma}=\hat\rho\oplus\hat\sigma$ with $\hat\rho\in\rep_\Lambda^d(D_r)$ and $\hat\sigma\in\rep_\Lambda^e(D_r)$. Using \hyperref[Lem:filtered]{Lemma \ref*{Lem:filtered}} once more we conclude that $\hat\rho\in X_{d,u}(D_r)$ and $\hat\sigma\in X_{e,v}(D_r)$, proving that $d_{\rho,\sigma}\Theta^{(r)}$ is surjective.
\end{proof}

\subsection{Irreducible Components}

We can now state the analogues of \hyperref[Thm:CBS1]{Theorems \ref*{Thm:CBS1}} and \hyperref[Thm:CBS2]{\ref*{Thm:CBS2}} for the schemes $X_{d,u}$. After making the obvious changes, the proofs go through exactly as before.

\begin{Thm}
Let $K$ be an algebraically-closed field, $\Lambda$ a finitely-generated $K$-algebra, and $\tau\in\rep_\Lambda^m(K)$. Let $X\subset X_{d,u}$ and $Y\subset X_{e,v}$ be irreducible components. Then the closure $\overline{X\oplus Y}$ of the image of $\GL_{d+e}\times X\times Y$ is an irreducible component of $X_{d+e,u+v}$ if and only if $e_{u+v}(X,Y)=0=e_{u+v}(Y,X)$.
\end{Thm}

\begin{Thm}
Every irreducible component $X\subset X_{d,u}$ can be written uniquely (up to reordering) as a direct sum $X=\overline{X_1\oplus\cdots\oplus X_n}$ of generically indecomposable components $X_i\subset X_{d_i,u_i}$, where $d=\sum_id_i$ and $u=\sum_iu_i$.
\end{Thm}

We observe that both of these theorems can again be refined to the case when we consider dimension vectors (with respect to some complete set of orthogonal idempotents).

\subsubsection{Remarks}

Zwara used the schemes $X_{d,u}$ to investigate the scheme-theoretic properties of the restriction of scalars functor, and in particular to relate the singularity types of orbit closures \cite{Zwara}.

As an example of the schemes $X'_{d,u}$, let $e_i\in\Lambda$ be a complete set of orthogonal idempotents. Fixing the dimensions $d^i$ of $\Hom_\Lambda(\Lambda e_i,-)$ we recover the scheme $Y_\Lambda^{d\updot}=\GL_d\times^{\GL_{d\updot}}\rep_\Lambda^{d\updot}$ (c.f. \cite[Remark 3.14]{Zwara}).

Suppose that $\Lambda=KQ$ is the path algebra of a quiver $Q$ without oriented cycles, so that $\Lambda$ is finite dimensional. Then $\Lambda$ has a complete set of idempotents $e_i$ indexed by the vertices $i$ of $Q$. Let  $M$ be a module of dimension vector $m\updot$. We now have the `standard resolution' of $M$
\[ 0 \to P_1\xrightarrow{\phi_M} P_0\to M\to 0, \quad\textrm{where } P_0 := \bigoplus_i\Lambda e_i\otimes e_iM \textrm{ and }
P_1 := \bigoplus_{i\to j}\Lambda e_j\otimes e_iM, \]
and the sums are over the vertices and arrows of $Q$ respectively  (see for example \cite{RZ}). If $\langle\underline\dim\,M,d\updot\rangle=0$ for some dimension vector $d\updot$, then $\dim\Hom_\Lambda(P_0,N)=\dim\Hom_\Lambda(P_1,N)$ for all modules $N$ of dimension vector $d\updot$. The semi-invariant $c_M$ is then defined to be
\[ c_M \colon \rep_\Lambda^{d\updot} \to \bA^1, \quad N\mapsto\det\Hom_{\Lambda^R}((\phi_M)^R,N). \]
One usually studies the zero set $c_M^{-1}(0)$, which can be given a scheme structure by taking $(\Phi')^{-1}(\mathrm{rank}_{<t})$ for $t=\sum_im_id_i$, where now
\[ \Phi'(\rho) \colon \bigoplus_i\Hom_L(M_i^L,N_i) \to \bigoplus_{i\to j}\Hom_L(M_i^L,N_j). \]

When $\Lambda$ is a finite-dimensional algebra our construction also covers the Ext functors $\Ext^r_\Lambda(M,-)$ and $\Ext^r_\Lambda(-,M)$. For, fix a short exact sequence $0\to\Omega\to\Lambda\otimes_KM\overset{\pi}{\to} M\to0$. Then for any field $L$ and $\rho\in\rep_\Lambda^d(L)$ we have
\[ \dim_L\Ext^1_{\Lambda^L}(M^L,M_\rho) = \dim_L\Hom_{\Lambda^L}(M^L\oplus\Omega^L)-dm. \]
So, using the module $M\oplus\Omega$, the corresponding scheme $X'_{d,u}$ has $L$-valued points
\[ X'_{d,u}(L) = \{\rho\in\rep^d_\Lambda(L) : \dim_L\Ext^1_{\Lambda^L}(M^L,M_\rho) = u-dm\}. \]
In general, if $\Omega^r$ is the $r$-th syzygy of $M$, then $\dim\Ext^1(\Omega^r,X)=\dim\Ext^{r+1}(M,X)$ for $r\geq1$, so we can construct a subscheme of $\rep_\Lambda^d$ whose $L$-valued points parameterise those $d$-dimensional modules satisfying $\dim_L\Ext^r_{\Lambda^L}(M^L,X)=u$, for any fixed $u$. Doing this for a sincere semisimple module $M$ (so every simple module is isomorphic to a direct summand of $M$), we can construct a subscheme of $\rep_\Lambda^d$ parameterising those $d$-dimensional modules of a given projective dimension.

Dually, for the functor $\Ext^1_\Lambda(-,M)$, we can use the short exact sequence
\[ 0 \to M \overset{\epsilon}{\to} \Hom_K(\Lambda,M) \to \Omega^{-1} \to 0, \quad \epsilon(m)(a):=am, \]
so that $\Hom_{\Lambda^L}(M_\rho,\Hom_K(\Lambda,M)^L)\cong\Hom_L(M_\rho,M^L)$ has the same dimension for all $\rho\in\rep_\Lambda^d(L)$. Then using the cosyzygies $\Omega^{-r}$ we can fix the dimension of $\Ext^r_{\Lambda^L}(M_\rho,M^L)$.

As a final example, let $\Lambda$ be a finite dimensional algebra, and let us fix representatives of the preprojective and preinjective indecomposable modules (up to a suitably large dimension depending on $d$). By specifying the dimensions of the homomorphism spaces to the preprojectives and from the preinjectives, one can consider subschemes parameterising modules with prescribed preprojective and preinjective summands. In particular, this yields an alternative approach to orbits. More precisely, given a representation $\tau\in\rep_\Lambda^d(K)$ such that $M_\tau$ has only preprojective and preinjective summands, we can construct a (possibly non-reduced) subscheme $X\subset\rep_\Lambda^d$ such that, for all fields $L$, we have $X(L)=\mathrm{Orb}_{\GL_d}(\tau)(L)$, and hence that $\mathrm{Orb}_{\GL_d}(\tau)=X_\red$. Note that, by \hyperref[Lem:jet-spaces]{Lemma \ref*{Lem:jet-spaces}}, if $D_{r-1}:=L[t]/(t^r)$ and $\hat\rho\in\rep_\Lambda^d(D_{r-1})$, then $\hat\rho\in X(D_{r-1})$ if and only if $M_\rho$ is isomorphic to $M_\tau$ and $M_{\tilde\rho}$ is isomorphic to $M_\tau^r$. In fact, this holds for all Artinian local rings $R$ when $K$ is a perfect field (so that $R$ has a coefficient field containing $K$).

\section{Grassmannians of submodules}\label{Sec:Grassmannians}

Our next application is to Grassmannians of submodules, which are projective schemes parameterising the $d$-dimensional submodules of a fixed $\Lambda$-module.

Again, $K$ will denote a perfect field and $\Lambda=K\langle x_1,\ldots,x_N\rangle/I$ a finitely-generated $K$-algebra.

\subsection{Grassmannians}

Let $M$ be an $m$-dimensional vector space over $K$. Then the Grassmannian of $d$-dimensional subspaces of $M$ is the projective scheme given by
\[ \Gr_K\binom Md(R) := \big\{ \textrm{$R$-module direct summands $U\leq M^R$ of rank $d$} \big\}. \]
If now $M=M_\tau$ for some representation $\tau\in\rep_\Lambda^m(K)$, then the Grassmannian of $d$-dimensional submodules of $M$ is the closed subscheme given by
\[ \Gr_\Lambda\binom Md(R) := \Big\{ U\in\Gr_K\binom Md(R) : \tau^R(U)\subset U \Big\}. \]
As usual we have written $\tau^R$ for the corresponding representation in $\rep_\Lambda^d(R)$.

It will be useful to consider the following construction of the Grassmannian as a geometric quotient. Recall that we have the open subscheme of $\bM_{m\times d}$ given via
\begin{align*}
\inj_{m\times d}(R) &:= \mathrm{rank}_d(R)\\
&\;= \{f\in\bM_{m\times d}(R) : \textrm{the $d$ minors of $f$ generate the unit ideal in $R$} \}.
\end{align*}
In particular, if $R$ is local, then $f\in\inj_{m\times d}(R)$ if and only if some $d$ minor of $f$ is invertible.

The free $\GL_d$-action on $\bM_{m\times d}$, $(g,f)\mapsto g\cdot f:=fg^{-1}$, restricts to an action on this open subscheme, and the map
\[ \pi \colon \inj_{m\times d} \to \Gr_K\binom Md, \quad \theta \mapsto \Ima(f) \]
is a principal $\GL_d$-bundle. Thus $\Gr_K\binom Md\cong \inj_{m\times d}/\GL_d$ is isomorphic to the quotient faisceau, by \hyperref[Lem:principal-G-bundle]{Lemma \ref*{Lem:principal-G-bundle}}, so in particular is a universal geometric quotient.

Explicitly, let $J\subset\{1,\ldots,m\}$ be a subset of size $d$. Identifying $M\cong K^m$, let $V_J\leq K^m$ be the subspace spanned by the basis elements $e_i$ for $i\not\in J$. Then the Grassmannian $\Gr_K\binom Md$ is covered by the open subschemes $\mathcal U_J$, having as $R$-valued points those $U$ such that $U\oplus V_J^R=R^m$. Next observe that $\pi^{-1}(\mathcal U_J)=D(\Delta_J)$, the distinguished open subscheme given by the minor $\Delta_J$. To see that $\pi$ is trivial over $\mathcal U_J$ let $\alpha_J\colon\bM_{m\times d}\to\bM_d$ be the trivial vector bundle given by restricting to the rows indexed by $J$, so that $\Delta_J(f)=\det(\alpha_J(f))$. Define $U_J\subset\bM_{m\times d}$ by taking those $f$ such that $\alpha_J(f)=1_d$ is the identity matrix. Then there is an isomorphism
\[ \GL_d\times U_J \xrightarrow{\sim} D(\Delta_J), \quad (g,f) \mapsto fg^{-1} \]
with inverse $f\mapsto(\alpha_J(f)^{-1},f\alpha_J(f)^{-1})$. This induces an isomorphism $U_J\xrightarrow{\sim}\mathcal U_J$; the inverse sends $U$ to the matrix with columns $u_j$ for $j\in J$, where $e_j=(u_j,v_j)\in U\oplus V_J$.

Note also that $U_J$ is an affine scheme. For, let $\alpha_J'\colon\bM_{m\times d}\to\bM_{(m-d)\times d}$ be the trivial vector bundle given by taking the rows not in $J$. Then $\alpha_J'$ restricts to an isomorphism $U_J\xrightarrow{\sim}\bM_{(m-d)\times d}$.

We can emulate this construction for the Grassmannian of $d$-dimensional submodules of $M$ as follows. Recall the scheme
\[ \rep_{\Lambda(2)}^{(d,m)}(R) := \{(\rho,\sigma,f)\in\rep_\Lambda^d(R)\times\rep_\Lambda^m(R)\times\bM_{m\times d}(R) : f\rho=\sigma f \} \]
and its closed subscheme
\[ \rep_{\Lambda(2)}^{(d,\tau)}(R) := \{(\rho,\sigma,f)\in\rep_{\Lambda(2)}^{(d,m)}(R):\sigma=\tau^R\} \cong \{(\rho,f) : f\rho=\tau^Rf \}. \]
We can therefore consider the open subscheme $\rep\inj_{\Lambda(2)}^{(d,\tau)}$ of $\rep_{\Lambda(2)}^{(d,\tau)}$ given by
\[ \rep\inj_{\Lambda(2)}^{(d,\tau)}(R) := \{(\rho,f)\in\rep_{\Lambda(2)}^{(d,\tau)}(R) : f\in\inj_{m\times d}(R) \}. \]
We will often abuse notation and simply write $\rep\inj_\Lambda^{(d,M)}$ instead of $\rep\inj_{\Lambda(2)}^{(d,\tau)}$.

\begin{Lem}
The map
\[ \iota \colon \rep\inj_\Lambda^{(d,M)} \to \inj_{m\times d}, \quad (\rho,f)\mapsto f, \]
is a closed immersion.
\end{Lem}

\begin{proof}
Since $f$ is injective, there is at most one $\rho$ such that $(\rho,f)\in\rep\inj_\Lambda^{(d,M)}(L)$. Such a $\rho$ exists if and only if the composite $C(f)\tau f=0$, where $C(f)$ is any cokernel for $f$. Locally we can choose cokernels as follows. Given $J\subset\{1,\ldots,m\}$ of size $d$ we have the minor $\Delta_J$ as well as the trivial vector bundles $\alpha_J\colon\bM_{m\times d}\to\bM_d$ and $\alpha'_J\colon\bM_{m\times d}\to\bM_{(m-d)\times d}$ described above. Dually we have trivial vector bundles $\bM_{(m-d)\times m}\to\bM_{(m-d)\times d}$ and $\bM_{(m-d)\times m}\to\bM_{m-d}$ by taking the columns in $J$, respectively the columns not in $J$; let $\beta_J$ and $\beta'_J$ be their respective zero sections. Clearly $\beta_J(A)f=A\alpha_J(f)$ for all matrices $A$, and similarly for $J'$. Thus
\[ C \colon D(\Delta_J)\to\bM_{(m-d)\times m}, \quad f \mapsto \beta'_J(1_{m-d})-\beta_J\big(\alpha'_J(f)\alpha_J(f)^{-1}\big), \]
is a morphism of schemes such that $C(f)$ is a cokernel for $f$. We may therefore define a closed subscheme $V_J\subset D(\Delta_J)$ by requiring that $C(f)\tau f=0$. Then the morphism $\iota\colon\rep\inj_\Lambda^{(d,M)}\to\inj_{m\times d}$ restricts to an isomorphism $\iota^{-1}(D(\Delta_J))\xrightarrow{\sim}V_J$. Since the $D(\Delta_J)$ cover $\inj_{m\times d}$, the result follows from \cite[I \S2 Corollary 4.9]{DG}.
\end{proof}

The group $\GL_d$ again acts freely via $g\cdot(\rho,f):=(g\cdot\rho,g\cdot f)=(g\rho g^{-1},f g^{-1})$ and the morphism $\iota$ is $\GL_d$-equivariant. Using \hyperref[Lem:principal-G-bundle-subscheme]{Lemma \ref*{Lem:principal-G-bundle-subscheme}} we conclude that the morphism
\[ \pi \colon \rep\inj_\Lambda^{(d,M)} \to \Gr_\Lambda\binom Md, \quad (\rho,f) \mapsto \Ima(f) \]
is again a principal $\GL_d$-bundle.

\subsection{Principal bundles}

The aim of this section is to relate the direct sum morphism for Grassmannians to the direct sum morphism for the schemes $\rep\inj_\Lambda^{(d,M)}$, which are easier to work with when computing jet spaces. In fact we will show that one of these direct sum morphisms satisfies the conditions of \hyperref[Cor:irred-cpt]{Corollary \ref*{Cor:irred-cpt}} if and only the other does, and that they both induce the same maps $\theta_x$.

We have the group scheme
\[ G = G_1\times G_2 = \GL_{d+e}\times\Aut_\Lambda(M\oplus N) \]
acting on the scheme
\[ Y = \rep\inj_\Lambda^{(d+e,M\oplus N)} \]
via $(g_1,g_2)\cdot(\rho,f):=(g_1fg_1^{-1},g_2fg_1^{-1})$, and the closed subgroup
\[ H = H_1\times H_2 = (\GL_d\times\GL_e)\times(\Aut_\Lambda(M)\times\Aut_\Lambda(N)) \]
acting on the closed subscheme
\[ X = \rep\inj_\Lambda^{(d,M)}\times\rep\inj_\Lambda^{(e,N)}. \]
Note that the closed immersion $X\to Y$ is just the restriction of the closed immersion for $\Lambda(2)$-representations
\[ \rep_{\Lambda(2)}^{(d,m)}\times\rep_{\Lambda(2)}^{(e,n)} \to \rep_{\Lambda(2)}^{(d+e,m+n)} \]
described earlier, but using the refinement to dimension vectors. Similarly the action of $G$ on $Y$ is the restriction of the action of
\[ \GL_{(d+e,m+n)} = \GL_{d+e}\times\GL_{m+n} \]
on $\rep_{\Lambda(2)}^{(d+e,m+n)}$, and the direct sum morphism for $\Lambda(2)$-representations restricts to a morphism
\[ \Theta \colon G \times X \to Y. \]
It follows immediately that this morphism satisfies conditions (1) and (2) of \hyperref[Cor:irred-cpt]{Corollary \ref*{Cor:irred-cpt}}.

We also have the principal bundles
\[ \pi_Y \colon Y \to Y/G_1 = \Gr_\Lambda\binom{M\oplus N}{d+e} \]
and
\[ \pi_X \colon X \to X/H_1 = \Gr_\Lambda\binom Md\times\Gr_\Lambda\binom Ne. \]
We therefore get an induced action of $G_2$ on $Y/G_1$, and an action of $H_2$ on $X/H_1$. Moreover, since principal bundles are quotients in the category of faisceaux, \hyperref[Lem:principal-G-bundle]{Lemma \ref*{Lem:principal-G-bundle}}, the $H_1$-invariant morphism $X\to Y\to Y/G_1$ induces a morphism $\iota\colon X/H_1\to Y/G_1$.

This is again a closed immersion. For, let $J\subset\{1,\ldots,m+n\}$ be any subset of size $d+e$ and write $J=J_1\coprod J_2$ where $J_1=J\cap\{1,\ldots,m\}$. If $J_1$ is not of size $d$, then $\iota^{-1}(\mathcal U_J)=\emptyset$. If, on the other hand, $J_1$ has size $d$, then $\iota^{-1}(\mathcal U_J)\cong\mathcal U_{J_1}\times\mathcal U_{J_2}$. Writing these as affine schemes, we obtain the morphism $\bM_{(m-d)\times d}\times\bM_{(n-e)\times e}\to\bM_{(m-d+n-e)\times(d+e)}$ given by taking the block diagonal matrices, which we know is a closed immersion. Thus $\iota$ is a closed immersion by \cite[I \S2 Corollary 4.9]{DG}.

The direct sum morphism induces a morphism
\[ \Theta \colon G_2 \times X/H_1 \to Y/G_1, \]
so a morphism
\[ \Theta \colon \Aut_\Lambda(M\oplus N)\times\Gr_\Lambda\binom Md\times\Gr_\Lambda\binom Ne \to \Gr_\Lambda\binom{M\oplus N}{d+e}. \]
This again satifies conditions (1) and (2) of \hyperref[Cor:irred-cpt]{Corollary \ref*{Cor:irred-cpt}}. To see this, note that the projection $H_1\times H_2\to H_2$ induces a group isomorphism $\Stab_H(x)\to\Stab_{H_2}(\pi_X(x))$ for each $x\in X(L)$, so (1) holds. (It is injective, since if $(h_1,h_2),(h_1',h_2)$ both fix $x$, then $h_1^{-1}h_1'$ fixes $h_2\cdot x$, whence $h_1=h_1'$. It is surjective, since if $h_2$ fixes $\pi_X(x)$, then $h_2\cdot x$ maps to $\pi_X(x)$, so $h_2\cdot x=h_1^{-1}\cdot x$ for some $h_1\in H_1$, whence $(h_1,h_2)$ fixes $x$.) For (2) we just need to observe that there is a bijection between the $H_2$-orbits on $X/H_1$ and the $H$-orbits on $X$.

As for the third and fourth conditions, we observe that for all $x\in X(L)$ and all $r\in[1,\infty]$ the morphism
\[ \pi_X \colon T_x^{(r)}X \to T_{\pi_X(x)}^{(r)}(X/H_1) \]
is surjective, with fibres the orbits under the action of the group $T_1^{(r)}H_1$ (viewed as a subgroup of $H_1(D_r)$). This follows from the local triviality of $\pi_X$ together with the fact that $D_r$ is a local ring.

When $r=1$ we deduce that there is an exact commutative diagram (of vector spaces over $L$)
\[ \begin{CD}
0 @>>> T_1H_1 @>>> T_1H @>>> T_1H_2 @>>> 0\\
@. @| @VVV @VVV\\
0 @>>> T_1H_1 @>>> T_xX @>>> T_{\pi_X(x)}(X/H_1) @>>> 0
\end{CD} \]
so the Snake Lemma yields the isomorphism $N_{X,x}\cong N_{X/H_1,\pi_X(x)}$. Similarly $N_{Y,x}\cong N_{Y/G_1,\pi_Y(x)}$, and so we can identify the two linear maps
\[ \theta_x \colon N_{X,x} \to N_{Y,x} \quad\textrm{and}\quad \theta_x \colon N_{X/H_1,\pi_X(x)} \to N_{Y/G_1,\pi_Y(x)}. \]
On the other hand, when $r=\infty$, we have the commutative square
\[ \begin{CD}
T_1^{(\infty)}G\times T_x^{(\infty)}X @>>> T_1^{(\infty)}G_2\times T_{\pi_X(x)}^{(\infty)}(X/H_1)\\
@VVV @VVV\\
T_x^{(\infty)}Y @>>> T_{\pi_Y(x)}^{(\infty)}(Y/G_1)
\end{CD} \]
and the left-hand map is surjective if and only if the right-hand map is surjective.

For, the upper and lower maps are both surjective, so if the left-hand map is surjective, so too is the right-hand map. Conversely, suppose the right-hand map is surjective and consider $\eta\in T_x^{(\infty)}Y$. By assumption we can find $(\gamma,\xi)\in T_1^{(\infty)}G\times T_x^{(\infty)}X$ such that $\eta$ and $(\gamma,\xi)$ are sent to the same element of $T_{\pi_Y(x)}^{(\infty)}(Y/G_1)$. Thus $\gamma\cdot\xi\in T_x^{(\infty)}Y$ and $\eta$ have the same image, so $\eta=\gamma_1\gamma\cdot\xi$ for some $\gamma_1\in T_1^{(\infty)}G_1$. It follows that $(\gamma_1\gamma,\xi)$ is sent to $\eta$, and so the left-hand map is surjective.

This proves that the morphism $\Theta \colon G_2\times X/H_1 \to Y/G_1$ satisfies conditions (3) and (4) of \hyperref[Cor:irred-cpt]{Corollary \ref*{Cor:irred-cpt}} if and only if the morphism $\Theta \colon G\times X \to Y$ does. In other words, we can work with either the direct sum morphism for Grassmannians
\[ \Theta \colon \Aut_\Lambda(M\oplus N)\times\Gr_\Lambda\binom Md\times\Gr_\Lambda\binom Ne \to \Gr_\Lambda\binom{M\oplus N}{d+e}, \]
or for the subschemes of $\Lambda(2)$-representations
\[ \Theta \colon \!\GL_{d+e}\times\Aut_\Lambda(M\oplus N)\times\rep\inj_\Lambda^{(d,M)}\!\times\!\rep\inj_\Lambda^{(e,N)} \to \rep\inj_\Lambda^{(d+e,M\oplus N)}\!. \]

\subsection{Jet spaces}

We need to compute the jet spaces of $\rep\inj_\Lambda^{(d,M)}$. By definition, given $(\rho,f)\in\rep\inj_\Lambda^{(d,M)}(L)$, the jet space $T_{(\rho,f)}^{(r)}\rep\inj_\Lambda^{(d,M)}$ consists of those pairs $(\hat\rho,\hat f)$, where $\hat\rho=\rho+\sum_{i=1}^r\xi_it^i\in\rep_\Lambda^d(D_r)$ and $\hat f=f+\sum_{i=1}^r\phi_it^i\in\inj_{m\times d}(D_r)$, such that $\hat f\hat\rho=\tau_M^L\hat f$. As usual we may identify $\hat\rho$ with the representation $\tilde\rho\in\rep_\Lambda^{(r+1)d}(L)$ given by
\[ \tilde\rho := \begin{pmatrix}
\rho&\xi_1&\xi_2&\cdots&\xi_r\\
0&\rho&\xi_1&\ddots&\vdots\\
\vdots&\ddots&\ddots&\ddots&\xi_2\\
\vdots&&\ddots&\rho&\xi_1\\
0&\cdots&\cdots&0&\rho
\end{pmatrix} \]
If we therefore define
\[ \tilde f := (f,\phi_1,\ldots,\phi_r) \in \bM_{m\times(r+1)d}(L), \]
then it is clear that $(\hat\rho,\hat f)\in T_{(\rho,f)}^{(r)}\rep\inj_\Lambda^{(d,M)}$ if and only if $\tilde\rho\in\rep_\Lambda^{(r+1)d}(L)$ and $\tilde f\in\Hom_{\Lambda^L}(M_{\tilde\rho},M^L)$, which we can write as $(\tilde\rho,\tilde f)\in\rep\hom_\Lambda^{((r+1)d,M)}(L)$.

In particular, when $r=1$ we have the tangent space
\[ T_{(\rho,f)}\rep\inj_\Lambda^{(d,M)} = \{(\xi,\phi)\in\Der(\rho,\rho)\times\bM_{m\times d}(L): f\xi = \tau_M^L\phi-\phi\rho \}. \]
Note that, as expected, since $f$ is injective the derivation $\xi$ is uniquely determined by $\phi$.

The action of $\GL_d$ on $\rep\inj_\Lambda^{(d,M)}$ induces the following action of $T_1\GL_d=\bM_d(L)$ on $T_{(\rho,f)}\rep\inj^{(d,M)}$
\begin{align*}
\bM_d(L) \times T_{(\rho,f)}\rep\inj^{(d,M)} &\to T_{(\rho,f)}\rep\inj^{(d,M)}\\
\gamma\cdot(\xi,\phi) &:= (\xi+\gamma\rho-\rho\gamma,\phi-f\gamma) = (\xi+\delta_\rho(\gamma),\phi-f\gamma).
\end{align*}
More generally we can apply our earlier results to the algebra $\Lambda(2)$, using the refinement to dimension vectors: given $(\rho,\rho',f)\in\rep_{\Lambda(2)}^{(d,m)}(L)$ and $(\sigma,\sigma',g)\in\rep_{\Lambda(2)}^{(e,n)}(L)$, we considered the vector space $\Der\big((\rho,\rho',f),(\sigma,\sigma',g)\big)$ consisting of triples $(\xi,\xi',\phi)$ such that
\[ \left(\begin{pmatrix}\sigma&\xi\\0&\rho\end{pmatrix},\begin{pmatrix}\sigma'&\xi'\\0&\rho'\end{pmatrix},\begin{pmatrix}g&\phi\\0&f\end{pmatrix}\right) \in \rep_{\Lambda(2)}^{(d+e,m+n)}(L), \]
which we can also write as
\[ \xi\in\Der(\rho,\sigma), \quad \xi'\in\Der(\rho',\sigma'), \quad g\xi+\phi\rho=\sigma'\phi+\xi' f. \]
The additive group $\bM_{e\times d}(L)\times\bM_{n\times m}(L)$ acted via
\[ (\gamma,\gamma')\cdot(\xi,\xi',\phi) := \big(\xi+\delta_{\rho,\sigma}(\gamma),\xi'+\delta_{\rho',\sigma'}(\gamma'),\phi+\gamma'f-g\gamma\big), \]
and from this we constructed the long exact sequence for Voigt's Lemma
\begin{multline*}
0 \longrightarrow \Hom_{\Lambda^L(2)}\big((\rho,\rho',f),(\sigma,\sigma',g)\big) \longrightarrow \bM_{e\times d}(L)\times\bM_{n\times m}(L) \overset{\delta}{\longrightarrow}\\
\Der\big((\rho,\rho',f),(\sigma,\sigma',g)\big) \longrightarrow \Ext_{\Lambda^L(2)}^1\big((\rho,\rho',f),(\sigma,\sigma',g)\big) \longrightarrow 0.
\end{multline*}
In the current situation we are fixing the representations $\rho':=\tau_M^L$ and $\sigma':=\tau_N^L$, so we restrict to the subspace
\[ \overline{\Der}\big((\rho,f),(\sigma,g)\big) := \big\{(\xi,\phi) : (\xi,0,\phi)\in\Der\big((\rho,\tau_M^L,f),(\sigma,\tau_N^L,g)\big)\big\}. \]
Accordingly we also take the subgroup
\[ \delta^{-1}\Big(\overline{\Der}\big((\rho,f),(\sigma,g)\big)\Big) = \bM_{e\times d}(L)\times\Hom_{\Lambda^L}(M^L,N^L). \]
Clearly $T_{(\rho,f)}\rep\inj_\Lambda^{(d,M)}=\overline{\Der}\big((\rho,f),(\rho,f)\big)$, so this construction specialises to the tangent spaces.

\subsubsection{Voigt's Lemma}

\begin{Lem}\label{Lem:Voigt3}
Let $(\rho,f)\in\rep\inj_\Lambda^{(d,M)}(L)$ and $(\sigma,g)\in\rep\inj_\Lambda^{(e,N)}(L)$. Then we have an exact sequence
\begin{multline*}
0 \longrightarrow \Hom_{\Lambda^L(2)}\big((\rho,\tau_M^L,f),(\sigma,\tau_N^L,g)\big) \longrightarrow \bM_{e\times d}(L)\times\Hom_{\Lambda^L}(M^L,N^L)\\
\overset{\delta}{\longrightarrow}\overline{\Der}\big((\rho,f),(\sigma,g)\big) \longrightarrow \Ext_{\Lambda^L(2)}^1\big((\rho,\tau_M^L,f),(\sigma,\tau_N^L,g)\big).
\end{multline*}
\end{Lem}

\begin{proof}
This is clear since $\bM_{e\times d}(L)\times\Hom_{\Lambda^L}(M^L,N^L)$ is precisely the preimage under $\delta$ of $\overline{\Der}\big((\rho,f),(\sigma,g)\big)$.
\end{proof}

We can transfer this result to Grassmannians by taking the quotient modulo the action of $\bM_{e\times d}(L)$.

\begin{Lem}
Let $(\rho,f)\in\rep\inj_\Lambda^{(d,M)}(L)$ and $(\sigma,g)\in\rep\inj_\Lambda^{(e,N)}(L)$, and choose a cokernel $(\bar\sigma,\bar g)$ for $(\sigma,g)$; that is, we have a short exact sequence
\[ 0 \to M_\sigma \xrightarrow{g} N^L \xrightarrow{\bar g} M_{\bar\sigma} \to 0. \]
Then the morphism
\[ \overline{\Der}\big((\rho,f),(\sigma,g)\big) \to \Hom_{\Lambda^L}(M_\rho,M_{\bar\sigma}), \quad (\xi,\phi)\mapsto\bar g\phi, \]
is a quotient for the $\bM_{e\times d}(L)$-action. In particular, when $(\rho,f)=(\sigma,g)$ we recover the usual isomorphism
\[ T_U\Gr_\Lambda\binom Md \cong \Hom_{\Lambda^L}(U,M^L/U). \]
\end{Lem}

\begin{proof}
We have $g\xi+\phi\rho=\tau_N^L\phi$, so from $\bar gg=0$ and $\bar g\tau_N^L=\bar\sigma\bar g$ we get $\bar g\phi\rho=\bar g\tau_N^L\phi=\bar\sigma\bar g\phi$, whence $\bar g\phi\in\Hom_{\Lambda^L}(M_\rho,M_{\bar\sigma})$. Next, if $(\xi',\phi')$ satisfies $\bar g\phi'=\bar g\phi$, then there exists a unique $\gamma\in\bM_{e\times d}(L)$ such that $\phi'=\phi-g\gamma$. It follows that $g\xi'=g(\xi+\delta_{\rho,\sigma}(\gamma))$, so the injectivity of $g$ gives $(\xi',\phi')=\gamma\cdot(\xi,\phi)$. Finally, given any $\bar\phi\colon M_\rho\to M_{\bar\sigma}$, we can form the pull-back
\[ \begin{CD}
0 @>>> M_\sigma @>>> M_\xi @>>> M_\rho @>>> 0\\
@. @| @VV{(g,\phi)}V @VV{\bar\phi}V\\
0 @>>> M_\sigma @>{g}>> N^L @>{\bar g}>> M_{\bar\sigma} @>>> 0.
\end{CD} \]
to obtain $(\xi,\phi)\in\overline{\Der}\big((\rho,f),(\sigma,g)\big)$ such that $\bar g\phi=\bar\phi$.
\end{proof}

Given $U\in\Gr_\Lambda\binom Md(L)$, we write $U\updot:=(U\subset M^L)$ for the corresponding $\Lambda^L(2)$-representation. For convenience we also identify $M^L:=(M^L=M^L)$. It follows that  $M^L/U\updot=(M^L/U\to0)$. The analogue of Voigt's Lemma for Grassmannians can now be stated as follows.

\begin{Cor}\label{Cor:Voigt3}
Given $U\in\Gr_\Lambda\binom Md(L)$ and $V\in\Gr_\Lambda\binom Ne(L)$, we have an exact sequence
\begin{multline*}
0 \longrightarrow \Hom_{\Lambda^L(2)}(U\updot,V\updot) \longrightarrow \Hom_{\Lambda^L}(M^L,N^L)\\
\longrightarrow\Hom_{\Lambda^L}(U,N^L/V) \longrightarrow \Ext_{\Lambda^L(2)}^1(U\updot,V\updot),
\end{multline*}
which we may identify with the exact sequence
\begin{multline*}
0 \longrightarrow \Hom_{\Lambda^L(2)}(U\updot,V\updot) \longrightarrow \Hom_{\Lambda^L(2)}(U\updot,N^L)\\
\longrightarrow\Hom_{\Lambda^L(2)}(U\updot,N^L/V\updot) \longrightarrow \Ext_{\Lambda^L(2)}^1(U\updot,V\updot)
\end{multline*}
obtained by applying $\Hom_{\Lambda^L(2)}(U\updot,-)$ to the short exact sequence
\[ 0 \longrightarrow V\updot \longrightarrow N^L \longrightarrow N^L/V\updot \longrightarrow 0. \]
\end{Cor}

\begin{proof}
The first sequence follows from \hyperref[Lem:Voigt3]{Lemma \ref*{Lem:Voigt3}} by taking the quotient for the $\bM_{e\times d}(L)$-action, as described above, and noting that, since $g\colon V\to N^L$ is injective, so too is the map
\[ \Hom_{\Lambda^L(2)}(U\updot,V\updot) \to \Hom_{\Lambda^L}(M^L,N^L). \]
We observe that the map from $\Hom_{\Lambda^L}(U,N^L/V)$ to the extension group is given by first taking a pull-back, yielding an extension $M_\xi$ of $U$ by $V$ together with a morphism $(g,\phi)\colon M_\xi\to N^L$, and then noting that this fits into an exact commutative diagram
\[ \begin{CD}
0 @>>> V @>>> M_\xi @>>> U @>>> 0\\
@. @VV{g}V @VV{\big(\begin{smallmatrix}g&\phi\\0&f\end{smallmatrix}\big)}V @VV{f}V\\
0 @>>> N^L @>>> N^L\oplus M^L @>>> M^L @>>> 0,
\end{CD} \]
which we regard as an extension of $\Lambda^L(2)$-modules.
\end{proof}

It is now clear that these maps interact well with the direct sum morphism. More precisely, the closed immersion
\[ \Gr_\Lambda\binom Md \times \Gr_\Lambda\binom Ne \to \Gr_\Lambda\binom{M\oplus N}{d+e} \]
gives rise to the embedding of tangent spaces
\[ T_U\Gr_\Lambda\binom Md \times T_V\Gr_\Lambda\binom Ne \to T_{U\oplus V}\Gr_\Lambda\binom{M\oplus N}{d+e} \]
corresponding to the standard embedding
\[ \Hom_{\Lambda^L}(U,M^L/U) \times \Hom_{\Lambda^L}(V,N^L/V) \to \Hom_{\Lambda^L}(U\oplus V,(M^L/U)\oplus(N^L/V)), \]
or equivalently the embedding
\begin{multline*}
\Hom_{\Lambda^L(2)}(U\updot,M^L/U\updot) \times \Hom_{\Lambda^L(2)}(V\updot,N^L/V\updot)\\
\to \Hom_{\Lambda^L(2)}(U\updot\oplus V\updot,M^L/U\updot\oplus N^L/V\updot).
\end{multline*}
This is compatible with the actions of the endomorphism groups
\[ \End_{\Lambda^L}(M^L)\times\End_{\Lambda^L}(N^L) \to \End_{\Lambda^L}(M^L\oplus N^L), \]
and so we see that the map $\theta_{U,V}$ is just the restriction of the standard embedding
\[ \Ext_{\Lambda^L(2)}^1(U\updot,U\updot) \times \Ext_{\Lambda^L(2)}^1(V\updot,V\updot) \to \Ext_{\Lambda^L(2)}^1(U\updot\oplus V\updot,U\updot\oplus V\updot). \]
It follows that this is always injective, so condition (3) of \hyperref[Cor:irred-cpt]{Corollary \ref*{Cor:irred-cpt}} also holds and the direct sum morphism $\Theta$ is separable on each irreducible component (with its reduced subscheme structure).

We will write $\overline\Ext(U\updot,V\updot)$ for the image of the map
\[ \Hom_{\Lambda^L}(U,N^L/V) \to \Ext_{\Lambda^L(2)}^1(U\updot,V\updot), \]
consisting of those extension classes which are pull-backs along a morphism $U\to N^L/V$. Moreover, $\overline\Ext(U\updot\oplus V\updot,U\updot\oplus V\updot)$ decomposes as
\[ \overline\Ext(U\updot,U\updot) \times \overline\Ext(V\updot,V\updot) \times \overline\Ext(U\updot,V\updot) \times \overline\Ext(V\updot,U\updot), \]
so $\theta_{U,V}$ is surjective if and only if $\overline\Ext(U\updot,V\updot)=0=\overline\Ext(V\updot,U\updot)$. In terms of extensions this says that every pull-back
\[ \begin{CD}
0 @>>> V @>{a}>> E @>{b}>> U @>>> 0\\
@. @| @VV{\phi}V @VV{\bar\phi}V\\
0 @>>> V @>{g}>> N^L @>{\bar g}>> N^L/V @>>> 0,
\end{CD} \]
gives rise to a split extension of $\Lambda^L(2)$-modules
\[ \begin{CD}
0 @>>> V @>{a}>> E @>{b}>> U @>>> 0\\
@. @VV{g}V @VV{\binom{\phi}{fb}}V @VV{f}V\\
0 @>>> N^L @>>> N^L\oplus M^L @>>> M^L @>>> 0,
\end{CD} \]
and similarly for every pull-back along some $V\to M^L/U$ of
\[ 0 \to U \to M^L \to M^L/U \to 0. \]

Note that this is a stronger condition than just saying that the pull-back of $\Lambda^L$-modules is split. For, the first pull-back is split as a sequence of $\Lambda^L$-modules if and only if there exists some $s\in\Hom_{\Lambda^L}(U,E)$ such that $bs=1$, in which case $\phi s\in\Hom_{\Lambda^L}(U,N^L)$, whereas the second diagram is split as a sequence of $\Lambda^L(2)$-modules if and only if we can find such an $s$ with the additional property that $\phi s=\gamma f$ for some $\gamma\in\Hom_{\Lambda^L}(M^L,N^L)$.

More generally we have the following result.
\begin{Lem}\label{Lem:compare-ext}
The forgetful functor $\modcat\Lambda(2)\to\modcat\Lambda$ given by taking only the first terms induces a map
\[ \Ext_{\Lambda^L(2)}^1\big(U\updot,V\updot) \to \Ext_{\Lambda^L}^1(U,V). \]
This sends $\overline\Ext(U\updot,V\updot)$ to $\overline\Ext(U,V)$, the space of extensions which are pull-backs along some $U\to N^L/V$, and has kernel isomorphic to the cokernel of
\[ \Hom_{\Lambda^L}(U,V)\times\Hom_{\Lambda^L}(M^L,N^L) \to \Hom_{\Lambda^L}(U,N^L), \  (\alpha,\beta) \mapsto g\alpha+\beta f. \]
\end{Lem}

\begin{proof}
Consider the exact commutative diagram
\[ \minCDarrowwidth20pt \begin{CD}
\Hom_{\Lambda^L}(M^L,N^L) @>>{\binom01}> \Hom_{\Lambda^L}(U,N^L) \times \Hom_{\Lambda^L}(M^L,N^L)\\
@VVV @VVV\\
\Hom_{\Lambda^L}(U,N^L/V) @= \Hom_{\Lambda^L}(U,N^L/V).
\end{CD} \]
The vertical map on the left has cokernel $\overline\Ext(U\updot,V\updot)$, whereas the map on the right has kernel
\[ \Hom_{\Lambda^L}(U,V) \times \Hom_{\Lambda^L}(M^L,N^L) \]
and cokernel $\overline\Ext(U,V)$. Applying the Snake Lemma we get the long exact sequence
\[ \Hom_{\Lambda^L}(U,V) \times \Hom_{\Lambda^L}(M^L,N^L) \to \Hom_{\Lambda^L}(U,N^L) \to \overline\Ext(U\updot,V\updot) \to \overline\Ext(U,V) \]
as required.
\end{proof}

\subsubsection{Upper semi-continuity}

\begin{Lem}
The function sending $\big((\rho,f),(\sigma,g)\big)\in\rep\inj_\Lambda^{(d,M)}(L)\times\rep\inj_\Lambda^{(e,N)}(L)$ to $\dim_L \overline\Ext(\Ima(f)\updot,\Ima(g)\updot)$ is upper semi-continuous. Similarly for the functions sending $(U,V)\in\Gr_\Lambda\binom Md(L)\times\Gr_\Lambda\binom Ne(L)$ to either
\[ \dim_L\Hom_{\Lambda^L(2)}(U\updot,V\updot) \quad\textrm{or}\quad \dim_L \overline\Ext(U\updot,V\updot) .\]
\end{Lem}

\begin{proof}
Consider the closed subscheme $\overline{\Der}(d,e)\subset\rep\inj_\Lambda^{(e+d,N\oplus M)}$ having $R$-valued points those pairs of the form
\[ \left( \begin{pmatrix}\sigma&\xi\\0&\rho\end{pmatrix}, \begin{pmatrix}g&\phi\\0&f\end{pmatrix} \right). \]
This comes with a projection map
\[ \pi \colon \overline{\Der}(d,e) \to \rep\inj_\Lambda^{(d,M)}\times\rep\inj_\Lambda^{(e,N)} \]
sending such a pair above to $\big((\rho,f),(\sigma,g)\big)$. The fibre of $\pi$ over an $L$-valued point $\big((\rho,f),(\sigma,g)\big)$ is isomorphic to $\overline{\Der}\big((\rho,f),(\sigma,g)\big)$, so the function
\[ \big((\rho,f),(\sigma,g)\big) \mapsto \dim_L \overline{\Der}\big((\rho,f),(\sigma,g)\big) \]
is upper semi-continuous by \hyperref[Cor:usc]{Corollary \ref*{Cor:usc}}. By the analogue of Voigt's Lemma we have
\begin{multline*}
\dim_L \overline\Ext(\Ima(f)\updot,\Ima(g)\updot) = \dim_L \overline{\Der}\big((\rho,f),(\sigma,g)\big) - de\\
- \dim_L\Hom_{\Lambda^L}(M^L,N^L) + \dim_L\Hom_{\Lambda^L(2)}(\Ima(f)\updot,\Ima(g)\updot),
\end{multline*}
so the function $\big((\rho,f),(\sigma,g)\big)\mapsto\dim_L\overline\Ext(\Ima(f)\updot,\Ima(g)\updot)$ is also upper semi-continuous.

For the Grassmannians we know that the sets
\[  \big\{ \big((\rho,f),(\sigma,g)\big) : \dim_L\Hom_{\Lambda^L(2)}\big((\rho,\tau_M^L,f),(\sigma,\tau_N^L,g)\big)\geq t \big\} \]
and
\[ \big\{ \big((\rho,f),(\sigma,g)\big) : \dim_L\overline\Ext(\Ima(f)\updot,\Ima(g)\updot)\geq t \big\} \]
are both closed inside $\rep\inj_\Lambda^{(d,M)}\times\rep\inj_\Lambda^{(d,N)}$. Since they are clearly $\GL_d\times\GL_e$-stable, their images in $\Gr_\Lambda\binom Md\times\Gr_\Lambda\binom Ne$ are also closed. 
\end{proof}

If $K$ is algebraically closed and $X\subset\Gr_\Lambda\binom Md$ and $Y\subset\Gr_\Lambda\binom Ne$ are irreducible, then we write $\overline\ext(X,Y)$ for the generic, or minimal, value of $\dim_L\overline\Ext(U\updot,V\updot)$ for $(U,V)\in(X\times Y)(L)$.

We remark that it is possible to work directly with the Grassmannians; that is, there are schemes $W$ and $Z$, together with morphisms to $\Gr_\Lambda\binom Md\times\Gr_\Lambda\binom Ne$, such that the fibres over an $L$-valued point $(U,V)$ are, respectively, $\Hom_{\Lambda^L(2)}(U\updot,V\updot)$ and $\Hom_{\Lambda^L}(U,N^L/V)$. It follows that the functions sending $(U,V)$ to the dimension of $\Hom_{\Lambda^L(2)}(U\updot,V\updot)$ and $\Hom_{\Lambda^L}(U,N^L/V)$ are both upper semi-continuous on $\Gr_\Lambda\binom Md\times\Gr_\Lambda\binom Ne$, and hence so too is the function sending $(U,V)$ to
\begin{multline*}
 \dim_L\overline\Ext(U\updot,V\updot) = \dim_L\Hom_{\Lambda^L}(U,N^L/V)\\
- \dim_L\Hom_{\Lambda^L}(M^L,N^L) + \dim_L\Hom_{\Lambda^L(2)}(U\updot,V\updot).
\end{multline*}

In fact we even have a commutative square
\[ \begin{CD}
\overline{\Der}(d,e) @>>> \rep\inj_\Lambda^{(d,M)}\times\rep\inj_\Lambda^{(e,N)}\\
@VVV @VVV\\
Z @>>> \Gr_\Lambda\binom Md\times\Gr_\Lambda\binom Ne
\end{CD} \]
where the left-hand morphism is a principal bundle for the natural action of the parabolic subgroup $P_{e,d}\leq\GL_{e+d}$ having zeros in the bottom left $d\times e$ block.

To see this, we observe that the tautological bundle $\mathcal R$ on $\Gr_K\binom Md$ is isomorphic to the associated fibration $\inj_{m\times d}\times^{\GL_d}\bM_{d\times1}$, and so its dual is isomorphic to the associated fibration $\inj_{m\times d}\times^{\GL_d}\bM_{1\times d}$. In particular there is a principal $\GL_d$-bundle $\inj_{m\times d}\times\bM_{d\times 1}\to\mathcal R$. The analogous statement for the quotient bundle $\mathcal Q$ on $\Gr_K\binom Ne$ would involve surjective maps, not injective ones; instead we may consider the semi-direct product $\GL_e\ltimes\bA^e$ (equivalently the subgroup of $\GL_{e+1}$ consisting of those matrices having bottom row $(0,0,\ldots,0,1)$) acting on the right on $\inj_{n\times e}\times\bM_{n\times 1}$ via $(f,y)\cdot(\gamma,x):=(f\gamma,y+f(x))$. Then the morphism $\inj_{n\times e}\times\bM_{n\times 1}\to\mathcal Q$ is a principal $\GL_e\ltimes\bA^e$-bundle. Finally, taking the dual of the tautological bundle on $\Gr_K\binom Md$ and the quotient bundle on $\Gr_K\binom Ne$, we can pull them back to obtain vector bundles on the product $\Gr_K\binom Md\times\Gr_K\binom Md$, and then tensor them together to obtain the vector bundle
\[ \mathcal R^\ast\boxtimes\mathcal Q \to \Gr_K\binom Md\times\Gr_K\binom Ne, \]
whose fibre over an $L$-valued point $(U,V)$ is $\Hom_L(U,N^L/V)$. Combining with the above principal bundles we conclude that there is a principal $P_{e,d}$-bundle
\[ \overline{\Der}_K(d,e) \to \mathcal R^\ast\boxtimes\mathcal Q \]
where $\overline{\Der}_K(d,e)\subset\inj_{(n+m)\times(e+d)}$ consists of those matrices whose lower left $m\times e$ block is zero, so
\[ \overline{\Der}_K(d,e) \cong \big(\inj_{m\times d}\times\inj_{n\times e})\times\bM_{n\times d}. \]
If we now restrict to $\overline{\Der}(d,e)=\overline{\Der}_K(d,e)\cap\rep\inj_\Lambda^{(e+d,N\oplus M)}$, then this is closed and $P_{d,e}$-stable, so yields the required principal $P_{d,e}$-bundle $\overline{\Der}(d,e)\to Z$ by \hyperref[Lem:principal-G-bundle-subscheme]{Lemma \ref*{Lem:principal-G-bundle-subscheme}}.

As for $W$, we know that $\Hom_{\Lambda^L(2)}(U\updot,V\updot)$ is the kernel of the homomorphism $\Hom_{\Lambda^L}(M^L,N^L)\to\Hom_{\Lambda^L}(U,N^L/V)$, so we can define $W$ to be the kernel of the morphism from the trivial vector bundle with fibre $\Hom_{\Lambda^L}(M^L,N^L)$ to $Z$.

\subsubsection{Surjectivity of the differential}

It remains to show that condition (4) of \hyperref[Cor:irred-cpt]{Corollary \ref*{Cor:irred-cpt}} holds for the direct sum morphism
\[ \Theta \colon \!\GL_{d+e}\times\Aut_\Lambda(M\oplus N)\times\rep\inj_\Lambda^{(d,M)}\!\times\!\rep\inj_\Lambda^{(e,N)} \to \rep\inj_\Lambda^{(d+e,M\oplus N)}\!. \]
This follows from the next proposition.

\begin{Prop}\label{Prop:surj-diff3}
The following are equivalent for points $(\rho,f)\in\rep\inj_\Lambda^{(d,M)}(L)$ and $(\sigma,g)\in\rep\inj_\Lambda^{(e,N)}(L)$.
\begin{enumerate}
\item $d_{(\rho,f),(\sigma,g)}^{(r)}\Theta$ is surjective for all $r\in[1,\infty]$.
\item $d_{(\rho,f),(\sigma,g)}^{(\infty)}\Theta$ is surjective for some $r\in[1,\infty]$.
\item $\theta_{(\rho,f),(\sigma,g)}$ is surjective.
\end{enumerate}
In particular, condition (4) of \hyperref[Cor:irred-cpt]{Corollary \ref*{Cor:irred-cpt}} holds true.
\end{Prop}

\begin{proof}
$(1)\Rightarrow(2)\colon$ Trivial.

$(2)\Rightarrow(3)\colon$ Set $U:=\Ima(f)$ and $V:=\Ima(g)$. Suppose we have $(\xi,\phi)\in\overline{\Der}\big((\rho,f),(\sigma,g)\big)$ such that $(M_\xi\subset M^L\oplus N^L)$ is not isomorphic to $U\updot\oplus V\updot$; that is, the image of $\bar g\phi\in\Hom_{\Lambda^L}(U,N^L/V)$ is non-zero in $\overline\Ext(U\updot,V\updot)$, where $\bar g$ is a cokernel for $g$. Then $\big((\rho\oplus\sigma)+\xi t,(f\oplus g)+\phi t\big)$ lies in $T_{(\rho\oplus\sigma,f\oplus g)}^{(\infty)}\rep\inj_\Lambda^{(d+e,M\oplus N)}$ but not in the image of $d_{(\rho,f),(\sigma,g)}^{(\infty)}\Theta$.

$(3)\Rightarrow(1)\colon$ We again use \hyperref[Prop:surj-diff1]{Proposition \ref*{Prop:surj-diff1}} as a guide. Consider an element $(A,\Phi)\in\rep\inj_\Lambda^{(d+e,M\oplus N)}(D_r)$, where $A=\sum_iA_it^i$ and $\Phi=\sum_i\Phi_it_i$, say
\[ A_0 = \begin{pmatrix}\rho&0\\0&\sigma\end{pmatrix}, \quad A_i = \begin{pmatrix}\xi_i&y_i\\x_i&\eta_i\end{pmatrix}, \quad \Phi_0 = \begin{pmatrix}f&0\\0&g\end{pmatrix}, \quad \Phi_i = \begin{pmatrix}\theta_i&\beta_i\\\alpha_i&\phi_i\end{pmatrix}. \]
Let $s\geq1$ be minimal such that not all $x_s,y_s,\alpha_s,\beta_s$ are zero. Regard $(A,\Phi)$ as an element of $\rep\hom_\Lambda^{((s+1)(d+e),M\oplus N)}(L)$ by restriction of scalars. We construct the subrepresentations $U\leq V$ as before, as well as the morphisms
\begin{align*}
(f,\theta_1,\ldots,\theta_{s-1}) &\colon U \to M^L
\intertext{and}
\begin{pmatrix}f&\theta_1&\cdots&\theta_{s-1}&0&\theta_s\\0&0&\cdots&0&g&\alpha_s\end{pmatrix}
&\colon V \to M^L\oplus N^L.
\end{align*}
It follows that
\[ V/U \leftrightarrow \begin{pmatrix} \sigma&x_s\\0&\rho \end{pmatrix} \quad\textrm{and}\quad (g,\alpha_s) \colon V/U \to N^L. \]
We conclude that $(x_s,\alpha_s)\in\overline{\Der}\big((\rho,f),(\sigma,g)\big)$. Now by assumption $\theta_{(\rho,f),(\sigma,g)}$ is surjective, so in particular $\overline\Ext(\Ima(f)\updot,\Ima(g)\updot)=0$. Thus Voigt's Lemma tells us that $(x_s,\alpha_s)=\big(\gamma\rho-\sigma\gamma,\gamma'f-g\gamma\big)$ for some $(\gamma,\gamma')\in\bM_{e\times d}(L)\times\Hom_{\Lambda^L}(M^L,N^L)$. Similarly we can find $(\delta,\delta')$ such that $(y_2,\beta_2)=(\delta\sigma-\rho\delta,\delta'g-f\delta)$. Set $g:=1-\begin{pmatrix}0&\delta\\\gamma&0\end{pmatrix}t^s\in\GL_{d+e}(D_r)$ and $h:=1-\begin{pmatrix}0&\delta'\\\gamma'&0\end{pmatrix}t^s\in\Aut_{\Lambda^{D_r}}(M^{D_r},N^{D_r}$. Then conjugating $(A,\Phi)$ by $(g,h)$ yields another element of $\rep\inj_\Lambda^{(d+e,M\oplus N)}(D_r)$ but now with $x_i,y_i,\alpha_i,\beta_i$ all zero for all $i\leq s$. Now set $g$ to be the product of all the $g_s$ (which again makes sense also for $r=\infty$). Then $g\cdot(A,\Phi)$ has zero off-diagonal entries for the coefficient of each $t^i$, and hence lies in the image of $d^{(r)}\Theta$.
\end{proof}

\subsection{Irreducible Components}

We can now prove the analogues of \hyperref[Thm:CBS1]{Theorems \ref*{Thm:CBS1}} and \hyperref[Thm:CBS2]{\ref*{Thm:CBS2}} for the Grassmannians $\Gr_\Lambda\binom Md$.

\begin{Thm}
Let $K$ be algebraically closed, $\Lambda$ a finitely-generated $K$-algebra, and $M$ and $N$ two $\Lambda$-modules. Let $X\subset\Gr_\Lambda\binom Md$ and $Y\subset\Gr_\Lambda\binom Ne$ be irreducible components. Then the closure $\overline{X\oplus Y}$ of the image of $\Aut_\Lambda(M\oplus N)\times X\times Y\to\Gr_\Lambda\binom{M\oplus N}{d+e}$ is again an irreducible component if and only if $\overline\ext(X,Y)=0=\overline\ext(Y,X)$.
\end{Thm}

\begin{proof}
This follows from \hyperref[Cor:irred-cpt]{Corollary \ref*{Cor:irred-cpt}} applied to the direct sum morphism $\Theta$.
\end{proof}

\begin{Thm}
Every irreducible component $X\subset\Gr_\Lambda\binom Md$ can be written uniquely (up to reordering) as $X=\overline{X_1\oplus\cdots\oplus X_n}$, where $M\cong M_1\oplus\cdots\oplus M_n$ and $d=d_1+\cdots+d_n$, and each $X_i\subset\Gr_\Lambda\binom{M_i}{d_i}$ is an irreducible component such that for all $U_i$ in an open dense subset of $X_i$, the corresponding $\Lambda^L(2)$-module $(U_i\subset M_i^L)$ is indecomposable.
\end{Thm}

\begin{proof}
This follows by applying the arguments of the proof of \hyperref[Thm:CBS2]{Theorem \ref*{Thm:CBS2}} to the locally-closed subscheme $\rep\inj_\Lambda^{(d,M)}\subset\rep_{\Lambda(2)}^{(d,m)}$. In particular, by the Krull-Remak-Schmidt Theorem for $\Lambda(2)$, every module of the form $(U\subset M)$ is isomorphic to a direct sum of indecomposable representations, which are necessarily of the form $(U_i\subset M_i)$ with $U\cong\bigoplus_iU_i$ and $M\cong\bigoplus_iM_i$.
\end{proof}

We again remark that both of these results can be refined to the case where one considers dimension vectors with respect to some complete set of orthogonal idempotents. Note that, since we are taking a $\GL_d$-quotient, we have $\Gr_\Lambda\binom Md=\coprod_{d\updot}\Gr_\Lambda\binom{M}{d\updot}$.

\subsection{Constructing irreducible components}

Given a Grassmannian of submodules $\Gr_\Lambda\binom Md$ one can look at the subscheme $\mathcal S_\rho$ given by those submodules isomorphic to a given module $M_\rho$. This is irreducible, and so it is natural to ask when such a subscheme is dense in an irreducible component of the Grassmannian. (Compare with \cite[\S4]{Reineke} for related work involving certain types of flags of representations of Dynkin quivers arising from desingularisations.)

Fix $\tau\in\rep_\Lambda^m(K)$ such that $M\cong M_\tau$. Given $\rho\in\rep_\Lambda^d(K)$, write $\inj_\Lambda(\rho,\tau)\subset\Hom_\Lambda(\rho,\tau)$ for the open subscheme of injective homomorphisms, so a smooth and irreducible scheme. We observe that this is also the fibre over $\rho$ of the projection $\rep\inj_\Lambda^{(d,M)}\to\rep_\Lambda^d$.

\begin{Thm}
Given $\rho\in\rep_\Lambda^d(K)$, the quotient faisceau
\[ \mathcal S_\rho := \inj_\Lambda(\rho,\tau)/\Aut_\Lambda(M_\rho) \]
is a subscheme of $\Gr_\Lambda\binom Md$. In particular, $\mathcal S_\rho$ is smooth and irreducible, and the morphism $\inj_\Lambda(\rho,\tau)\to\mathcal S_\rho$ is smooth, affine and separable, and a universal geometric quotient.
\end{Thm}

\begin{proof}
We begin by defining $\widetilde{\mathcal S}_\rho$ to be the preimage of $\Orb_{\GL_d}(\rho)$ in $\rep\inj_\Lambda^{(d,M)}$, so a $\GL_d$-stable subscheme. We have a morphism
\[ \phi \colon \GL_d\times\inj_\Lambda(\rho,\tau) \to \widetilde{\mathcal S}_\rho, \quad (g,f) \mapsto g\cdot(\rho,f) = (g\rho g^{-1},fg^{-1}), \]
which is constant on $\Aut_\Lambda(M_\rho)$-orbits and fits into a commutative square
\begin{equation}\label{eq:tilde-S-rho}
\begin{CD}
\GL_d\times\inj_\Lambda(\rho,\tau) @>>> \widetilde{\mathcal S}_\rho\\
@VVV @VVV\\
\GL_d @>>> \Orb_{\GL_d}(\rho).
\end{CD}
\end{equation}
This is a pull-back diagram, since given an $R$-valued point $(g,(\sigma,f'))$ in the fibre product, we have $\sigma=g\cdot\rho^R$, so that $f'g\in\inj_\Lambda(\rho,\tau)(R)$ and hence $(g,(\sigma,f'))$ is the image of $(g,f'g)$. It follows from \hyperref[Lem:assoc-fib]{Lemma \ref*{Lem:assoc-fib}} that $\widetilde{\mathcal S}_\rho$ is isomorphic to the associated fibration $\GL_d\times^{\Aut_\Lambda(M_\rho)}\inj_\Lambda(\rho,\tau)$.

Next consider the principal $\GL_d$-bundle $\pi\colon\rep\inj_\Lambda^{(d,M)}\to\Gr_\Lambda\binom Md$. Since $\widetilde{\mathcal S}_\rho$ is $\GL_d$-stable, we can apply \hyperref[Lem:principal-G-bundle-subscheme]{Lemma \ref*{Lem:principal-G-bundle-subscheme}} to obtain a subscheme $\mathcal S_\rho\subset\Gr_\Lambda\binom Md$ such that $\pi\colon\widetilde{\mathcal S}_\rho\to\mathcal S_\rho$ is again a principal $\GL_d$-bundle.

Putting this together we see that $\GL_d\times\Aut_\Lambda(M_\rho)$ acts freely on $\GL_d\times\inj_\Lambda(\rho,\tau)$ via
\[ (g,a)\cdot(h,f) := (gha^{-1},fa^{-1}) \]
and
\[ \mathcal S_\rho\cong\big(\GL_d\times\inj_\Lambda(\rho \tau)\big)/\big(\GL_d\times\Aut_\Lambda(M_\rho)\big) \cong\inj_\Lambda(\rho,\tau)/\Aut_\Lambda(M_\rho). \]
The other properties now follow from \hyperref[Cor:geom-quotient]{Corollary \ref*{Cor:geom-quotient}}.
\end{proof}

\begin{Cor}
Let $\rho\in\rep_\Lambda^d(K)$. If $D_r=L[[t]]/(t^{r+1})$ for some field $L$ and some $r\in[0,\infty]$, then the morphism $\inj_\Lambda(\rho,\tau)(D_r)\to\mathcal S_\rho(D_r)$ is onto. In particular,
\[ \mathcal S_\rho(L) = \Big\{ U\in\Gr_\Lambda\binom Md(L) : U\cong M_\rho^L \Big\}. \]
\end{Cor}

\begin{proof}
We know that $\GL_d(D_r)\to\Orb_{\GL_d}(\rho)(D_r)$ is onto by \hyperref[Lem:orbit-points]{Lemma \ref*{Lem:orbit-points}}. Thus, using the pull-back diagram \eqref{eq:tilde-S-rho}, the morphism $\GL_d(D_r)\times\inj_\Lambda(\rho,\tau)(D_r)\to\widetilde{\mathcal S}_\rho(D_r)$ is onto. Finally, since $\widetilde{\mathcal S}_\rho\to\mathcal S_\rho$ is a principal $\GL_d$-bundle, it is locally trivial, and hence $\widetilde{\mathcal S}_\rho(R)\to\mathcal S_\rho(R)$ is onto for all local rings $R$.
\end{proof}

\begin{Lem}\label{Lem:S-rho-irred-cpt}
Let $\rho\in\rep_\Lambda^d(K)$ and suppose that there is an embedding $M_\rho\hookrightarrow M$. If the map $\Hom_\Lambda(M_\rho,M)\to\Hom_\Lambda(M_\rho,M/M_\rho)$ is surjective, then $\overline{\mathcal S}_\rho$ is an irreducible component of $\Gr_\Lambda\binom Md$.

In particular, if $\Ext^1_\Lambda(M_\rho,M_\rho)=0$, then $\overline{\mathcal S}_\rho$ is an irreducible component of $\Gr_\Lambda\binom Md$ (and also $\overline{\Orb_{\GL_d}(\rho)}$ is an irreducible component of $\rep_\Lambda^d$).
\end{Lem}

\begin{proof}
We have the morphism $\inj_\Lambda(\rho,\tau)\to\Gr_\Lambda\binom Md$, having image $\mathcal S_\rho$. The result therefore follows immediately from \hyperref[Lem:sep-non-sing]{Lemma \ref*{Lem:sep-non-sing}}.
\end{proof}

Dually we may fix the isomorphism class of the factor module of $M$, obtaining a scheme $\mathcal Q_{\bar\rho}$ for $\bar\rho\in\rep_\Lambda^{m-d}(K)$. This has $L$-valued points for a field $L$ those $U\subset M^L$ such that $M^L/U\cong M_{\bar\rho}^L$. All the results of this section immediately transfer to this situation.

\subsection{Examples}

It is easy to give examples where the Grassmannian of submodules is non-reduced, even generically non-reduced.

Let $\Lambda=K[X]/(X^4)$, so that $\rep_\Lambda^n(R)=\{\rho\in\bM_n(R):\rho^4=0\}$. For $i=1,2,3,4$ we set $\tau_i\in\rep_\Lambda^i(K)$ to be the matrix having ones on the upper-diagonal, and write $S_i$ for the corresponding (indecomposable) module.

\begin{enumerate}
\item We have
\[ \Gr_\Lambda\binom{S_2}{1} \cong \Proj(K[x,y]/(y^2)), \]
so consists of a single non-reduced point.

For, $\rep\inj_\Lambda^{(1,S_2)}(R)$ consists of those triples $(\rho,x,y)\in R^3$ such that $x,y$ generate the unit ideal in $R$, $\rho^4=0$, and
\[ \begin{pmatrix}0&1\\0&0\end{pmatrix}\begin{pmatrix}x\\y\end{pmatrix}=\begin{pmatrix}x\\y\end{pmatrix}\rho. \]
It follows that $y=\rho x$, $y^2=0$ and $x$ is invertible. Thus the closed immersion $\rep\inj_\Lambda^{(1,S_2)}\to\inj_{2\times1}$ has image $X$, where
\[ X(R) := \big\{ (x,y)\in R^2 : x\textrm{ invertible,  } y^2=0 \big\}. \]
Now apply the usual description of
\[ \inj_{2\times1}/\GL_1\cong\bP^1=\Proj\big(K[x,y]\big). \]

\item We also have
\[ \Gr_\Lambda\binom{S_1\oplus S_2}{1} \cong \Proj\big(K[x,y,z]/(xz,z^2)\big), \]
so looks like $\bP^1$ with a single non-reduced point corresponding to the embedding $\binom{0}{1}\colon S_1\hookrightarrow S_1\oplus S_2$.

To see this, we first observe that the closed immersion
\[ \iota \colon \rep\inj_\Lambda^{(1,S_1\oplus S_2)} \to \inj_{3\times1} \]
has image the closed subscheme $X$, where
\[ X(R) = \big\{ (x,y,z)\in R^3 : (x,y,z)^{\mathrm{tr}}\in\inj_{3\times 1}(R) : xz=z^2=0 \big\}. \]

For, $\rep\inj_\Lambda^{(1,S_1\oplus S_3)}(R)$ consists of those quadruples $(\rho,x,y,z)\in R^4$ such that $x,y,z$ generate the unit ideal in $R$, $\rho^4=0$, and
\[ \begin{pmatrix}0&0&0\\0&0&1\\0&0&0\end{pmatrix}\begin{pmatrix}x\\y\\z\end{pmatrix}=\begin{pmatrix}x\\y\\z\end{pmatrix}\rho. \]
The latter condition is equivalent to $x\rho=z\rho=0$ and $z=y\rho$. Since $x,y,z$ generate the unit ideal we deduce from $x\rho=z\rho=y\rho^2=0$ that $\rho^2=0$, and hence that $xz=z^2=0$, so $\iota$ maps to $X$. For the inverse note that $z$ is nilpotent, so $x$ and $y$ generate the unit ideal. If $ax+by=1$, then map $(x,y,z)$ to $(bz,x,y,z)$. This is well-defined, since $\rho=bz$ is the unique element of $R$ satisfying $x\rho=0$ and $y\rho=z$. Now apply the usual description of $\inj_{3\times1}/\GL_1\cong\bP^2=\Proj\big(K[x,y,z]\big)$.

We note that we can write this Grassmannian as $\overline{X_1\oplus X_2}$, where $X_1=\Gr_\Lambda\binom{S_1}{1}$ and $X_2=\Gr_\Lambda\binom{S_2}{0}$, both of which are (reduced) points. Setting $U=M=S_1$, $V=0$ and $N=S_2$, we have submodules $U\subset S_1$ and $V\subset S_2$, and the dimension conditions $\overline\Ext(U\updot,V\updot)=0=\overline\Ext(V\updot,U\updot)$ are satisfied. For, we have
\begin{align*}
\dim\Hom_\Lambda(U,N/V) &= 1, & \dim\Hom_\Lambda(V,M/U) &= 0\\
\dim\Hom_\Lambda(M,N) &= 1, & \dim\Hom_\Lambda(N,M) &= 1\\
\dim\Hom_{\Lambda(2)}(U\updot,V\updot) &= 0, & \dim\Hom_{\Lambda(2)}(V\updot,U\updot) &= 1.
\end{align*}

\item More interestingly we have
\[ \Gr_\Lambda\binom{S_1\oplus S_3}{2} \cong V := \Proj\big(K[x,y,s,t]/(xt-ys,s^3,st,t^3)\big), \]
so looks like $\bP^1$, but generically non-reduced.

To see this, let $\widetilde V\subset\inj_{4\times1}$ be the preimage of $V$, so the closed subscheme having $R$-valued points those quadruples $(x,y,s,t)\in R^4$ such that $x,y,s,t$ generate the unit ideal in $R$, and
\[ xt=ys, \quad s^3=st=t^3=0. \]
We construct a morphism
\[ \widetilde F \colon \rep\inj_\Lambda^{(2,S_1\oplus S_3)} \to \widetilde V \subset\inj_{4\times 1}, \quad (\rho,f) \mapsto (x_{12},x_{23},x_{13},x_{24}), \]
where $x_{ij}=\Delta_{ij}(f)$ is a 2 minor of $f$.

To see that the image of $\widetilde F$ lies in $\widetilde V$ note that if
\[ f = \begin{pmatrix}a&b\\c&d\\\alpha&\beta\\\gamma&\delta\end{pmatrix}, \]
then
\[ f\rho = \begin{pmatrix}0&0&0&0\\0&0&1&0\\0&0&0&1\\0&0&0&0\end{pmatrix}f = \begin{pmatrix}0&0\\\alpha&\beta\\\gamma&\delta\\0&0\end{pmatrix}. \]
Thus
\[ (a,b)\rho=0, \quad (c,d)\rho=(\alpha,\beta), \quad (\alpha,\beta)\rho=(\gamma,\delta), \quad (\gamma,\delta)\rho=0. \]
Now, $x_{13}=a\beta-b\alpha$, so upon substituting $(\alpha,\beta)=(c,d)\rho$ we see that $x_{13}=x_{12}T$ for $T:=\mathrm{Tr}(\rho)$. Similarly
\[ x_{14}=x_{13}T, \quad x_{24}=x_{23}T \quad\textrm{and}\quad x_{34}=x_{24}T. \]
Also, using that $(a,b)\rho=0=(\gamma,\delta)\rho$, we obtain $x_{14}\rho_{ij}=0=x_{34}T$, where $\rho=(\rho_{ij})$. It follows that $s^3=x^2x_{14}T=0$ and similarly $t^3=0$. Finally, we have
\[ st = x_{14}y = -\det\begin{pmatrix}a&b\\\gamma&\delta\end{pmatrix}\begin{pmatrix}d&\beta\\c&\alpha\end{pmatrix} = -\det\begin{pmatrix}ad+bc&a\beta+b\alpha\\\gamma d+\delta c&\gamma\beta+\delta\alpha\end{pmatrix}. \]
Substituting for $(\alpha,\beta)=(c,d)\rho$ and using $(a,b)\rho=0=(\gamma,\delta)\rho$, we get
\begin{align*}
st &= (\rho_{11}-\rho_{22})\det\begin{pmatrix}ad+bc&ad-bc\\\gamma d+\delta c&\gamma d-\delta c\end{pmatrix} = 2(\rho_{11}-\rho_{22})\det\begin{pmatrix}ad&bc\\\gamma d&\delta c\end{pmatrix}\\
&= 2cd(\rho_{11}-\rho_{22})x_{14} = 0.
\end{align*}
Thus the relations for $\widetilde V$ are satisfied.

Next, since $\widetilde F(g\cdot(\rho,f))=\det(g)^{-1}\widetilde F(\rho,f)$ for all $g\in\GL_2(R)$ we see that $\widetilde F$ induces a morphism
\[ F \colon \Gr_\Lambda\binom{S_1\oplus S_3}{2} \to V. \]
We prove that this is an isomorphism by defining the inverse locally. First note that if $(x,y,s,t)\in\widetilde V(R)$, then since $s,t$ are nilpotent we must have that $x,y$ generate the unit ideal. Thus $\widetilde V$ is covered by the distinguished open affines $D(x)$ and $D(y)$. If $x$ is invertible, then we have a morphism
\[ \widetilde G_x \colon D(x) \to \rep\inj_\Lambda^{(2,S_1\oplus S_3)}, \quad (x,y,s,t) \mapsto (\rho,f) \]
where
\[ \rho = \begin{pmatrix}s/x&y\\0&0\end{pmatrix} \quad\textrm{and}\quad
f = \begin{pmatrix}0&-x\\1&0\\s/x&y\\s^2/x^2&t\end{pmatrix}, \]
whereas if $y$ is invertible, then we have a morphism
\[ \widetilde G_y \colon D(y) \to \rep\inj_\Lambda^{(2,S_1\oplus S_3)}, \quad (x,y,s,t) \mapsto (\rho,f) \]
where
\[ \rho = \begin{pmatrix}0&y\\-t^2/y^3&t/y\end{pmatrix} \quad\textrm{and}\quad
f = \begin{pmatrix}s/y&-x\\1&0\\0&y\\-t^2/y^2&t\end{pmatrix}. \]
For $\lambda\in\GL_1(R)=R^\times$ set $g:=\begin{pmatrix}1&0\\0&\lambda\end{pmatrix}\in\GL_2(R)$. Then
\begin{align*}
\widetilde G_x\big(\lambda^{-1}(x,y,s,t)\big) &= g \cdot \widetilde G_x(x,y,s,t)
\intertext{and similarly}
\widetilde G_y\big(\lambda^{-1}(x,y,s,t)\big) &= g \cdot \widetilde G_y(x,y,s,t).
\end{align*}
Also, if both $x$ and $y$ are invertible, then $s^2=t^2=0$ and hence
\[ \widetilde G_y(x,y,s,t) = \begin{pmatrix}1&0\\s/xy&1\end{pmatrix}\cdot \widetilde G_x(x,y,s,t). \]
It follows that the morphisms $\widetilde G_x,\widetilde G_y$ induce morphisms from open subsets of $V$ to $\Gr_\Lambda\binom{S_1\oplus S_3}{2}$, and that they glue to give a morphism $G$ defined on all of $V$.

Finally, the morphisms $F$ and $G$ are mutually inverse. For, it is clear that $\widetilde F\widetilde G_x$ is the identity on $D(x)$, and similarly that $\widetilde F\widetilde G_y$ is the identity on $D(y)$, so $FG=\mathrm{id}$. On the other hand, given $(\rho,f)$, if $x_{12}$ is invertible, then there exists $g\in\GL_2(R)$ such that
\[ g\cdot\rho = \begin{pmatrix}s&y\\0&0\end{pmatrix}, \quad fg^{-1} = \begin{pmatrix}0&-1\\1&0\\s&y\\s^2&ys\end{pmatrix}, \]
so $g\cdot(\rho,f)=\widetilde G_x(1,y,s,ys)$. It follows that $\widetilde G_x\widetilde F(\rho,f)=h\cdot(\rho,f)$, where
\[ h := \begin{pmatrix}1&0\\0&\det(g)^{-1}\end{pmatrix}g. \]
Similarly, if $x_{23}$ is invertible, then we can find $g\in\GL_2(R)$ with $g\cdot(\rho,f)=\widetilde G_y(x,1,xt,t)$, whence  $\widetilde G_y\widetilde F(\rho,f)=h\cdot(\rho,f)$, where $h$ is again defined from $g$ as above. We conclude that $GF=\mathrm{id}$ as well.

We also observe that the Grassmannian cannot be decomposed further. For, over a field $L$, the $\Lambda^L(2)$-representation $\widetilde G_y(x,1,0,0)$ is isomorphic to
\[ \begin{pmatrix}0&-x\\1&0\\0&1\\0&0\end{pmatrix} \colon S_2 \hookrightarrow S_1\oplus S_3, \]
which is indecomposable.

\item Let $\Lambda=KQ_2$ be the algebra of upper-triangular matrices in $\bM_2(K)$ and use the standard matrix idempotents $E_{11}$ and $E_{22}$. Then $\rep_\Lambda^{(d,e)}\cong\bM_{e\times d}$, given by the image of $E_{12}$. There are three indecomposable $\Lambda$-modules up to isomorphism: the two simples $S_1$ and $S_2$ and the indecomposable projective-injective module $T$, of dimension vectors $(1,0)$, $(0,1)$ and $(1,1)$ respectively.

We take the representation $\tau\in\rep_\Lambda^{(2,2)}(K)$ given by the matrix $\begin{pmatrix}1&0\\0&0\end{pmatrix}$, whose corresponding module is $M\cong S_1\oplus S_2\oplus T$, and consider submodules of dimension vector $(1,1)$. Then
\[ \rep\inj_\Lambda^{((1,1),M)}(R) = \{ (\rho,f,f')\in R\times\inj_{2\times 1}(R)^2, \ f'\rho=\tau f \}, \]
the group $\GL_{(1,1)}:=\GL_1\times\GL_1$ acts via
\[ (g,h)\cdot(\rho,f,f') = (h\rho g^{-1},fg^{-1},f'h^{-1}), \]
and
\[ \Gr_\Lambda\binom{M}{(1,1)} := \rep\inj_\Lambda^{(1,1),M}/\GL_{(1,1)}. \]
Writing $f=\binom ab$ and $f'=\binom{a'}{b'}$, we see that $\rep\inj_\Lambda^{((1,1),M)}(R)$ consists of the quintuples $(\rho,a,b,a',b')\in R^5$ such that $a,b$ generate the unit ideal, as do $a',b'$, and also
\[ a = a'\rho, \quad b'\rho=0. \]
Setting
\[ \widetilde V := \Spec\big(K[x,y,z]/(xz)\big) \quad\textrm{and}\quad V := \Proj\big(K[x,y,z]/(xz)\big), \]
we can then prove as above that the morphism
\[ \rep\inj_\Lambda^{((1,1),M)} \to \widetilde V, \quad (\rho,a,b,a',b') \mapsto (aa',ba',bb') \]
induces an isomorphism
\[ \Gr_\Lambda\binom{M}{(1,1)} \xrightarrow\sim V. \]
Thus the Grassmannian is a union of two copies of $\bP^1$ intersecting in a point, with its reduced scheme structure.

The two irreducible components can be decomposed as
\[ \overline{X_1\oplus X_2\oplus X_3} \quad\textrm{and}\quad \overline{Y_1\oplus Y_2\oplus Y_3}, \]
where
\[ X_1 = \Gr_\Lambda\binom{S_1}{(1,0)}, \quad X_2 = \Gr_\Lambda\binom{S_2}{(0,1)}, \quad X_3 = \Gr_\Lambda\binom{T}{(0,0)} \]
and
\[ Y_1 = \Gr_\Lambda\binom{S_1}{(0,0)}, \quad Y_2 = \Gr_\Lambda\binom{S_2}{(0,0)}, \quad Y_3 = \Gr_\Lambda\binom{T}{(1,1)}, \]
all six of which are (reduced) points.
\end{enumerate}

\section{Flags of submodules}\label{Sec:flags}

More generally one can consider flags of $\Lambda$-modules of length $r$; that is, we fix a representation $\tau\in\rep_\Lambda^m(K)$ and a sequence $d\updot=(d^1,\ldots,d^r)$ with $0\leq d^1\leq\cdots\leq d^r\leq m$ and, writing $M=M_\tau$, consider the scheme
\[ \Fl_\Lambda\binom M{d\updot} (R) := \big\{ (U\updot):U^i\in\Gr_\Lambda\binom M{d^i}:U^i\subset U^{i+1} \big\}. \]
\textit{A priori} this is more general than Grassmannians, but by changing the algebra, we can view it as a special case. Recall that $\Lambda(r)$ is the algebra of upper-triangular matrices inside $\bM_r(\Lambda)$, and the elementary matrices $E_{ii}$ form a complete set of orthogonal idempotents such that $E_{ii}\Lambda(r)E_{ii}\cong\Lambda$. A module for $\Lambda(r)$ can be regarded as a sequence $(U\updot,f\updot)$, where $f^i\colon U^i\to U^{i+1}$ is a homomorphism of $\Lambda$-modules. Fixing the dimension vector $d\updot$ is then the same as fixing $\dim U^i=d^i$ for all $i$.

As for Grassmannians, we identify a $\Lambda$-module $M$ with the $\Lambda(r)$-module having $U^i=M$ and $f^i=\mathrm{id}_M$. An element of $\Gr_{\Lambda(r)}\binom{M}{d\updot}(R)$ thus consists of a sequence $U\updot$ such that $U^i\subset M^R$ is a $\Lambda^R$-submodule which is furthermore a direct summand of rank $d^i$ as an $R$-module, and also $U^i\subset U^{i+1}$ for all $i$. It follows that
\[ \Fl_\Lambda\binom M{d\updot} \cong \Gr_{\Lambda(r)}\binom M{d\updot}. \]

We can now apply all the results of the previous section to the schemes $\Fl_\Lambda\binom M{d\updot}$. In particular, the tangent space at a flag $U\updot$ equals
\[ T_{U\updot}\Fl_\Lambda\binom M{d\updot} \cong \Hom_{\Lambda^L(r)}(U\updot,M^L/U\updot), \]
where $M^L/U\updot$ is the cokernel of $U\updot\hookrightarrow M^L$, so is given by the sequence of epimorphisms
\[ M^L/U^1 \twoheadrightarrow M^L/U^2 \twoheadrightarrow \cdots \twoheadrightarrow M^L/U^r. \]

Note that, in the analogue of Voigt's Lemma for Grassmannians, \hyperref[Cor:Voigt3]{Corollary \ref*{Cor:Voigt3}}, we interpreted the differential $d\Theta$ of the direct sum morphism as a morphism in the long exact sequence for $\Hom$ in the category of $\Lambda(2)$-modules. For flags of length $r$ this becomes a morphism in the long exact sequence for $\Hom$ in the category of $\Lambda(r+1)$-modules.

More precisely, given a flag $U\updot$ of length $r$ inside $M^L$, we can extend this to a flag $\widetilde U\updot$ of length $r+1$ by adjoining the inclusion $U^r\hookrightarrow M^L$. Then, given $U\updot\in\Fl_\Lambda\binom M{d\updot}(L)$ and $V\updot\in\Fl_\Lambda\binom N{e\updot}(L)$, we have
\[ \Hom_{\Lambda^L}(M^L,N^L) \cong \Hom_{\Lambda^L(r+1)}(\widetilde U\updot,N^L) \]
and
\[ \Hom_{\Lambda^L(r)}(U\updot,N^L/V\updot) \cong \Hom_{\Lambda^L(r+1)}(\widetilde U\updot,N^L/\widetilde V\updot), \]
and hence the analogue for Voigt's Lemma becomes the following.

\begin{Prop}
Let $U\updot\in\Fl_\Lambda\binom M{d\updot}(L)$ and $V\updot\in\Fl_\Lambda\binom N{e\updot}(L)$. Then the map
\[ \Hom_{\Lambda^L}(M^L,N^L) \to \Hom_{\Lambda^L(r)}(U\updot,N^L/V\updot), \]
can be indentified with the middle map in the exact sequence
\begin{multline*}
0 \to \Hom_{\Lambda^L(r+1)}(\widetilde U\updot,\widetilde V\updot) \to \Hom_{\Lambda^L(r+1)}(\widetilde U\updot,N^L)\\
\to \Hom_{\Lambda^L(r+1)}(\widetilde U\updot,N^L/\widetilde V\updot) \to \Ext^1_{\Lambda^L(r+1)}(\widetilde U\updot,\widetilde V\updot).
\end{multline*}
Moreover, when $U\updot=V\updot$ this map coincides with the differential of the orbit map
\[ \Aut_{\Lambda^L}(M^L)\to\Fl_\Lambda\binom M{d\updot}, \]
and we can identify the above exact sequence with
\[ 0 \to \End_{\Lambda^L(r+1)}(\widetilde U\updot) \to \End_{\Lambda^L}(M^L) \to T_{U\updot}\Fl_\Lambda\binom M{d\updot} \to \Ext^1_{\Lambda^L(r+1)}(\widetilde U\updot,\widetilde U\updot). \]
\end{Prop}

Write $\overline\Ext(\widetilde U\updot,\widetilde V\updot)$ for the image of $\Hom_{\Lambda^L(r)}(U\updot,N^L/V\updot)$ in $\Ext^1_{\Lambda^L(r+1)}(\widetilde U\updot,\widetilde V\updot)$. As for Grassmannians, the function sending $(U\updot,V\updot)$ to $\dim_L\overline\Ext(\widetilde U\updot,\widetilde V\updot)$ is upper semi-continuous. If $K$ is algebraically closed, and $X\subset\Fl_\Lambda\binom M{d\updot}$ and $Y\subset\Fl_\Lambda\binom N{e\updot}$ are irreducible, then we denote by $\overline\ext(X,Y)$ the generic, or minimal, value of $\dim_L\overline\Ext(\widetilde U\updot,\widetilde V\updot)$ for $(U\updot,V\updot)\in X(L)\times Y(L)$.

We then have the following two theorems on irreducible components of flag varieties.

\begin{Thm}
Let $K$ be algebraically closed, $\Lambda$ a finitely-generated $K$-algebra, and $M$ and $N$ two $\Lambda$-modules. Let $X\subset\Fl_\Lambda\binom M{d\updot}$ and $Y\subset\Fl_\Lambda\binom N{e\updot}$ be irreducible components. Then the closure $\overline{X\oplus Y}$ of the image of $\Aut_\Lambda(M\oplus N)\times X\times Y\to\Fl_\Lambda\binom{M\oplus N}{d\updot+e\updot}$ is again an irreducible component if and only if $\overline\ext(X,Y)=0=\overline\ext(Y,X)$.
\end{Thm}

\begin{Thm}
Every irreducible component $X\subset\Fl_\Lambda\binom M{d\updot}$ can be written uniquely (up to reordering) as $X=\overline{X_1\oplus\cdots\oplus X_n}$, where $M\cong M_1\oplus\cdots\oplus M_n$ and $d\updot=d\updot_1+\cdots+d\updot_n$, and each $X_i\subset\Fl_\Lambda\binom{M_i}{d\updot_i}$ is an irreducible component such that for all $U\updot_i$ in an open dense subset of $X_i$, the corresponding $\Lambda^L(r+1)$-module $(U\updot_i\subset M_i^L)$ is indecomposable.
\end{Thm}

\bibliographystyle{amsplain}

\end{document}